\documentclass[11pt]{amsart}
\usepackage[margin=1in]{geometry}
 \usepackage{amsaddr}

\usepackage{appendix}
\usepackage{caption}
\usepackage{graphicx}
\usepackage{amsfonts}
\usepackage{amsmath}
\usepackage{amssymb}
\usepackage{amsthm}

\usepackage{slashbox}
\usepackage{diagbox}

\usepackage{color}
\usepackage[mathscr]{euscript}
\usepackage{float}
\usepackage{multicol}

\usepackage{amsmath}
\usepackage{amssymb}
\usepackage{amsfonts}
\usepackage{amsthm}
\usepackage{mathtools}
\usepackage{graphicx}
\usepackage{cancel}
\usepackage{dsfont}
\usepackage{textcomp}

\newcommand{\usc}{\text{USC}}
\newcommand{\lsc}{\text{LSC}}
\renewcommand{\phi}{\varphi}
\renewcommand{\bar}[1]{\overline{#1}}

\newcommand{\one}{\mathds{1}}

\newcommand{\bali}{\begin{alignat*}{4}}
\newcommand{\eali}{\end{alignat*}}

\newcommand{\eq}    {\begin{equation}}
\newcommand{\eeq}  {\end   {equation}}
\newcommand{\eqy}  {\begin{eqnarray}}
\newcommand{\eeqy}{\end   {eqnarray}}

\newcommand{  \figc}{\begin{figure} \begin{center}}
\newcommand{\efigc}{\end   {center} \end{figure}}

\newcommand{\bmat}{\left( \begin{array}}
\newcommand{\emat}{\end{array} \right)}

\newcommand{\R}{\mathbb{R}}
\newcommand{\N}{\mathbb{N}}
\renewcommand{\S}{\mathbb{S}}
\renewcommand{\P}{\mathbb{P}}
\newcommand{\M}{\mathcal{M}}

\newcommand{\Z}{\mathbb{Z}}

\newcommand{\x}{\mathbf{x}}

\newcommand{\nd}{\noindent}

\numberwithin{equation}{section}
\newtheorem{theorem}{Theorem}[section]

\newtheorem{definition}[theorem]{Definition}
\newtheorem{lemma}[theorem]{Lemma}

\newtheorem{proposition}[theorem]{Proposition}
\theoremstyle{remark}
\newtheorem{remark}[theorem]{Remark}

\newcommand{\nl} {\newline}
\newcommand{\hsp}{\hspace{10pt}}

\newcommand{\bdm}{\begin{displaymath}}
\newcommand{\edm}{\end{displaymath}}

\newcommand{\benum}{\begin{enumerate}}
\newcommand{\eenum}{\end{enumerate}}
\newcommand{\pbegin}{\left\{\begin{array}{lr}}
\newcommand{\pend}{\end{array}\right.}

\newcommand{\mO}{\mathcal{O}}

\newcommand{\mW}{\mathcal{W}}
\newcommand{\mX}{\mathcal{X}}

\newcommand{\ra}{\rightarrow}

\newcommand{\ep}{\epsilon}

\renewcommand{\div}{\mbox{div}}

\newcommand{\bpiece}{\left\{\begin{array}{lr}}
\newcommand{\epiece}{\end{array}\right.}

\usepackage{multirow}
\usepackage{subfig}

\usepackage{color}
\usepackage[hypertexnames=false]{hyperref}
\usepackage{cite}
\hypersetup{colorlinks,breaklinks,
             linkcolor=blue,urlcolor=blue,
             anchorcolor=blue,citecolor=blue}

\newcommand{\eps}{\varepsilon}
\newcommand{\dist}{\text{dist}}
\renewcommand{\L}{\mathscr{L}}
\newcommand{\X}{\mathcal{X}}

\numberwithin{equation}{section}

\begin{document}

\title{Analysis and algorithms for $\ell_p$-based semi-supervised learning on graphs}
\author{Mauricio Flores, Jeff Calder, Gilad Lerman}
\address{Department of Mathematics, University of Minnesota}
\email{mauricio.a.flores.math@gmail.com,jcalder@umn.edu,lerman@umn.edu}
\thanks{{\bf Source Code}: \href{https://github.com/mauriciofloresML/Laplacian_Lp_Graph_SSL.git}{https://github.com/mauriciofloresML/Laplacian\_Lp\_Graph\_SSL.git}\\
\indent{\bf Funding:} The authors gratefully acknowledge National Science Foundation grants 1713691, 1821266, 1830418, and a University of Minnesota Grant in Aid Award.}

\begin{abstract} 
This paper addresses theory and applications of $\ell_p$-based Laplacian regularization in semi-supervised learning. The graph $p$-Laplacian for $p>2$ has been proposed recently as a replacement for the standard ($p=2$) graph Laplacian in semi-supervised learning problems with very few labels, where Laplacian learning is degenerate. 

In the first part of the paper we prove new discrete to continuum convergence results for $p$-Laplace problems on $k$-nearest neighbor ($k$-NN) graphs, which are more commonly used in practice than random geometric graphs. Our analysis shows that, on $k$-NN graphs, the $p$-Laplacian retains information about the data distribution as $p\to \infty$ and Lipschitz learning ($p=\infty$) is sensitive to the data distribution. This situation can be contrasted with random geometric graphs, where the $p$-Laplacian \emph{forgets} the data distribution as $p\to \infty$. We also present a general framework for proving discrete to continuum convergence results in graph-based learning that only requires pointwise consistency and monotonicity.

In the second part of the paper, we develop fast algorithms for solving the variational and game-theoretic $p$-Laplace equations on weighted graphs for $p>2$.  We present several efficient and scalable algorithms for both formulations, and present numerical results on synthetic data indicating their convergence properties. Finally, we conduct extensive numerical experiments on the MNIST, FashionMNIST and EMNIST datasets that illustrate the effectiveness of the $p$-Laplacian formulation for semi-supervised learning with few labels. In particular, we find that Lipschitz learning ($p=\infty$) performs well with very few labels on $k$-NN graphs, which experimentally validates our theoretical findings that Lipschitz learning retains information about the data distribution (the unlabeled data) on $k$-NN graphs.

\end{abstract}

\maketitle

\vspace{-1cm}
\section{Introduction}

Data science problems, such as regression and classification, are pervasive in today's world, and the size of datasets is growing rapidly. In the supervised setting, data needs to be labeled, requiring substantial effort (e.g.~writing a transcript for speech recognition), or may require expert input (deciding whether a brain scan is healthy or not). In contrast, unlabeled data can often be acquired in large quantities with substantially less effort. Semi-supervised learning harnesses the additional information present in unlabeled data to improve learning tasks. This can include geometric or topological properties of unlabeled data, which can provide valuable information about where to place decisions boundaries, for instance. This can be contrasted with fully supervised algorithms, which only make use of labeled data. Fully supervised learning algorithms typically learn parameterized functions and require an abundant amount of labeled data.

A common setting within semi-supervised learning is \emph{graph-based} semi-supervised learning, which is concerned with propagating label information on graphs. Here, we are given an undirected weighted graph $G = (\mX, \mW)$, where $\mX$ are the vertices and $\mW = \{ w_{xy} \}_{x, y \in \mX}$ are nonnegative edge weights, which are chosen so that $w_{xy} \approx 1$ when $x$ is similar to $y$, and $w_{xy} \approx 0$ when $x$ and $y$ are dissimilar. Each vertex $x$ in the \emph{observation set} $\mO \subset \mX$ is assigned a label $g(x)$, where $g: \mO \ra \R^k$. In a classification problem with $k$ classes, the $i^{\rm th}$ class is usually assigned the label vector $g(x)=e_i$, where $e_i$ is the $i^{\rm th}$ standard basis vector in $\R^k$, that is, the vector with all zeros and a one in the $i^{\rm th}$ coordinate (called a ``one-hot'' vector in machine learning). The task of graph-based semi-supervised learning is to extend the labels from the observation set $\mO$ to a label function $u:\mX \to \R^k$ on the whole graph in some meaningful way. In practice, the equations that are solved for propagating labels are separable among the coordinates of $\R^k$, and the problem reduces to solving for $k$ functions $u_i:\mX \to \R$, one for each class $i=1,\dots,k$, and setting $u=(u_1,\dots,u_k)$ (this is referred to as the ``one-vs-rest'' approach in machine learning). Thus, we can, without loss of generality, focus on the scalar case $k=1$ for algorithms and analysis in this paper.

Since the problem of extending labels is \emph{a priori} an ill-posed problem (there are infinitely many solutions), one usually makes the semi-supervised \emph{smoothness assumption}, which asks that the learned labeling function $u:\mX\to \R$ should be smooth in dense regions of the graph \cite{ssl}. The smoothness assumption is often enforced by defining a functional (or regularizer) $J(u)$ that measures the \emph{smoothness} of a labeling $u:\mX \to \R$, and then minimizing $J(u)$ subject to either hard label constraints $u(x)=g(x)$ for $x\in \mO$, or a soft penalty constraint like the mean squared error in the labels. Soft penalties are useful when the labels are corrupted by noise. In this paper, we are concerned with learning problems with very few labels, in the range of one label per class, so a basic assumption we make is that the labels are clean and are not corrupted by noise.  Thus, the hard constraint is natural to impose and nothing is gained by considering a soft constraint. In fact, the soft-constraint is more likely to be ignored, when there are very few labels, unless the penalty parameter is chosen sufficiently large so that the constraint is essentially a hard one. The soft-constrained problem would be interesting to consider in the context of more moderate label rates, where noisy labels can be better tolerated. All of the techniques we discuss in this paper extend directly, with minor modifications, to problems with soft constraints.

One of the most widely used methods in semi-supervised learning is Laplacian regularization \cite{zhu2003semi}, which uses the smoothness functional
\begin{equation}\label{eq:l2energy}
J_2(u):= \frac{1}{4}\sum_{x, y \in \mX} w_{xy} ( u(x) - u(y) )^2.
\end{equation}
Minimizing $J_2$ attempts to force similar data points in dense regions of the graph to have similar labels. Minimizers of $J_2$ are graph harmonic, and solve the graph $2$-Laplace equation $\Delta^G_2 u = 0$, where
\begin{equation}\label{eq:graph2}
\Delta_2^G u(x):= \sum_{y\in \mX}w_{xy}(u(y)-u(x)).
\end{equation}
In classification, the values of $u$ are rounded to the nearest label. Laplacian regularization, and ideas based upon it, are very widely used in machine learning \cite{ando2007learning, zhou2004learning, zhou2005learning, zhou2004ranking,he2004manifold, he2006generalized, wang2013multi, xu2011efficient, yang2013saliency}, and have achieved great successes. However, it has been noted first in \cite{nadler2009semi} and later in \cite{el2016asymptotic}, that Laplacian regularization becomes ill-posed (degenerate) in problems with very few labels. We say a graph-based learning problem is ill-posed in the limit of infinite unlabeled data and finite labeled data if the sequence of learned functions does not continuously attain the labeled (e.g., boundary) data in the continuum limit. In this case, the learned function becomes nearly constant on the whole graph, with sharp spikes near the labeled data. Thus, even with a hard constraint the labels are almost entirely ignored.  See Figure \ref{f.p2}
 for a depiction of this degeneracy. In the continuum, this is merely reflecting the fact that the capacity of a point is zero in dimension $d\geq 2$ \cite{leoni2017first}.
\begin{figure}
\centering
\subfloat[$p=2$]{\label{f.p2}\includegraphics[trim=10 10 10 10, clip=true,width=.24\textwidth]{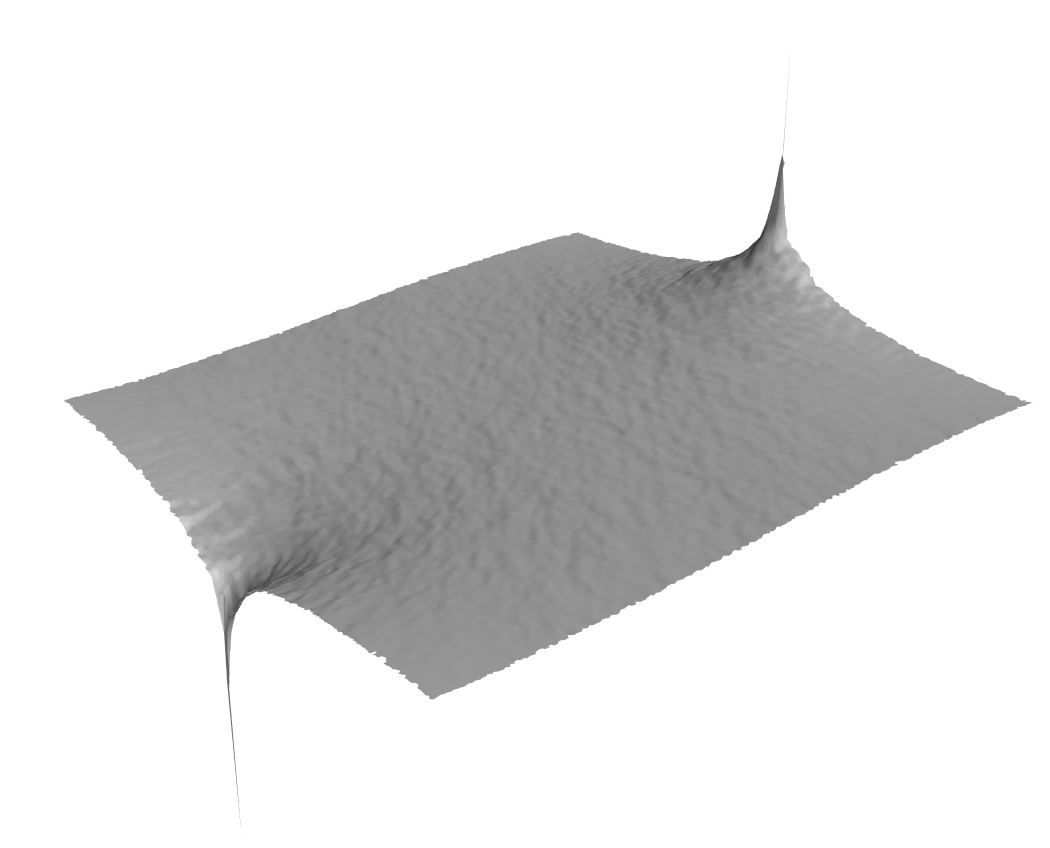}}
\subfloat[$p=2.5$]{\label{f.p25}\includegraphics[trim=10 10 10 10, clip=true,width=.24\textwidth]{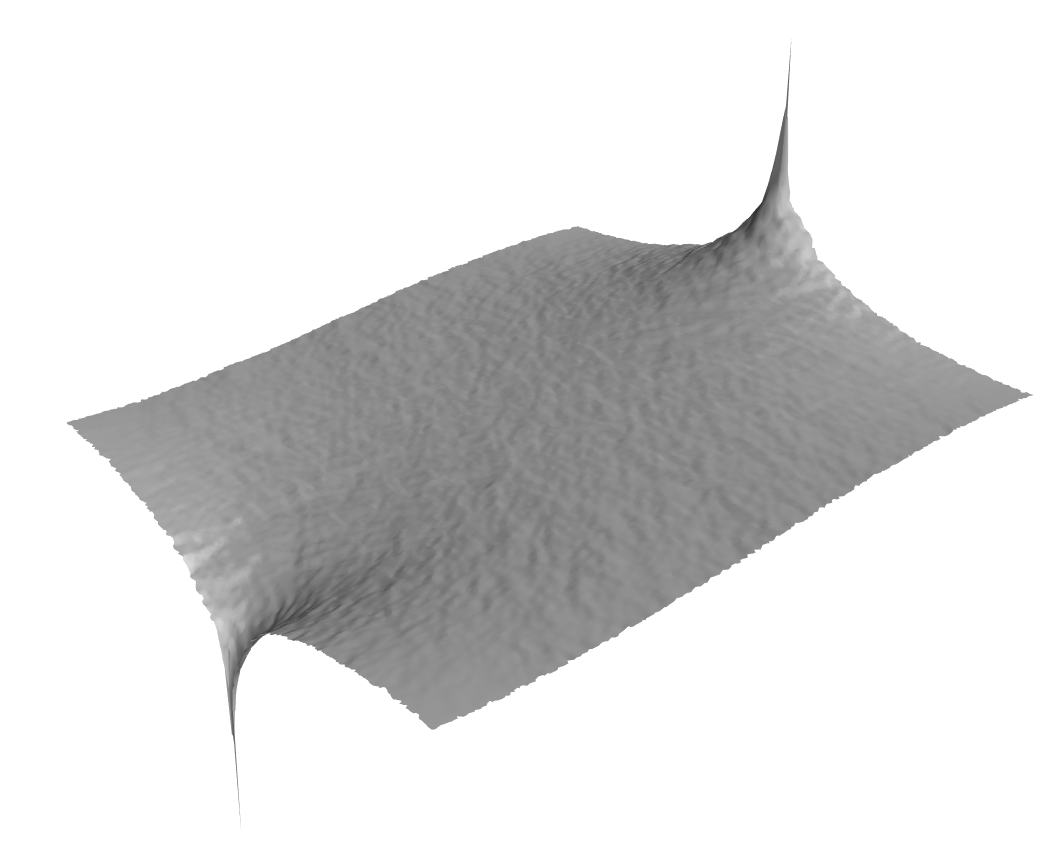}}
\subfloat[$p=3$]{\label{f.p3}\includegraphics[trim=10 10 10 10, clip=true,width=.24\textwidth]{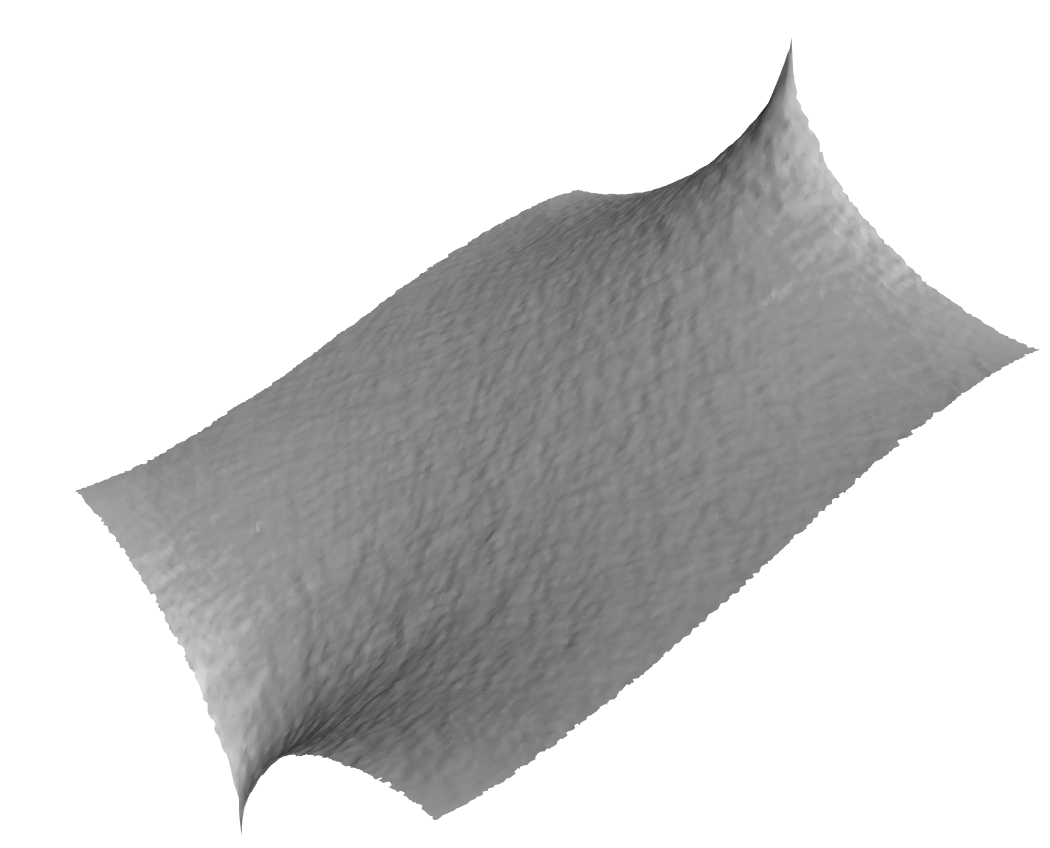}}
\subfloat[$p=10$]{\label{f.p10}\includegraphics[trim=10 10 10 10, clip=true,width=.24\textwidth]{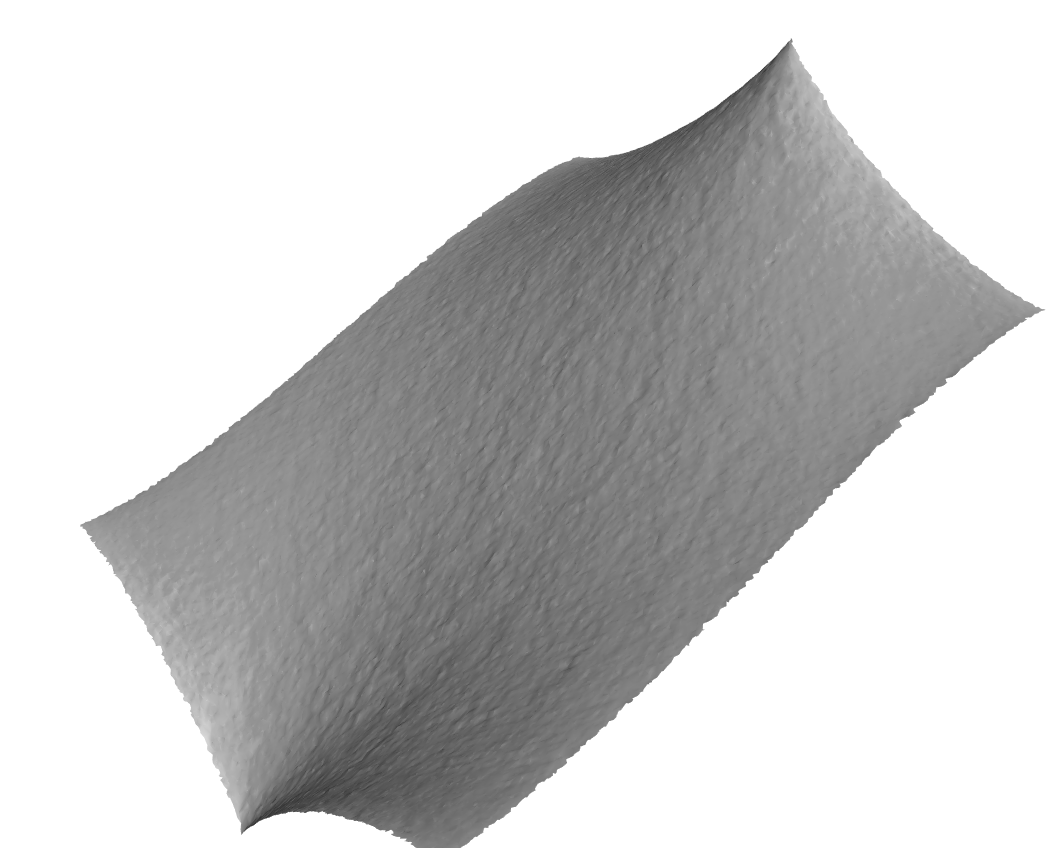}}\\
\caption{Numerical results for a toy learning problem with two labeled and $10^5$ unlabeled data points on $[0,1]^2$. For $p = 2$ the surface is nearly constant, with spikes near the labeled points, while as $p$ becomes larger the surface becomes smoother. The result for $p=\infty$ looks very similar to $p=10$, and the result does not change much for $10 \leq p < \infty$.} 
\label{fig:spikes_ill_posed}
\end{figure}

To address this issue, El Alaoui et al.~\cite{el2016asymptotic} proposed a class of $\ell_p$-based Laplacian regularizers, which use the smoothness functional
\begin{equation}\label{eq:Jp_def_intro}
J_p(u):= \frac{1}{2p}\sum_{x, y \in \mX} w_{xy} \big| u(x) - u(y) \big|^p.
\end{equation}
Choosing $p>2$ places a heavier penalty on large gradients $|u(x)-u(y)|$, which discourages the solution from developing sharp spikes. We also note that choosing $1\leq p<2$, often $p=1$, encourages the gradient $|u(x)-u(y)|$ to be sparse, and can be viewed as a relaxation of graph-cut energies. This can yield good results for classification at moderate label rates (see, e.g., \cite{jung2016semi}), but at very low label rates the issue with spikes is even more pronounced when $p<2$, and the results are similar to Laplace learning.

We note that minimizers of $J_p$ satisfy the graph $p$-Laplace equation $\Delta_p^Gu = 0$, where
\begin{equation}\label{eq:graphp}
\Delta_p^G u(x):= \sum_{y\in \mX}w_{xy}|u(x)-u(y)|^{p-2}(u(y)-u(x)).
\end{equation}
We call (\ref{eq:graphp}) the \emph{variational graph $p$-Laplacian}. Figure \ref{fig:spikes_ill_posed} depicts $\ell_p$-regularization for different values of $p$. As $p$ increases the learned function transitions more smoothly between labeled and unlabeled points. From a continuum perspective, the energy $J_p$ is related to the $p$-Dirichlet energy $\int_\Omega |\nabla u|^p\,dx$, and the Sobolev embedding $W^{1,p}(\Omega)\hookrightarrow C^{0,1-d/p}(\Omega)$ allows isolated boundary points when $p>d$, where $d$ is the dimension.\footnote{Here, $\Omega\subset \R^d$ is an open, bounded domain, and $W^{1,p}(\Omega)$ is the Sobolev space of functions $u:\Omega\to\R$ such that $\int_{\Omega}u^p + |\nabla u|^p\, dx <\infty$.} Indeed, by Morrey's inequality \cite{EvansPDE} we have
\begin{equation}\label{eq:Morrey}
|u(x)-u(y)| \leq C\left( \int_\Omega |\nabla  u|^p\, dx \right)^{1/p}|x-y|^{1-d/p}
\end{equation}
whenever $p>d$ and $|x-y|\leq \tfrac{1}{2}\text{dist}(x,\partial\Omega)$. Morrey's inequality implies that $u$ is H\"older continuous, and prevents spikes  in Figure \ref{f.p2} from occurring. 
In particular, the continuum $p$-Dirichlet problem (see \eqref{eq:contprob}) with constraints at isolated points is well-posed (e.g., admits a unique solution attaining the boundary data continuously) if and only if $p>d$. 

The variational graph $p$-Laplacian \eqref{eq:graphp} has appeared previously in machine learning \cite{alamgir2011phase, bridle2013p,zhou2005regularization}, but was first suggested for problems with few labels in \cite{el2016asymptotic} with $p\geq d+1$. Recently, it was rigorously proven that $\ell_p$-based regularization is ill-posed (its minimizer is degenerate) for $p\leq d$, and well-posed for $p>d$ in the continuum limit of infinite unlabeled and finite labeled data \cite{slepcev2019analysis}. This justifies the continuum heuristics described above.  

Formally sending $p \to \infty$ in \eqref{eq:Jp_def_intro} one obtains \emph{Lipschitz learning}~\cite{kyng2015algorithms, luxburg2004distance}, which corresponds to the smoothness functional
\begin{equation}\label{eq:JInf}
J_\infty(u) = \max_{x,y\in \mX}w_{xy}|u(x)-u(y)|.
\end{equation}
We note that minimizers of \eqref{eq:JInf} are not unique.  To see why this is the case, note that if $w_{xy}|u(x)-u(y)| < J_\infty(u)$, then we can change the values of $u(x)$ and $u(y)$ slightly, without changing the value of $J_\infty(u)$. So minimizers cannot be unique unless the maximum gradient is attained everywhere. Among the non-unique minimizers, one generally looks for one whose gradient cannot be locally improved (i.e., made smaller by adjusting some values of $u(x)$). More precisely, in \cite{kyng2015algorithms} the authors show that there is a unique minimizer whose gradient is smallest in the lexicographical order, called the \emph{lex-minimizer}. This turns out to be equivalent to the notion of \emph{absolutely minimal}, which has been used in the partial differential equation (PDE) and analysis community to select the unique  Lipschitz extension in the continuum  \cite{aronsson2004tour}. 

Lex-minimizers of \eqref{eq:JInf} satisfy the graph $\infty$-Laplace equation $\Delta^G_\infty u= 0$ where
\begin{equation}\label{eq:graphInf}
\Delta_\infty^G u(x):= \min_{y\in \mX}w_{xy}(u(y)-u(x)) + \max_{y\in \mX}w_{xy}(u(y)-u(x)).
\end{equation}
To see why \eqref{eq:graphInf} is the correct form for the graph $\infty$-Laplacian, consider a graph $p$-harmonic function $u$, which satisfies $\Delta^G_p u(x)=0$, where $\Delta^G_p$ is defined in \eqref{eq:graphp}. We split the terms in the sum defining $\Delta^G_p u(x)$ by their sign to obtain
\[\sum_{u(y) > u(x)}w_{xy}(u(y)-u(x))^{p-1} = \sum_{u(x) > u(y)}w_{xy}(u(x)-u(y))^{p-1}.\]
Taking the $p^{\rm th}$ root of both sides and sending $p\to \infty$ yields 
\[\max_{y\sim x}(u(y) - u(x)) = \max_{y\sim x}(u(x)-u(y)),\]
where we write $y\sim x$ if $w_{xy}>0$. This can be simplified to
\[\max_{y\sim x}(u(y) - u(x)) + \min_{y\sim x}(u(y)-u(x)) = 0,\]
which is the graph $\infty$-Laplacian defined in \eqref{eq:graphInf} for an unweighted graph. To obtain the weighted graph $\infty$-Laplacian \eqref{eq:graphInf} in the limit as $p\to \infty$, we simply replace $w_{xy}$ in \eqref{eq:graphp} with $w_{xy}^p$.

It was proven in \cite{calder2019consistency} that Lipschitz learning is well-posed with arbitrarily few labels, and the continuum limit on random geometric graphs is the continuum $\infty$-Laplace equation (see \eqref{eq:infL}). Several papers (see, e.g., \cite{el2016asymptotic,calder2019consistency}) have noted that the continuum $\infty$-Laplacian \eqref{eq:infL} does not involve the data distribution, making it presumably unsuitable for semi-supervised learning, which is supposed to use properties of the unlabeled data, often through its distribution. In \cite{calder2019consistency}, it was shown how to re-weight the graph to introduce varying degrees of sensitivity to the data distribution in Lipschitz learning. 

In order to combine the well-posedness of Lipschitz learning with the distributional sensitivity of Laplacian regularization, it is natural to augment the $2$-Laplacian with a small $\infty$-Laplace term and solve an equation of the form $\Delta^G_2 u + \eps \Delta^G_\infty u = 0$. To this end, we define the \emph{game theoretic} $p$-Laplacian on the graph  (the name will be explained shortly)
\begin{equation} \label{eq:plap_game} 
\L^G_p u(x) = \frac{1}{d_x p} ~ \Delta^G_2 u(x) + \lambda\left(1-\tfrac{2}{p}\right) \Delta^G_\infty u(x), 
\end{equation}
where $d_x = \sum_{y \in \mX} w_{xy}$ is the degree of vertex $x$, and $\lambda>0$ is a constant. For semi-supervised learning with the game-theoretic $p$-Laplacian we solve $\L^G_p u =0$ subject to $u=g$ on $\mO$ (see Section \ref{sec:main} for precise definitions). The second author proved in \cite{calder2018game} that the game theoretic $p$-Laplacian is well-posed with very few labels for $p>d$, and argued for the use of this formulation as an alternative regularization for semi-supervised learning on graphs. In the context of these results, $\lambda$ is chosen (explicitly) depending on the kernel used to define the weights $w_{xy}$, in order to produce the consistency results  described below.  Since having two parameters, $p$ and $\lambda$, is redundant in practice, we take $\lambda=1$ in the numerical sections of the paper. Compared to the variational $p$-Laplacian, the game-theoretic $p$-Laplacian appears better conditioned numerically when $p$ is large, since it does not require computing large powers of $p$. Another main difference is that the game-theoretic $p$-Laplacian does not arise through an optimization problem, and so the methods for solving the equation are somewhat different. We note that the game-theoretic graph $p$-Laplacian (and similar models) have been used very recently for data clustering and learning problems \cite{elmoataz2015p,elmoataz2017game,elmoataz2017nonlocal,hafiene2018nonlocal}, though not in the context of very few labeled data points. A related definition of the game-theoretic $p$-Laplacian on graphs was also studied in \cite{manfredi2015nonlinear}. We also mention recent work \cite{shi2017weighted,calder2019properly} that approaches the semi-supervised learning problem with few labels by re-weighting the graph so that the weights $w_{xy}$ are large near labels.

Both the variational \eqref{eq:graphp} and the game-theoretic \eqref{eq:plap_game} graph $p$-Laplace equations are consistent in the continuum, on random geometric graphs, with the $p$-Laplace equation
\begin{equation}\label{eq:pLaplaceEq}
\Delta_p u:= \text{div}(|\nabla u|^{p-2}\nabla u) = 0.
\end{equation}
It is important to point out that the weights $1/p$ and $\lambda(1-2/p)$ in \eqref{eq:plap_game} are chosen precisely so that $\L^G_p$ is consistent with $\Delta_p$. The operator $\Delta_p$ is called the \emph{$p$-Laplacian}, and solutions of \eqref{eq:pLaplaceEq} are called $p$-harmonic functions \cite{lindqvist2017notes}. The $p$-Laplace equation arises as the necessary conditions (Euler-Lagrange equation) for the $p$-Dirichlet problem
\begin{equation}\label{eq:contprob}
\min_{u}\int_\Omega |\nabla u|^p \, dx.
\end{equation}
Note that we can expand the $p$-Laplacian to obtain
\[\Delta_p u = |\nabla u|^{p-2}(\Delta u + (p-2)\Delta_\infty u),\]
where $\Delta_\infty$ is the $\infty$-Laplacian,  given by
\begin{equation}\label{eq:infL}
\Delta_\infty u:= \frac{1}{|\nabla u|^2}\sum_{i,j=1}^d u_{x_ix_j}u_{x_i}u_{x_j}.
\end{equation}
Thus, any solution of $\Delta_pu=0$ also satisfies
\begin{equation}\label{eq:gameth}
\frac{1}{p}\Delta u + \left(1-\tfrac{2}{p}\right)\Delta_\infty u = 0.
\end{equation}
The left hand side in \eqref{eq:gameth} is often called the \emph{game-theoretic} or \emph{homogeneous} $p$-Laplacian since it arises in two player stochastic tug-of-war games \cite{peres2009tug,lewicka2014game}. This justifies the definition \eqref{eq:plap_game} of the game-theoretic graph $p$-Laplacian. We note that while the $p$-Laplace equation \eqref{eq:pLaplaceEq} is equivalent to the game theoretic version \eqref{eq:gameth} at the continuum level, these are different formulations at the discrete level.

Given the recent interest in graph $p$-Laplacian models in machine learning, it is important to have both strong theoretical results that are relevant in practice, and efficient and scalable algorithms for solving the equations in real-world settings.  Most of the literature on discrete to continuum convergence in graph-based learning, such as the recent work on the $p$-Laplacian \cite{slepcev2019analysis,calder2018game}, assumes the graph is a random geometric graph. However, such graphs have poor sparsity properties\footnote{As an example of the poor sparsity properties of random geometric graphs, consider the MNIST dataset, which has $70,000$ images of handwritten digits. A $k$-nearest neighbor graph on MNIST with $k=3$ neighbors is connected, while the sparsest random geometric graph that is connected requires a bandwidth of $\eps=8.5$ and has on average $294$ neighbors per image, which is roughly $100$ times less sparse than the $k$-nearest neighbor graph. Graph connectivity is one of the basic properties required by most graph-based learning algorithms.}, and practitioners almost always use some form of a  $k$-nearest neighbor ($k$-NN) graph instead. Thus, it is important to develop theory for the graph $p$-Laplacian on $k$-NN graphs in order to study the graphs that are used in practice. It turns out, looking forward to Section \ref{sec:knn}, that the discrete to continuum theory for the $p$-Laplacian on $k$-NN graphs is fundamentally different from the existing theory on random geometric graphs, and this discrepancy accounts for the results of our numerical experiments conducted later in the paper.

From a computational perspective, there are relatively few works on fast algorithms for graph $p$-Laplacians. Kyng et al., \cite{kyng2015algorithms} developed an efficient algorithm for Lipschitz learning ($p=\infty$). Their algorithm has a poor worst case complexity analysis (roughly quadratic complexity in the number of data points), but seems to run very fast in practice. Oberman \cite{oberman2013finite} considers the game-theoretic formulation on regular grids in dimensions $d=2,3$, and developed a fast semi-implicit solution method, as well as gradient-descent methods. One contribution of this paper is an adaptation of Oberman's semi-implicit method to the graph setting. Other works \cite{elmoataz2015p,elmoataz2017game,elmoataz2017nonlocal} use slow iterative methods, such as Jacobi iteration or gradient descent. It was suggested in \cite{el2016asymptotic} to use Newton's method for the variational $p$-Laplacian, but the method was not investigated in any depth. The energy $J_p$ is smooth and convex, but not strongly convex when $p>2$. Other works, such as \cite{kyng2015algorithms}, suggest to use convex programming to solve the variational $p$-Laplacian.

\subsection{Main results and contributions}
\label{sec:main}

Before summarizing the main contributions of our paper, let us give precise definitions of the problems we study. We let $G = (\mX, \mW)$ be a connected graph with vertices $\mX$ and edge weights $\mW = \{ w_{xy} \}_{x, y \in \mX}$. The subset of labeled vertices is denoted $\mO\subset \mX$, and the label function is $g:\mO\to \R$. We denote the number of data points in $\mO$ by $m$, and the number of data points in $\mX$ by $n$.  The $\ell_p$-based Laplacian regularized learning problem \cite{el2016asymptotic} is given by
\begin{equation}\label{eq:plap_optimality}
\min_{u:\mX\to \R}J_p(u) \ \ \text{ subject to }u(x) = g(x) \text{ for all }x\in \mO,
\end{equation}
where we extend the definition of $J_p$ to be
\begin{equation}\label{eq:Jp_def}
J_p(u):= \frac{1}{2p}\sum_{x, y \in \mX} w_{xy} \big| u(x) - u(y) \big|^p + \sum_{x\in \X}f(x)u(x),
\end{equation}
where $f$ is a source function. The unique minimizer $u : \mX \ra \R$ of \eqref{eq:plap_optimality} satisfies the optimality conditions 
\begin{equation}\label{eq:plap_graph}
\left\{\begin{aligned}
-\Delta^G_p u(x) &= f(x)&&\text{if }x \in \mX\setminus \mO\\ 
u(x) &=g(x)&&\text{if }x \in \mO,
\end{aligned}\right.
\end{equation}
where $\Delta^G_p$ is defined in \eqref{eq:graphp}. In machine learning applications we always take $f\equiv 0$ in \eqref{eq:Jp_def} and \eqref{eq:plap_graph}. However, it is useful to formulate the equation in more generality so that we can construct exact solutions by choosing $u(x)$ and computing $f(x)=-\Delta^G_p u(x)$ accordingly. This allows us to evaluate and compare the convergence rates of different methods. We will refer to \eqref{eq:plap_optimality} as the {\bf variational problem}, and to \eqref{eq:plap_graph} as the {\bf variational $p$-Laplace} equation. 

The game-theoretic graph $p$-Laplacian equation  is given by
\begin{equation}\label{eq:plap_graph_game}
\left\{\begin{aligned}
-\L^G_p u(x) &= f(x)&&\text{if }x \in \mX\setminus \mO\\ 
u(x) &=g(x)&&\text{if }x \in \mO,
\end{aligned}\right.
\end{equation}
where $\L^G_p$ is the game-theoretic graph $p$-Laplacian defined in \eqref{eq:plap_game}. We will refer to \eqref{eq:plap_graph_game} as the  {\bf game-theoretic problem}. As before, we are only concerned with $f\equiv 0$ in machine learning, but it is convenient to proceed in generality. The {\bf Lipschitz learning problem} corresponds to the game-theoretic problem with $p=\infty$.

We now summarize the contributions of the paper.

\subsubsection{Discrete to continuum on \texorpdfstring{$k$}-NN graphs}

In Section \ref{sec:knn}, we give a detailed analysis of discrete to continuum convergence of graph-based algorithms on $k$-nearest neighbor graphs.  We prove that the continuum limit of the game-theoretic $p$-Laplacian on \emph{symmetrized} $k$-NN graphs has a significantly different form compared to random geometric graphs. In particular, we show that the continuum operator has an additional drift term along the gradient of the data distribution that \emph{does not} vanish as $p\to \infty$. The drift term arises through the symmetrization process in the $k$-NN graph construction. This means the conventional wisdom that Lipschitz learning is not sensitive to the data distribution (see, e.g., \cite{el2016asymptotic,calder2019consistency,slepcev2019analysis}) is a phenomenon specific to random geometric graphs, and does not hold true for $k$-NN graphs, which highlights the need to perform analysis on different graph constructions. This result is also borne out in our experimental results in Section \ref{sec:experiments}, which show that the  $p=\infty$ graph Laplacian gives good results for semi-supervised learning with very few labels, indicating sensitivity to the data distribution as predicted by our theoretical results.

We also prove a discrete to continuum convergence result for general elliptic learning algorithms on graphs. The result is very general and only requires the equation to have a certain monotonicity property (essentially ellipticity) that allows for the comparison principle to be used. The results show that all elliptic semi-supervised learning algorithms are well-posed at label rates as low as $O(\eps)$, where $\eps>0$ is the bandwidth of the graph (or average distance to the $k^{\rm th}$ nearest neighbor in a $k$-NN graph). This gives a baseline for comparing algorithms at low label rates, and indicates that algorithms should only be claimed to be superior at low label rates if they are well-posed at label rates significantly less than $O(\eps)$. For example, recent work \cite{calder2019rates} by the second author shows that Laplacian regularization ($p=2$) is well-posed at label rates as low as $O(\eps^2)$, and previous work \cite{slepcev2019analysis,calder2018game} showed that the $p$-Laplacian is well-posed for arbitrarily low label rates when $p>d$. It is an open problem to determine the lowest label rates for which $p$-Laplacian regularization is well-posed for $2 < p < d$.

\subsubsection{Efficient algorithms}

We develop and study a range of algorithms for solving both the variational and game-theoretic $p$-Laplace equations, and identify the algorithms that are efficient and scalable in each setting. In each case, we conduct numerical experiments on synthetic data to measure execution time and its dependence on the intrinsic dimensionality of the graph,

For the variational $p$-Laplacian, we propose to use Newton's method with homotopy on $p$. Other solvers, such as iteratively reweighted least squares (IRLS) and a primal-dual method, were considered in \cite{flores2018algorithms}. The primal dual method is slower than Newton with homotopy, and IRLS converges only for $p < 3$, though recent work has found ways to bypass this restriction \cite{adil2019fast}. 

For the game theoretic $p$-Laplacian, we study a gradient descent-like algorithm, a new Newton-like method, and a semi-implicit method, which is an extension of Oberman's method \cite{oberman2013finite} to graphs. The Newton-like method with homotopy on $p$, and the semi-implicit method converge faster than gradient descent, but do not have provable convergence guarantees. The gradient descent-like method is slower, though provably convergent (see Theorem \ref{thm:GDconv}). 

All algorithms are presented and analyzed in Section \ref{sec:alg}, while the numerical experiments on synthetic data are presented in Section \ref{sec:experiments}

\subsubsection{Experimental study}

We conduct a thorough experimental study of $p$-Laplacian semi-supervised learning on real data, including MNIST \cite{lecun1998gradient}, Fashion MNIST \cite{xiao2017fashion}, and Extended MNIST \cite{cohen2017emnist}. In particular, we study classification problems with very few labels, and show that graph $p$-Laplacian learning with $p>2$ is superior to Laplacian regularization. Our results show that $p$-Laplacian learning becomes more accurate when provided with more unlabeled data, which confirms the semi-supervised learning paradigm. The experiments on real data are presented in  Section \ref{sec:real}.  

\subsubsection{Source code}

The code for all numerical experiments is available online\footnote{\href{https://github.com/mauriciofloresML/Laplacian_Lp_Graph_SSL.git}{https://github.com/mauriciofloresML/Laplacian\_Lp\_Graph\_SSL.git}.}, and the $p$-Laplacian semi-supervised learning algorithm is implemented in the GraphLearning Python package \cite{calder2022graphlearning}.

\section{Continuum limits on k-nearest neighbor graphs}
\label{sec:knn}

Since random geometric graphs (also called $\eps$-ball graphs) generally have poor sparsity, it is common in practice to use $k$-nearest neighbor ($k$-NN) graphs, were each point is connected to its $k$-nearest neighbors. However, there are very few discrete to continuum or consistency results for graph Laplacians on $k$-NN graphs. The only results on $k$-NN graphs that we are aware of are pointwise consistency results (without rates)  \cite{ting2011analysis}, $\Gamma$-convergence results for variational problems \cite{garcia2019variational}, and spectral convergence rates \cite{calder2019improved}.

We give here a detailed analysis of pointwise consistency and discrete to continuum convergence for graph $p$-Laplace equations on various $k$-NN and $\eps$-ball graphs (Sections \ref{sec:epsgraph}, \ref{sec:nonknn}, and \ref{sec:symknn}). After a careful study, we find that previous observations about $p$-Laplacian regularization on $\eps$-graphs do not hold on  $k$-NN graphs. In particular, previous work \cite{el2016asymptotic,calder2018game,slepcev2019analysis} has shown that $p$-Laplacian regularization \emph{forgets} the distribution of the unlabeled data as $p\to \infty$, which renders the algorithm unsuitable for semi-supervised learning.  We show in the following sections (see Remarks \ref{rem:pinf_eps}, \ref{rem:pinf_knna}, and \ref{rem:pinf_sym}) that this observation is true \emph{only} for $\eps$-ball graphs, and the situation  is completely different for symmetrized $k$-NN graphs. For symmetrized $k$-NN graphs, even Lipschitz learning ($p=\infty$) is sensitive to the distribution of unlabeled data.

We also give, in Section \ref{sec:continuum}, a general framework for proving discrete to continuum convergence results for elliptic equations on graphs. The main result, Theorem \ref{thm:D2C}, gives a general discrete to continuum convergence result for a wide class of elliptic equations on graphs, and shows that all suitable semi-supervised learning algorithms are well-posed at label rates of $O(\eps)$ and higher. This result indicates that algorithms for semi-supervised learning at low label rates should be judged by their ability to operate below the $O(\eps)$ threshold, and theoretical results should aim to establish this.

Let us remark that the pointwise consistency results in this section share some similarities, in terms of proof techniques, to the results in \cite{calder2018game} and \cite{calder2019improved}. In \cite{calder2018game}, the second author established pointwise consistency, without any rate, for the game theoretic $p$-Laplacian on $\eps$-ball graphs. In Theorem \ref{thm:Ceps} we extend this result to an $O(\eps)$ convergence rate. This is already well-known for $p=2$ (see \cite{hein2007graph}), and since the $p$-Laplacian is a convex combination of the $p=2$ and $p=\infty$ Laplacians, we quote existing results for $p=2$ (actually, we quote the $p=2$ result for $\eps$-ball graphs from \cite{calder2019improved} in the proof of Theorem \ref{thm:Ceps}, since the form of the result from \cite{calder2019improved} is simpler to use). 

For $k$-NN graphs, the only existing pointwise consistency results with rates are the $p=2$ results in the manifold setting established by the second author in \cite{calder2019improved}. The only other work we are aware of on pointwise consistency for $k$-NN graphs is \cite{ting2011analysis}, which considers $p=2$ and proves consistency without any rates. In this section, we prove pointwise consistency with linear convergence rates for the game-theoretic $p$-Laplacian on $k$-NN graphs, both symmetrized (Theorem \ref{thm:knnsympLap}) and non-symmetrized (Theorem \ref{thm:Cknna}). The analogous results for $p=2$ were established by the second author in \cite{calder2019improved}, and these results can be viewed as an extension of those to the $p$-Laplacian with $p>2$. Naturally, we make use of the  $p=2$ results from \cite{calder2019rates}, since the game-theoretic $p$-Laplacian is a convex combination of the $p=2$ and $p=\infty$ graph Laplacians. For both $k$-NN and $\eps$-ball graphs, we prove linear rates of $O(\eps)$ (where $\eps=\left(\frac{k}{n} \right)^{1/d}$ for $k$-NN graphs). For $\eps$-ball graphs and $p=2$, it is well-known that sharper $O(\eps^2)$ pointwise consistency rates are available (see, e.g.,\cite[Remark 5.26]{calderCVnotes}). It would be interesting to investigate whether these sharper rates can be extended to $k$-NN graphs.

For reference, we include here a table of notation used in this section.

\subsection*{Notation}

\begin{itemize} \addtolength{\itemsep}{3pt}
\item[$n$:] Number of vertices in the graph.
\item[$\eps$:] Graph connectivity length scale for $\eps$-ball graphs.
\item[$k$:] Number of neighbors in a $k$-NN graph.
\item[$\eta_\eps$:] Rescaled weight kernel $\eta_\eps(t) = \eta\left( \tfrac{t}{\eps} \right)$.
\item[$\X_n$:] Vertices of our graphs---an \emph{i.i.d.}~sample of size $n$ with density $\rho$ on $\Omega\subset \R^d$.
\item[$\beta$:] Bound on the density $0 < \beta \leq \rho \leq \beta^{-1}$.
\item[$r_0$:] Unique maximum of $r\mapsto r \eta(r)$.
\item[$d^{n,\eps}$:] Degrees on an $\eps$-ball graph.
\item[$N_\eps(x)$:] The number of points in the set $\X_n\cap B(x,\eps)$.
\item[$\eps_k(x)$:] The distance from $x$ to its $k^{\rm th}$ nearest neighbor in $\X_n$.
\item[$s_k(x)$:] Typical distance from $x$ to its $k^{\rm th}$ nearest neighbor (satisfies $\alpha(d)ns_k(x)^d\rho(x) = k$).
\item[$\eps_k(x,y)$:] Symmetrization of $\eps_k(x)$, given by $\eps_k(x,y) = \max\{\eps_k(x),\eps_k(y)\}$.
\item[$s_k(x,y)$:] Symmetrization of $s_k(x)$, given by $s_k(x,y) = \max\{s_k(x),s_k(y)\}$.
\item[$\L^{n,\eps}$:] Unnormalized graph Laplacian on an $\eps$-ball graph.
\item[$\L^{n,\eps}_{rw}$:] Random walk graph Laplacian on an $\eps$-ball graph.
\item[$\L^{n,k}_{a,rw}$:] Random walk graph Laplacian on a nonsymmetric $k$-NN graph.
\item[$\L^{n,k}_{s,rw}$:] Random walk graph Laplacian on a symmetric $k$-NN graph.
\item[$\L^{n,\eps}_p$:] Game-theoretic graph $p$-Laplacian on an $\eps$-ball graph.
\item[$\L^{n,k}_{a,p}$:] Game-theoretic graph $p$-Laplacian on a nonsymmetric $k$-NN graph.
\item[$\L^{n,k}_{s,p}$:] Game-theoretic graph $p$-Laplacian on a symmetric $k$-NN graph.
\end{itemize}

\subsection{Graph construction}

Some parts of the construction of the random graphs are common to $\eps$-ball graphs and $k$-NN graphs, and we review this now. Let $X_1,X_2,\dots,X_n$ be a sequence of \emph{i.i.d.}~random variables on an open connected domain $\Omega\subset \R^d$ with a probability density $\rho:\Omega \to [0,\infty)$. We assume the boundary $\partial \Omega$ is smooth, $\rho\in C^2(\bar{\Omega})$ and there exists $\beta>0$ such that
\begin{equation}\label{eq:rho}
\beta \leq \rho(x) \leq \beta^{-1}  \ \ \ \text{ for all }x\in \Omega.
\end{equation}
Let $\partial_r \Omega= \{x\in \Omega \, : \, \text{dist}(x,\partial \Omega)\leq r\}$ and $\Omega_r = \Omega\setminus \partial_r \Omega$. The vertices of our graph are 
\begin{equation}\label{eq:vertices}
\X_n :=\{X_1,X_2,\dots,X_n\}.
\end{equation}
Let $\eta:[0,\infty)\to [0,\infty)$ be smooth and nonincreasing such that $\eta(t)\geq 1$ for $0\leq t \leq \frac{1}{2}$, and $\eta(t)=0$ for $t>1$. For $\eps>0$ define $\eta_\eps(t) =\eta\left( \tfrac{t}{\eps} \right)$ and set  $\sigma_\eta = \int_{\R^n} |z_1|^2\eta(|z|)\, dz$. Since the graph is unchanged under scaling the weights by a constant, we may as well assume that 
\begin{equation}\label{eq:etaprob}
\int_{B(0,1)} \eta(z) \, dz = 1.
\end{equation}
Then, in particular, $\int_{B(0,\eps)}\eta_\eps(z)\,dz = 1$ as well. As in \cite{calder2018game}, we assume there exists $r_0\in (0,1)$ and $\theta>0$ so that
\begin{equation}\label{eq:strictmax}
r\eta(r) + \theta(r-r_0)^2 \leq r_0 \eta(r_0) \ \ \text{ for all }0 \leq r \leq 1.
\end{equation}
\begin{remark}\label{rem:strictmax}
The condition \eqref{eq:strictmax} says that $r\mapsto r\eta(r)$ has a unique maximum, at $r=r_0$, and is strongly concave near the maximum. Looking forward to the definition of the graph $\infty$-Laplacian (e.g., \eqref{eq:graphLinf}) and Lemma \ref{lem:Linfrate}, the $\max$ and $\min$ over $y\in \X_n$ in \eqref{eq:graphLinf} turn out to occur at a distance of $|x-y|\sim \eps r_0$ from $x$, at least asymptotically as $n\to \infty$ and $\eps\to 0$, and the quantitative assumption \eqref{eq:strictmax} allows us to control this approximation error. If the maximum of $r\mapsto r\eta(r)$ is not unique, then the graph $\infty$-Laplacian may not be asymptotically consistent with the $\infty$-Laplace operator $\Delta_\infty$. 

It is possible to construct nonincreasing $\eta$ where \eqref{eq:strictmax} is not satisfied. For example, consider the kernel
\begin{equation*}
\eta(r) = 
\begin{cases}
2,& \text{if } 0\leq r\leq \tfrac12\\
\frac1r,& \text{if } \tfrac12 \leq r\leq \tfrac34\\
0,& \text{if } r\geq \tfrac34.
\end{cases}
\end{equation*}
The maximum of $r\mapsto r\eta(r)$ is attained on the interval $[\tfrac12,\tfrac34]$ where $r\eta(r)=1$. We can mollify $\eta$ to obtain a smooth kernel wth the same property on the slightly smaller interval $[\tfrac12+\delta,\tfrac34-\delta]$ for any $\delta>0$. We remark that the commonly used Gaussian kernel $\eta(r) = e^{-\frac{r^2}{2\sigma^2}}$ satisfies \eqref{eq:strictmax} (i.e., the maximum of $f(r):=r\eta(r)$ occurs at $r=\sigma$ and $f''(\sigma) = -\frac{2f(r)}{\sigma} < 0$. 
\end{remark}

\begin{remark}
We note it is also common to make the \emph{manifold assumption}, where $X_1,X_2,\dots, X_n$ are a sequence of \emph{i.i.d.}~random variables sampled from an $m$-dimensional manifold $\M$ embedded in $\R^d$, where $m\ll d$. We expect that the analysis here will carry over to the manifold setting with additional technical details. 
\end{remark}

\subsection{\texorpdfstring{$\eps$}{epsilon}-ball graphs}
\label{sec:epsgraph}
The graph constructed with vertices $\X_n$ and edge weights $w_{xy} = \eta_\eps(|x-y|)$ is called a \emph{random geometric graph}, or sometimes an $\eps$-\emph{ball graph}, and is the most widely used graph construction in theoretical analysis of graph-based learning algorithms. Here, we review pointwise consistency for the $p$-Laplacian on $\eps$-ball graphs. Much of the consistency theory was established previously in \cite{calder2018game}, so this section is mostly review, with the additional observation that the arguments in \cite{calder2018game} establish  pointwise consistency \emph{rates} for the $\infty$-Laplacian.

We define the graph Laplacian on a random geometric graph by 
\begin{equation}\label{eq:graphLne}
\L^{n,\eps} u(x) = \frac{2}{\sigma_\eta n\eps^{d+2}}\sum_{y\in \X_n}\eta_\eps(|x-y|) (u(y)-u(x)).
\end{equation}
We also define the degree
\begin{equation}\label{eq:deg}
d^{n,\eps}(x) = \sum_{y\in \X_n} \eta_\eps(|x-y|).
\end{equation}
The \emph{random walk} graph Laplacian is defined by
\begin{equation}\label{eq:graphRW}
\L^{n,\eps}_{rw} u(x) = \frac{2}{\sigma_\eta \eps^2 d^{n,\eps}(x)}\sum_{y\in \X_n}\eta_\eps(|x-y|) (u(y)-u(x)).
\end{equation}
The random walk Laplacian is the generator for a random walk on the graph with transition probabilities $\eta_\eps(|x-y|)/d^{n,\eps}(x)$ of transitioning from $x$ to $y$.  The graph $\infty$-Laplacian is defined by
\begin{equation}\label{eq:graphLinf}
\L_\infty^{n,\eps} u(x) = \frac{1}{r_0^2\eta(r_0)\eps^2}\left(\min_{y\in \X_n}\left\{\eta_\eps(|x-y|) (u(y)-u(x))\right\}+\max_{y\in \X_n}\left\{\eta_\eps(|x-y|) (u(y)-u(x))\right\}\right),
\end{equation}
and the game-theoretic  graph $p$-Laplacian is defined by
\begin{equation}\label{eq:graphLp}
\L^{n,\eps}_p u(x) = \frac{1}{p}\L^{n,\eps}_{rw} u(x)  + \left( 1-\tfrac{2}{p} \right)\L^{n,\eps}_\infty u(x).
\end{equation}
In the continuum, $\L^{n,\eps}_p$ is consistent with the weighted $p$-Laplace operator
\begin{equation}\label{eq:Dps}
\Delta_p u = \frac{1}{p}\rho^{-2}\div(\rho^2 \nabla u)+ \left( 1-\tfrac{2}{p} \right)\Delta_\infty u,
\end{equation}
as is shown in the following theorem.
\begin{theorem}[Consistency on $\eps$-graphs]\label{thm:Ceps}
There exists $C_1,C_2,C_3>0$ such that for any $\eps>0$ with $n\eps^d\geq 1$ and any $0 < \lambda \leq 1$, the event that
\begin{equation}\label{eq:UPCLp}
\max_{x\in \X_n\setminus \partial_\eps \Omega}|\L^{n,\eps}_p u(x) - \Delta_p u(x)| \leq C_1\left(\|u\|_{C^2(B(x,\eps))}^2 |\nabla u(x)|^{-1}\theta^{-1}\eps + \|u\|_{C^3(B(x,\eps))}(\lambda+\eps)  \right) 
\end{equation}
holds for all $u\in C^3(\bar{\Omega})$ has probability at least $1 - C_2n\exp\left( -C_3 n\eps^{3d/2}\lambda^{d/2} \right) - C_2n\exp\left( -C_3 n\eps^{d+2}\lambda^2 \right)$.
\end{theorem}
\begin{proof}
It was shown in \cite[Theorem 5]{calder2018game} that the event that
\begin{equation}\label{eq:PC}
\max_{x\in \X_n\setminus \partial_\eps \Omega}|\L^{n,\eps}_{rw} u(x)  - \rho^{-2}\div\left( \rho^2\nabla u \right)(x)| \leq C\|u\|_{C^3(B(x,\eps))}(\lambda + \eps)
\end{equation}
holds for all $u \in C^3(\bar{\Omega})$ has probability at least $1 - Cn\exp\left( -cn\eps^{d+2}\lambda^2 \right)$ for constants $C,c>0$ and any $0 < \lambda \leq 1$. This is a uniform (over functions $u$) version of pointwise consistency for the graph Laplacian. The weaker nonuniform version dates back to results in \cite{hein2005graphs,hein2007graph}. The proof follows by combining \eqref{eq:PC} with Lemma \ref{lem:Linfrate} below.
\end{proof}
\begin{lemma}\label{lem:Linfrate}
There exists $C_1,C_2,C_3>0$ such that for any $\eps>0$ with $n\eps^d\geq 1$, the event that
\begin{equation}\label{eq:UPCLinf}
\max_{x\in \X_n\setminus \partial_\eps \Omega}|\L^{n,\eps}_\infty u(x) - \Delta_\infty u(x)| \leq C_1\left(\|u\|_{C^2(B(x,\eps))}^2 |\nabla u(x)|^{-1}\theta^{-1}\eps + \|u\|_{C^3(B(x,\eps))}(\lambda+\eps)  \right) 
\end{equation}
holds for all $u\in C^3(\bar{\Omega})$ has probability at least $1 - C_2n\exp\left( -C_3 n\eps^{3d/2}\lambda^{d/2} \right)$.
\end{lemma}
Before proving Lemma \ref{lem:Linfrate}, we make a few remarks.
\begin{remark}
In fact, pointwise consistency of the graph $\infty$-Laplacian (Lemma \ref{lem:Linfrate}) requires $\nabla u(x)\neq 0$, since the $\infty$-Laplacian is discontinuous (as a function of $\nabla u$) at $\nabla u(x)=0$.  We interpret the right hand side of \eqref{eq:UPCLinf} to be $\infty$ if $\nabla u(x)=0$. The viscosity solution framework for proving discrete to continuum convergence does not require consistency at points where $\nabla u(x)=0$, since the viscosity sub-~and supersolution conditions are not required to hold at such points (see, e.g., \cite{calder2018game}).
\label{rem:zerograd}
\end{remark}
\begin{remark}
We note that the continuum operator $\Delta_p$ \emph{forgets} about the distribution $\rho$ as $p\to \infty$. This was first observed in \cite{el2016asymptotic}, and proved rigorously in \cite{calder2018game,slepcev2019analysis}. On the other hand, for semi-supervised learning problems with very few labels, the theory in \cite{el2016asymptotic,calder2018game,slepcev2019analysis} suggests one should use $d < p < p + \delta$ for some small $\delta>0$, so that the algorithm is well-posed with very few labels, and still has maximal dependence on the data distribution $\rho$, which is essential for semi-supervised learning. We show in Section \ref{sec:symknn} that this observation is completely dependent on the $\eps$-ball graph construction, and in particular, the story is very much different on $k$-NN graphs.
\label{rem:pinf_eps}
\end{remark}

We now turn to the proof of Lemma \ref{lem:Linfrate}. By \cite[Theroem 6]{calder2018game}, we know that
\begin{equation}\label{eq:PCLinf}
\lim_{\substack{n\to \infty \\ \eps_n\to 0}}\L_\infty^{n,\eps_n} u(x) = \Delta_\infty u(x)
\end{equation}
holds for all $u \in C^3(\bar{\Omega})$ and $x\in \Omega$ with $\nabla u(x)\neq 0$ with probability one, provided $\eps_n\to 0$ so that 
\begin{equation}\label{eq:epscond}
\lim_{n\to \infty}\frac{n\eps_n^{3d/2}}{\log n}= \infty.
\end{equation}
Lemma \ref{lem:Linfrate} strengthens this to pointwise consistency with a convergence rate. For this, we need some additional notation. We define
\begin{equation}\label{eq:rne}
\delta_n = \sup_{x\in \Omega} \dist(x,\X_n).
\end{equation}

We need a couple of preliminary results before proving Lemma \ref{lem:Linfrate}. The first concerns the approximation of the maximum value of a function over $\Omega$ with the maximum over the point cloud $\X_n$. 
\begin{proposition}\label{prop:c2max}
Let $\psi\in C^2(\bar{\Omega})$ and assume $\psi$ attains its maximum value over $\bar{\Omega}$ at some $x_0\in \Omega_{\delta_n}$. Then we have
\begin{equation}\label{eq:c2maxbound}
\left|\max_{x\in \Omega}\psi(x) - \max_{x\in \X_n}\psi(x)\right| \leq \|\nabla^2 \psi\|_{L^\infty(\Omega)}\delta_n^2.
\end{equation}
\end{proposition}
\begin{proof}
Since $\nabla \psi(x_0)=0$ and  $\psi$ is $C^2$, for any $x\in \X_n$ we have
\[\psi(x_0) - \psi(x) \leq \|\nabla^2 \psi\|_{L^\infty(\Omega)}|x_0-x|^2,\]
and therefore
\[ \|\nabla^2 \psi\|_{L^\infty(\Omega)}\delta_n^2 \geq  \|\nabla^2 \psi\|_{L^\infty(\Omega)}\min_{x\in \X_n}|x_0-x|^2 \geq \psi(x_0) - \max_{x\in \X_n}\psi(x).\qedhere\] 
\end{proof}
The second preliminary result we quote directly from \cite[Proposition 4.2]{calder2019consistency}. 
\begin{proposition}\label{prop:maxprop}
For any $p\in \R^d$ with $p\neq 0$ and $A\in \R^{d\times d}$, we have
\[\left| \max_{|z|=r}\left\{ p\cdot z + \frac{1}{2}z^T Az \right\} - r|p| - \frac{1}{2}r^2|p|^{-2}p^TAp\right | \leq 2r^3 \|A\|^2 |p|^{-1}.\]
\end{proposition}
We note that we prove a generalization of Proposition \ref{prop:maxprop} in Section \ref{sec:symknn} when studying symmetrized $k$-NN graphs (see Proposition \ref{prop:Hcell}).

We now give the proof of Lemma \ref{lem:Linfrate}.
\begin{proof}[Proof of Lemma \ref{lem:Linfrate}]
First, we claim that
\begin{equation}\label{eq:exp1}
r_0^2\eta(r_0)\L_\infty^{n,\eps} u(x) = B^+(x) - B^-(x) + O\left(\|u\|_{C^3(B(x,\eps)}(\delta_n^2\eps^{-3} + \eps)\right),
\end{equation}
where
\[B^\pm(x) = \max_{0 \leq r \leq 1}\left\{ \frac{1}{\eps}r\eta(r)|\nabla u(x)| \pm \frac{1}{2}r^2 \eta(r)\Delta_\infty u(x) \right\}.\]
To see this, consider the $\max$ term in $\L^{n,\eps}_\infty u(x)$, 
\[M := \frac{1}{\eps^2}\min_{y\in \X_n}\left\{\eta_\eps(|x-y|) (u(y)-u(x))\right\}.\]
We use Proposition \ref{prop:c2max} to obtain
\begin{align*}
M &= \frac{1}{\eps^2}\max_{y\in B(x,\eps)}\left\{\eta_\eps(|x-y|) (u(y)-u(x))\right\} + O(\delta_n^2\eps^{-3})\\
&= \frac{1}{\eps^2}\max_{z\in B(0,1)}\left\{\eta(|z|) (u(x+\eps z)-u(x))\right\} + O(\delta_n^2\eps^{-3})\\
&= \frac{1}{\eps^2}\max_{z\in B(0,1)}\left\{\eta(|z|) \left(\eps \nabla u(x)\cdot z + \frac{\eps^2}{2}z^T \nabla^2 u(x)z\right)\right\} + O(\delta_n^2\eps^{-3} + \eps)\\
&= \max_{0 \leq r \leq 1}\left\{\eta(r) \max_{|z|=r}\left\{\frac{1}{\eps}\nabla u(x)\cdot z + \frac{1}{2}z^T \nabla^2 u(x)z\right\}\right\} + O(\delta_n^2\eps^{-3} + \eps).
\end{align*}
Applying Proposition \ref{prop:maxprop} we have
\begin{align*}
M &= \max_{0 \leq r \leq 1}\left\{\frac{1}{\eps}r\eta(r)|\nabla u(x)| + \frac{1}{2}r^2\eta(r)|\nabla u(x)|^{-2} \nabla u(x)^T\nabla^2 u(x)\nabla u(x)\right\} + O(\delta_n^2\eps^{-3} + \eps)\\
&= B^+(x) + O(\delta_n^2\eps^{-3} + \eps),
\end{align*}
since $\Delta_\infty u(x) = |\nabla u(x)|^{-2} \nabla u(x)^T\nabla^2 u(x)\nabla u(x)$. We can apply a similar argument to the $\min$ term from $\L^{n,\eps}_\infty u(x)$, and this establishes the claim \eqref{eq:exp1}.

Now, let $r^\pm\in [0,1]$ such that
\[B^\pm(x) =  \frac{1}{\eps}r^\pm\eta(r^\pm)|\nabla u(x)| \pm \frac{1}{2}(r^\pm)^2 \eta(r^\pm)\Delta_\infty u(x).\]
By \eqref{eq:strictmax} we have
\begin{align*}
\frac{1}{\eps}r_0\eta(r_0)|\nabla u(x)| \pm \frac{1}{2}r_0^2\eta(r_0)\Delta_\infty u(x)&\leq B^\pm(x) \\
&=\frac{1}{\eps}r^\pm\eta(r^\pm)|\nabla u(x)| \pm \frac{1}{2}(r^\pm)^2\eta(r^\pm)\Delta_\infty u(x)\\
&\leq\frac{1}{\eps}|\nabla u(x)|\left(r_0\eta(r_0) -\theta(r^\pm-r_0)^2\right)\pm \frac{1}{2}(r^\pm)^2\eta(r^\pm)\Delta_\infty u(x).
\end{align*}
Therefore
\[\frac{\theta}{\eps}(r^\pm - r_0)^2|\nabla u(x)| \leq \frac{1}{2}\left|r_0^2\eta(r_0) - (r^{\pm})^2\eta(r^\pm) \right| |\Delta_\infty u(x)|.\]
Since $r\mapsto r^2\eta(r)$ is Lipschitz continuous we have
\[\frac{\theta}{\eps}(r^\pm - r_0)^2|\nabla u(x)|  \leq C|\Delta_\infty u(x)| |r^{\pm}-r_0|,\]
and we deduce that
\[|r^\pm - r_0|\leq C\|u\|_{C^2(B(x,\eps))}\theta^{-1}|\nabla u(x)|^{-1}\eps.\]
Since
\[B^-(x) \geq \frac{1}{\eps}r^+\eta(r^+)|\nabla u(x)| - \frac{1}{2}(r^+)^2 \eta(r^+)\Delta_\infty u(x)\]
we have
\[B^+(x) - B^-(x)\leq (r^+)^2 \eta(r^+)\Delta_\infty u(x) \leq r_0^2\eta(r_0)\Delta_\infty u(x) + C\|u\|_{C^2(B(x,\eps))}^2\theta^{-1}|\nabla u(x)|^{-1}\eps.\]
Similarly, since 
\[B^+(x) \geq \frac{1}{\eps}r^-\eta(r^-)|\nabla u(x)| +\frac{1}{2}(r^-)^2 \eta(r^-)\Delta_\infty u(x)\]
we have
\[B^+(x) - B^-(x)\geq (r^-)^2 \eta(r^-)\Delta_\infty u(x) \geq r_0^2\eta(r_0)\Delta_\infty u(x) - C\|u\|_{C^2(B(x,\eps))}^2\theta^{-1}|\nabla u(x)|^{-1}\eps.\]
Combining this with \eqref{eq:exp1} and Proposition \ref{prop:distXn} below completes the proof.
\end{proof}

\begin{proposition}\label{prop:distXn}
There exists $C_1,C_2>0$ such that for every $t>0$ with $nt^d \geq 1$ we have $\P(\delta_n > t) \leq C_1n\exp\left(-C_2 \beta nt^d \right)$.
\end{proposition}
\begin{proof}
Let $t>0$ with $nt^d \geq 1$. Since $\Omega$ has a smooth (and hence Lipschitz) boundary, for any $h>0$ we can find a covering of $\Omega$ by balls $B(x_1,h),B(x_2,h),\dots,B(x_M,h)$  so that $|B(x_i,h)\cap \Omega|\geq C_1h^d$ and $M\leq C_2h^{-d}$ for some universal constants $C_1,C_2>0$. Since $\rho \geq \beta > 0$, the probability that $B(x_i,h)\cap \Omega \cap \X_n$ is empty is bounded by $(1-C_1\beta h^d)^n \leq \exp(-C_1\beta nh^d)$. Hence, the event that $B(x_i,h)\cap \Omega\cap \X_n$ has at least one point for all $i=1,\dots,M$ has probability at least $1 -C_2h^{-d}\exp(-C_1\beta nh^d)$. Since the balls $B(x_i,h)$ cover $\Omega$, for each $x\in \Omega$ there exists $x_i$ such that $|x-x_i|\leq h$. If $B(x_i,h)\cap \Omega \cap \X_n$ is nonempty, then there exists $y\in \X_n$ such that $|x-y|\leq 2h$. Thus, if $\delta_n > 2h$ then at least one of $B(x_i,h)\cap \Omega\cap \X_n$ is empty. Therefore 
\[\P(\delta_{n}> 2h) \leq C_2h^{-d}\exp(-C_1\beta nh^d).\]
The proof is completed by setting $h = t/2$ and recalling $h^{-d}= 2^dt^{-d} \leq 2^dn$.
\end{proof}

\subsection{Nonsymmetric k-nearest neighbor graphs}
\label{sec:nonknn}

We now consider the simplest $k$-nearest neighbor graph. Let
\begin{equation}\label{eq:Ne}
N_\eps(x) = \sum_{y\in \X_n}\one_{B(x,\eps)}(y)
\end{equation}
be the number of samples in an $\eps$-neighborhood of $x\in \Omega$. Here, $B(x,\eps)$ is the closed ball of radius $\eps>0$ centered about $x$. For $k\leq n$ we define 
\begin{equation}\label{eq:epsk}
\eps_k(x) = \min\{\eps > 0 \, :\, N_\eps(x)\geq k\}.
\end{equation}
The value of $\eps_k(x)$ for $x\in \X_n$ is the distance from $x$ to its $k^{\rm th}$ nearest neighbor in $\X_n$. The \emph{nonsymmetric} $k$-nearest neighbor random walk graph Laplacian is then given by 
\begin{equation}\label{eq:graphRWknn}
\L^{n,k}_{a,rw} u(x) = \frac{2}{\sigma_\eta  d^{n,\eps_k(x)}(x)}\left(\frac{n\alpha(d)}{k} \right)^{2/d}\sum_{y\in \X_n}\eta_{\eps_k(x)}(|x-y|) (u(y)-u(x)).
\end{equation}
We note that since $\eps_k(x)$ satisfies, on average, $\alpha(d) n\eps_k(x)^d \rho(x) \approx k$, the normalization in \eqref{eq:graphRWknn} is equivalent, up to the factor $\rho(x)$, with the normalization in the case of the $\eps$-ball graph given in \eqref{eq:graphRW}.

The graph $\infty$-Laplacian is defined by
\begin{align}\label{eq:graphLinfknn}
\L_{a,\infty}^{n,k} u(x) &= \frac{1}{r_0^2\eta(r_0)}\left(\frac{n\alpha(d)}{k} \right)^{2/d}\Bigg(\min_{y\in \X_n}\left\{\eta_{\eps_k(x)}(|x-y|) (u(y)-u(x))\right\}\\
&\hspace{2.5in}+\max_{y\in \X_n}\left\{\eta_{\eps_k(x)}(|x-y|) (u(y)-u(x))\right\}\Bigg).\notag
\end{align}
We now define the game-theoretic $k$-nearest neighbor graph $p$-Laplacian to be 
\begin{equation}\label{eq:graphLpknn}
\L^{n,k}_{a,p} u(x) = \frac{1}{p}\L^{n,k}_{a,rw} u(x)  + \left( 1-\tfrac{2}{p} \right)\L^{n,k}_{a,\infty} u(x).
\end{equation}
The main result in this section is consistency for the nonsymmetric $k$-nearest neighbor graph Laplacian.
\begin{theorem}[Consistency on nonsymmetic $k$-NN graphs]\label{thm:Cknna}
There exists $C_1,C_2,C_3,c_1,c_2>0$ such that for $k\leq c_1n$ and $0 < \lambda \leq 1/4$, the event that
\begin{equation*}
|\L^{n,k}_{a,p} u(x) - \rho(x)^{-2/d}\Delta_p u(x)| \leq C_1\left(\|u\|_{C^2(B_k)}^2 |\nabla u(x)|^{-1}\theta^{-1}\left( \tfrac{k}{n} \right)^{1/d} + \|u\|_{C^3(B_k)}\left(\lambda+\left( \tfrac{k}{n} \right)^{1/d}\right)  \right),
\end{equation*}
holds for all $u\in C^3(\bar{\Omega})$ and $x\in\Omega_{c_2(k/n)^{1/d}}\cap \X_n$, where $B_k=B(x,c_2(k/n)^{1/d})$,  has probability at least 
\[1 - C_2n^3\exp\left( -C_3 \left(\tfrac{k}{n}\right)^{1/2} k\lambda^{d/2} \right) - C_2n^3\exp\left( -C_3 \left(\tfrac{k}{n}\right)^{2/d}k\lambda^2 \right).\] 
\end{theorem}
\begin{remark}
Notice that, up to the factor $\rho^{-2/d}$, the nonsymmetric $k$-NN graph $p$-Laplacian is consistent with the same operator $\Delta_p$ as the $\eps$-ball graph Laplacian. This is due to the ability to treat the $k$-NN Laplacian as an $\eps$-ball graph Laplacian with spatially varying $\eps$, since there is no symmetrization to be concerned with. The same observation as in Remark \ref{rem:pinf_eps} holds here; that is, the continuum operator $\Delta_p$ \emph{forgets} the distribution $\rho$ as $p\to \infty$. We also remark that the rates in Theorems \ref{thm:Cknna} and \ref{thm:Ceps} are essentially the same, since the quantity $\left( \frac{k}{n} \right)^{1/d}$ is the average distance to the $k^{\rm th}$ nearest neighbor (obtained by setting $n\eps^d \sim k$), and plays the same role as $\eps$ in Theorem \ref{thm:Ceps}. 
\label{rem:pinf_knna}
\end{remark}

We now turn to the proof of Theorem \ref{thm:Cknna}. To connect $\L^{n,k}_{a,p}$ to $\L^{n,\eps}_p$, we define
\begin{equation}\label{eq:epsx}
s_k(x) = \left( \frac{k}{\alpha(d)\rho(x)n} \right)^{1/d}.
\end{equation}
The radius $s_k(x)$ satisfies $\alpha(d) ns_k(x)^d\rho(x) = k$, so that the ball $B(x,s_k(x))$ contains on average $k$ points from $\X_n$. Then by algebraic manipulations we have
\begin{equation}\label{eq:nkne}
\L^{n,k}_{a,p} u(x) = \frac{\eps_k(x)^2}{s_k(x)^2\rho(x)^{2/d}}\L^{n,\eps_k(x)}_pu(x).
\end{equation}
This identity allows us to view the nonsymmetric $k$-nearest neighbor graph Laplacian $\L^{n,k}_{a,p}$ as an $\eps$-ball graph Laplacian for spatially varying $\eps=\eps(x)$, and consistency will follow by applying Theorem \ref{thm:Ceps} and a covering argument.

Before proving consistency, we recall some facts about $N_\eps(x)$ and $\eps_k(x)$.  These facts can be found in \cite{calder2019improved}, in the manifold setting, but we include the proofs here for the reader's convenience, as they are simpler in the Euclidean setting.
\begin{proposition}\label{prop:Nex}
For any $0 < t \leq 1$ and $x\in \Omega_\eps$ we have
\begin{equation}\label{eq:Nex}
\P( |N_\eps(x) - np(x,\eps)| \geq np(x,\eps) t) \leq 2 \exp\left(-\tfrac{3}{8}\alpha(d)\beta n\eps^d t^2\right),
\end{equation}
where $p(x,\eps) = \int_{B(x,\eps)}\rho(y)\, dy$.
\end{proposition}
\begin{proof}
Note that $N_\eps(x)$ is the sum of Bernoulli zero/one random variables with parameter $p(x,\eps) = \int_{B(x,\eps)}\rho(x)\, dx$. By the Chernoff bounds we have
\[\P\left( |N_\eps(x) - np(x,\eps)| \geq np(x,\eps) t \right) \leq 2 \exp\left( -\frac{3}{8}np(x,\eps)t^2 \right)\]
for any $0 < t \leq 1$. The proof is completed by noting that $p(x,\eps) \geq  \alpha(d)\beta \eps^d$.
\end{proof}
\begin{lemma}\label{lem:epsk}
There exists $C_1,C_2,c_1,c_2>0$ such that for any $0 < t \leq 1/4$, $k\leq c_1n$, and $x\in \Omega_{c_2(k/n)^{1/d}}$, we have
\begin{equation}\label{eq:knndist}
\P\left(\left|\frac{\eps_k(x)}{s_k(x)} - 1\right| \geq t + C_1\left( \frac{k}{n} \right)^{2/d}\right) \leq 2 \exp\left(-C_2k t^2\right).
\end{equation}
\end{lemma}
\begin{proof}
For any $\eps>0$ we have
\[\P(\eps_k(x) \geq \eps) \leq \P(N_\eps(x) \leq k).\]
Choose $\eps>0$ so that $np(x,\eps)=k(1+t)$ for $t>0$. Then we have
\begin{align*}
\P\left(\eps_k(x) \geq \eps \right) &\leq \P(N_{\eps}(x)\leq k)\\
&=\P\left(N_{\eps}(x) - np(x,\eps)\leq -kt\right)\\
&\leq\P\left(N_{\eps}(x) - np(x,\eps)\leq -\tfrac{1}{2}np(x,\eps)t\right),
\end{align*}
provided $t\leq 1$. By Proposition \ref{prop:Nex} we have
\[\P\left(\eps_k(x) \geq \eps \right) \leq 2 \exp\left(-\tfrac{3}{32}\alpha(d)\beta n\eps^d t^2\right),\]
for all $0 < t \leq 1$. We now note that 
\[2 \geq 1 + t = \frac{n}{k}p(x,\eps) \geq \frac{n}{k}\alpha(d)\beta \eps^d,\]
and
\[1 \leq 1 + t = \frac{n}{k}p(x,\eps) \leq \frac{n}{k}\alpha(d)\beta^{-1} \eps^d.\]
Therefore
\begin{equation}\label{eq:rough}
\left(\frac{\beta}{\alpha(d)}\right)\frac{k}{n}\leq\eps^d \leq \left(\frac{2}{\alpha(d)\beta}\right)\frac{k}{n}.
\end{equation}
For a more refined estimate, we have
\begin{equation}\label{eq:TE}
p(x,\eps) = \int_{B(x,\eps)} \rho(y) \, dy = \alpha(d)\rho(x)\eps^d + O(\eps^{d+2}).
\end{equation}
Therefore
\[k(1+t) = np(x,\eps) \geq \alpha(d)\rho(x)n\eps^d(1-C\eps^2),\]
and so
\[\eps^d \leq \frac{k(1+t)}{\alpha(d)\rho(x)n(1-C\eps^2)} = \frac{s_k(x)^d(1+t)}{1-C\eps^2}\leq s_k(x)^d(1+t)(1+4C\eps^2),\]
provided $C\eps^2 \leq \frac{1}{2}$. Hence, due to \eqref{eq:rough}, there exists $c>0$ such that $k\leq cn$ implies $C\eps^2\leq \frac{1}{2}$ and so 
\[\eps^d \leq s_k(x)^d (1 + t + 8C\eps^2).\]
It follows that there exists $C_1,C_2>0$ such that
\[\P\left(\frac{\eps_k(x)}{s_k(x)} \geq 1 + t + C_1\left( \tfrac{k}{n} \right)^{2/d}\right) \leq 2 \exp\left(-C_2k t^2\right),\]
for any $0 < t \leq 1$. The proof for the estimate in the other direction is similar, but requires a restriction $t \leq 1/4$.
\end{proof}
\begin{proof}[Proof of Theorem \ref{thm:Cknna}]
Fix $0 < a < b\leq1$ with $na^d \geq 1$. Let $C_1,C_2,C_3>0$ so that Theorem \ref{thm:Ceps} holds and let $x\in \Omega_b$. For $0 < \delta \le a/2$ and $0 \leq \lambda \leq 1$, let $A_{\delta,\lambda}$ denote the event that
\[|\L^{n,\eps}_p u(x) - \Delta_p u(x)| \leq C_1\left(\|u\|_{C^2(B(x,\eps))}^2 |\nabla u(x)|^{-1}\theta^{-1}\eps + \|u\|_{C^3(B(x,\eps))}(\lambda+\eps)  \right) \]
holds for all $u\in C^3(\bar{\Omega})$ and all $\eps \in [a-\delta,b+\delta]\cap \Z_\delta$, where $\Z_\delta = \delta \Z$. By Theorem \ref{thm:Ceps} and a union bound we have
\[\P(A_{\delta,\lambda}) \geq  1 - C_2n\delta^{-1}\exp\left( -C_3 na^{3d/2}\lambda^{d/2} \right) - C_2n\delta^{-1}\exp\left( -C_3 na^{d+2}\lambda^2 \right).\]
For $\eps>0$, let us denote by $\lfloor \eps\rfloor_\delta$ the largest number belonging to the set $\Z_\delta:=\delta \Z$ that is less than or equal to $\eps$, and write $\lceil \eps\rceil_\delta = \lfloor \eps\rfloor_\delta + \delta$. Noting that
\[|\partial_\eps \eta_\eps(|x-y|)| = |x-y|\eps^{-2}\eta_\eps(|x-y|) \leq \frac{C}{\eps},\]
provided $|x-y|\leq \eps$, we can compute that
\[|\partial_\eps \L^{n,\eps}_p u(x)| \leq \frac{C\|u\|_{C^1(B(x,\eps)}}{\eps^2}\]
holds for all $u\in C^3(\bar{\Omega})$ and $\eps \in [a-\delta,b+\delta]$ with probability at least $1 - 2\exp\left( -c n a^d \right)$ for $c>0$, due to Proposition \ref{prop:Nex}. Therefore, when $A_{\delta,\lambda}$ occurs and $\eps_k(x) \in [a,b]$ we have
\[|\L^{n,\eps_k(x)}_p u(x) - \Delta_p u(x)| \leq C_1\left(\|u\|_{C^2(B(x,b+\delta))}^2 |\nabla u(x)|^{-1}\theta^{-1}\eps_k(x) + \|u\|_{C^3(B(x,b+\delta))}(\lambda+\eps_k(x))  \right), \]
provided we choose $\delta=a^3$. Using \eqref{eq:nkne} we can write
\[|\rho(x)^{2/d}\L^{n,k}_{a,p} u(x) - \Delta_pu(x)|\leq \frac{\eps_k(x)^2}{s_k(x)^2}|\L^{n,\eps_k(x)}_pu(x) - \Delta_p u(x)| + \left|\frac{\eps_k(x)^2}{s_k(x)^2} - 1 \right| |\Delta_pu(x)|.\]
We now invoke Lemma \ref{lem:epsk}, set $a = c_1(k/n)^{1/d}$ and $b = c_2(k/n)^{1/d}$ for constants $0 < c_1<c_2$, and union bound over $x\in\Omega_{c_2(k/n)^{1/d}}\cap \X_n$. The proof is completed by noting that $\delta^{-1} \leq Cn^{3/d}\leq Cn^2$.
%
%
\end{proof}

\subsection{Symmetric k-nearest neighbor graphs}
\label{sec:symknn}

We now consider symmetrized $k$-nearest neighbor graphs. Define
\begin{equation}\label{eq:rk}
\eps_k(x,y) = \max\{\eps_k(x),\eps_k(y)\}.
\end{equation}
The symmetric $k$-nearest neighbor random walk graph Laplacian is then defined by 
\begin{equation}\label{eq:graphRWsknn}
\L^{n,k}_{s,rw} u(x) = \frac{2}{\sigma_\eta  d^{n,k}_s(x)}\left(\frac{n\alpha(d)}{k} \right)^{2/d}\sum_{y\in \X_n}\eta_{\eps_k(x,y)}(|x-y|) (u(y)-u(x)),
\end{equation}
where the degree $d^{n,k}_s(x)$ is defined by
\begin{equation}\label{eq:dnks}
d^{n,k}_s(x) = \sum_{y\in \X_n}\eta_{\eps_k(x,y)}(|x-y|).
\end{equation}
We can also define a symmetrized graph using $\eps_k(x,y)=\min\{\eps_k(x),\eps_k(y)\}$, which would produce the \emph{mutual} $k$-NN graph, or any other suitable symmetric combination of $\eps_k(x)$ and $\eps_k(y)$. We expect the consistency results in this section to hold, with minor modifications, for many other types of symmetrization.

The graph $\infty$-Laplacian on the symmetric $k$-nearest neighbor graph  is defined by
\begin{align}\label{eq:graphLinfknn_sym}
\L_{s,\infty}^{n,k} u(x) &= \frac{1}{r_0^2\eta(r_0)}\left(\frac{n\alpha(d)}{k} \right)^{2/d}\Bigg(\min_{y\in \X_n}\left\{\eta_{\eps_k(x,y)}(|x-y|) (u(y)-u(x))\right\}\\
&\hspace{2.5in}+\max_{y\in \X_n}\left\{\eta_{\eps_k(x,y)}(|x-y|) (u(y)-u(x))\right\}\Bigg),\notag
\end{align}
and the graph $p$-Laplacian is defined by
\begin{equation}\label{eq:graphLpknnsym}
\L^{n,k}_{s,p} u(x) = \frac{1}{p}\L^{n,k}_{s,rw} u(x)  + \left( 1-\tfrac{2}{p} \right)\L^{n,k}_{s,\infty} u(x).
\end{equation}
The graph $p$-Laplacian on a symmetric $k$-NN graph is consistent in the continuum with the operator
\begin{equation}\label{eq:knnsymDel}
\Delta_p^s u = \rho^{-2/d}\left[\tfrac{1}{p}\rho^{-(1+2/d)}\div\left( \rho^{1-2/d}\nabla u \right) +(1-\tfrac{2}{p})\left(\Delta_\infty u-\tfrac{1}{d}\nabla \log \rho\cdot \nabla u\right)\right],
\end{equation}
as is shown in the following theorem.
\begin{theorem}[Consistency on symmetric $k$-NN graphs]\label{thm:knnsympLap}
There exists $C_1,C_2,C_3,c_1,c_2>0$ such that for $0  < t \leq 1/4$ and $k\leq c_1n$, the event that
\[\left|\L^{n,k}_{s,p}u(x)-\Delta_p^s u(x) \right|\leq C_1(1+|\nabla u(x)|^{-1}\|u\|^2_{C^2(B(x,\eps))} + \|u\|_{C^3(B(x,\eps))})(\delta_n\eps^{-2} + t+\eps)\]
holds for all $x\in \X_n\cap \Omega_\eps$, where $\eps= c_2(k/n)^{1/d}$, and for all $u\in C^3(\bar{\Omega})$, has probability at least $1-C_3n\exp\left( -C_2k\left( \tfrac{k}{n} \right)^{2/d}t^2 \right)$.
\end{theorem}
\begin{proof}
The proof simply combines Lemma \ref{lem:PCrwsknn} and Lemma \ref{lem:knnLinf}, both of which are proved below.
\end{proof}

\begin{remark}
It is important to point out the additional drift term $-\frac{1}{d}\nabla \log \rho\cdot \nabla u$ in $\Delta_p^s$, compared to $\Delta_p$. This term is due to the symmetrization of the $k$-nearest neighbor relation, and is a result of the Taylor expansion
\[\left( \frac{\rho(y)}{\rho(x)} \right)^{\frac{1}{d}} = 1 + \frac{1}{d}\nabla \log \rho(x)\cdot(y-x) + O(|x-y|^2),\]
which makes an appearance in the proof of Lemma \ref{lem:skewball} below. Since this term persists in the limit as $p\to \infty$, this shows that $p$-Laplacian learning for large $p$, and in particular, Lipschitz learning ($p=\infty$) are sensitive to the distribution  $\rho$  of the unlabeled data \emph{provided} the graph is constructed as a symmetrized $k$-nearest neighbor graph, as is often done in practice. This should be contrasted with the case of $\eps$-ball graphs and nonsymmetric $k$-NN graphs, where the continuum operator $\Delta_p$ \emph{forgets} the distribution $\rho$ as $p\to \infty$.
\label{rem:pinf_sym}
\end{remark}

We first record a result about $\eps_k(x,y)$, which was proved in \cite{calder2019improved}, but not explicitly stated as a lemma. We define $a_-=\min\{a,0\}$.
\begin{lemma}\label{lem:skewball}
There exists $C_1,C_2,c_1,c_2>0$ such that for any $0 < t \leq 1/4$,  $k\leq c_1n$ and $x\in \Omega_\eps$, the event that
\[\left| \frac{\eps_k(x,y)}{s_k(x)}\left(1 + \tfrac{1}{d}[\nabla \log \rho(x)\cdot (y-x)]_-\right)- 1\right| \leq t + C_1\eps^2,\]
holds for all $y\in \X_n\cap B(x,\eps)$, with $\eps=c_2(k/n)^{1/d}$, has probability at least $1-2n\exp\left( -C_2kt^2 \right)$.
\end{lemma}
\begin{proof}
Let $x\in \Omega_\eps$. By Lemma \ref{lem:epsk} and a union bound, there exists $C_1,C_2,c>0$ such that for any $0 < t \leq 1/4$ and $k\leq cn$ we have
\begin{equation}\label{eq:epsksk}
\max_{y\in \{x\}\cup\X_n\cap\Omega_\eps}\left|\frac{\eps_k(y)}{s_k(y)} - 1\right| \leq t + C_1\eps^2
\end{equation}
holds with probability at least $1-2n\exp\left(-C_2k t^2\right)$, where $\eps=c_2(k/n)^{1/d}$. We assume for the rest of the proof that \eqref{eq:epsksk} holds. For $y\in \X_n\cap B(x,\eps)$ it follows from \eqref{eq:epsksk} that
\[\left| \frac{\eps_k(x,y)}{s_k(x,y)} - 1\right| \leq t + C_1\eps^2,\]
where $s_k(x,y) = \max\{s_k(x),s_k(y)\}$. Since $|x-y|\leq \eps$, we compute
\begin{align*}
s_k(x,y)&=s_k(x)\max\left\{ 1,\frac{s_k(y)}{s_k(x)} \right\}\\
&=\frac{s_k(x)}{\min\left\{ 1,\frac{\rho(y)^{1/d}}{\rho(x)^{1/d}} \right\}}\\
&=\frac{s_k(x)}{\min\left\{ 1,1 + \frac{1}{d}\nabla \log \rho(x)\cdot (y-x) + O(\eps^2) \right\}}\\
&=\frac{s_k(x)}{1 + \frac{1}{d}[\nabla \log \rho(x)\cdot (y-x)]_- + O(\eps^2)}.
\end{align*}
Therefore
\[\left| \frac{\eps_k(x,y)}{s_k(x)}\left(1 + \tfrac{1}{d}[\nabla \log \rho(x)\cdot (y-x)]_-\right)- 1\right| \leq t + C_1\eps^2,\]
which completes the proof.
\end{proof}

\begin{lemma}\label{lem:dnks}
There exists $C_1,C_2,C_3,c>0$ such that for each $x\in \Omega_\eps$ with $\eps = c(k/n)^{1/d}$
\begin{equation}\label{eq:dnks_bound}
\left|\frac{\alpha(d)}{k} d^{n,k}_s(x) - 1\right| \leq C_1(t + \eps)
\end{equation}
holds with probability at least $1-C_2\exp\left( -C_3kt^2 \right)$ for any $0 < t \leq 1/4$.
\end{lemma}
\begin{proof}
By Lemma \ref{lem:skewball} there exists $c>0$ so that for every $x\in \Omega_\eps$, where $\eps=c(k/n)^{1/d}$, we have
\begin{equation}\label{eq:ekxyexp}
\left|\frac{s_k(x)}{\eps_k(x,y)} - 1 \right| \leq C(t + \eps)
\end{equation}
for all $y\in \X_n\cap B(x,\eps)$ with probability at least $1 - 2n\exp\left( -C_2kt^2 \right)$. Notice the drift term involving $\nabla \log \rho(x) \cdot (y-x)_-$ form Lemma \ref{lem:skewball} is absorbed into the $O(\eps)$ error term on the right hand side, since $|x-y|\leq \eps$. Assuming \eqref{eq:ekxyexp} holds we have
\begin{align*}
d^{n,k}_s(x) &= \sum_{y\in \X_n}\eta\left( \frac{|x-y|}{\eps_k(x,y)} \right)\\
&=\sum_{y\in \X_n}\eta\left(\frac{|x-y|}{s_k(x)}(1 + O(t + \eps))\right) \\
&=\sum_{y\in \X_n}\eta\left(\frac{|x-y|}{s_k(x)}\right) + O(N_{Cs_k(x)}(x)(t+\eps)),
\end{align*}
for $C>0$, provided $t\leq 1$ and $\eps\leq 1$.  By Proposition \ref{prop:Nex} we have $N_{Cs_k(x)}(x)\leq \bar{C}k$ with probability at least $1-\exp\left( -ck \right)$.  Thus,
\[d^{n,k}_s(x) = d^{n,s_k(x)}(x) + O(k(t + \eps)).\]
By the Bernstein inequality (see, e.g., \cite[Theorem 5]{calder2018game}) we have that
\[d^{n,s_k(x)}(x) = \rho(x) ns_k(x)^d + O(ns_k(x)^{d}t) = \alpha(d)^{-1}k + O\left( kt \right),\]
with probability at least $1-2\exp\left( -ck t^2 \right)$. This completes the proof.
\end{proof}

We now prove pointwise consistency for the random walk graph Laplacian $\L^{n,k}_{s,rw}$ on a symmetric $k$-nearest neighbor graph. The analogous result for the unnormalized graph Laplacian was established in \cite{calder2019improved}.
\begin{lemma}\label{lem:PCrwsknn}
There exists $C_1,C_2,C_3,c_1,c_2>0$ such that for $0  < t \leq 1/4$ and $k\leq c_1n$, the event that
\begin{equation*}
\left|\L^{n,k}_{s,rw} u(x) - \rho(x)^{-1}\div\left( \rho^{1 - 2/d}\nabla u \right)(x) \right| \leq C_1(1 +\|u\|_{C^3(B_k)})\left(t+\left( \tfrac{k}{n} \right)^{1/d}\right) 
\end{equation*}
holds for all $u\in C^3(\bar{\Omega})$ and $x\in\Omega_{c_2(k/n)^{1/d}}\cap \X_n$, where $B_k=B(x,c_2(k/n)^{1/d})$,  has probability at least $1 -C_2n\exp\left( -C_3 \left(\tfrac{k}{n}\right)^{2/d}kt^2 \right)$.
\end{lemma}

\begin{proof}
Let $u\in C^3(\bar{\Omega})$ and $x \in \Omega$. Let us define the unnormalized $k$-nearest neighbor graph Laplacian
\[Lu(x)  = \frac{2}{\sigma_\eta n}\left( \frac{n\alpha(d)}{k} \right)^{1 + 2/d}\sum_{y\in \X_n}\eta_{\eps_k(x,y)}(|x-y|)(u(y) - u(x)).\]
Then by \cite[Theorem 3.6]{calder2019improved} we have that
\[\left|Lu(x) - \rho(x)^{-1}\div\left( \rho^{1 - 2/d}\nabla u \right)(x) \right| \leq C(1 + \|u\|_{C^3(B(x,\eps))})(t + \eps)  \]
holds for all $x\in \Omega_\eps\cap \X_n$ with probability at least $1 - Cn\exp\left( -c\left( \tfrac{k}{n} \right)^{2/d}kt^2 \right)$ for $0 \leq t \leq 1/4$, $k\leq c_2n$, and $\eps = c_2(k/n)^{1/d}$ with $c_1,c_2>0$. Noting that
\[\L^{n,k}_{s,rw}u(x) = \frac{k}{\alpha(d)d^{n,k}_s(x)}Lu(x)\]
and invoking Lemma \ref{lem:dnks} we have
\[\left|\L^{n,k}_{s,rw}u(x) - \rho(x)^{-1}\div\left( \rho^{1 - 2/d}\nabla u \right)(x) \right| \leq C(1 + \|u\|_{C^3(B(x,\eps))})(t + \eps)  \]
holds for all $x\in \Omega_\eps\cap \X_n$ with probability at least $1 - Cn\exp\left( -c\left( \tfrac{k}{n} \right)^{2/d}kt^2 \right)$.

To prove uniformity over $u\in C^3(\bar{\Omega})$, we fix $x\in \Omega_\eps$ and $u\in C^3(\bar{\Omega})$ and we first Taylor expand $u$ to find
\[u(y) = u(x) + \nabla u(x)\cdot (y-x) + \frac{1}{2}(y-x)^T \nabla^2u(x)(y-x) + O(\|u\|_{C^3(B(x,\eps))}\eps^3)\]
for $|x-y|\leq \eps$.  Therefore
\begin{equation}\label{eq:Lnkexp}
\L^{n,k}_{s,rw}u(x) = \sum_{i=1}^d u_{x_i}(x)\L_{s,rw}^{n,k}p_i(x) + \frac{1}{2}\sum_{i,j=1}^du_{x_ix_j}(x)\L^{n,k}_{s,rw}(p_ip_j)(x) + O(\|u\|_{C^3(B(x,\eps))}\eps), 
\end{equation}
where $p_i(y) = y_i-x_i$ and the error term is controlled with probability at least $1-C\exp\left( -ck \right)$, due to Lemma \ref{lem:dnks}. By the argument above
\[|\L_{s,rw}^{n,k}p_i(x) - (1-\tfrac{2}{d})\rho(x)^{-(1+2/d)}\rho_{x_i}(x)| \leq C(t+\eps),\]
and
\[|\L_{s,rw}^{n,k}(p_ip_j)(x) - 2\rho(x)^{-2/d}\delta_{ij} | \leq C(t+\eps)\]
hold for all $x\in \Omega_\eps\cap \X_n$ with probability at least $1 - Cn\exp\left( -c\left( \tfrac{k}{n} \right)^{2/d}kt^2 \right)$, where $\delta_{ij}=1$ if $i=j$ and $\delta_{ij}=0$ otherwise. Substituting this into \eqref{eq:Lnkexp} we have
\begin{align*}
\L^{n,k}_{s,rw}u(x) &=(1-\tfrac{2}{d})\rho^{-(1+2/d)}\sum_{i=1}^d u_{x_i}(x)\rho_{x_i}(x)+ \rho(x)^{-2/d}\sum_{i,j=1}^du_{x_ix_i}(x)\delta_{ij} + O(\|u\|_{C^3(B(x,\eps))}\eps), \\
&=\rho(x)^{-2/d}\left( (1-\tfrac{2}{d})\rho(x)^{-1}\nabla \rho(x)\cdot \nabla u(x) + \Delta u(x) \right)+ O(\|u\|_{C^3(B(x,\eps))}\eps)\\
&=\rho(x)^{-1}\div\left( \rho^{1 - 2/d}\nabla u \right)(x)  + O(\|u\|_{C^3(B(x,\eps))}\eps),
\end{align*}
which completes the proof.
\end{proof}

We require the following simple proposition.
\begin{proposition}\label{prop:max}
Let $\psi\in C^1(\bar{\Omega})$ and assume $\psi$ attains its maximum value over $\bar{\Omega}$ at some $x_0\in \Omega$. Then we have
\begin{equation}\label{eq:maxbound}
\left|\max_{x\in \Omega}\psi(x) - \max_{x\in \X_n}\psi(x)\right| \leq \|\nabla \psi\|_{L^\infty(\Omega)}\delta_n
\end{equation}
\end{proposition}
\begin{proof}
For any $x\in \X_n$ we have
\[\psi(x_0) - \psi(x) \leq \|\nabla \psi\|_{L^\infty(\Omega)}|x_0-x|,\]
and therefore
\[\delta_n \geq \min_{x\in \X_n}|x_0-x| \geq \|\nabla \psi\|_{L^\infty(\Omega)}^{-1}(\psi(x_0) - \max_{x\in \X_n}\psi(x)),\] 
where we recall $\delta_n$ was defined in \eqref{eq:rne}. 
\end{proof}
We define
\begin{equation}\label{eq:H}
H_{r,\eps}(v,p,X) = \max_{|z|=r}\left\{ p\cdot(z +[v\cdot z]_+\eps z) + \frac{\eps}{2}z^TXz  \right\}.
\end{equation}
The following proposition can be viewed as an extension of Proposition \ref{prop:maxprop}, which applies when $v=0$. It is used to evaluate the $\max$ and $\min$ in the graph $\infty$-Laplacian, asymptotically after Taylor expansions, in Lemma \ref{lem:knnLinf} below.
\begin{proposition}\label{prop:Hcell}
For any $\eps,r>0$, $p\in \R^d$ with $p\neq 0$ and $X\in \R^{d\times d}$ we have
\[\left|H_{r,\eps}(v,p,X) - r|p| - [v\cdot p]_+ r^2\eps - \frac{r^2\eps}{2|p|^2}p^T Xp\right|\leq 2|p|^{-1}(|p||v| + \|X\|)^2 r^3 \eps^2.\]
\end{proposition}
\begin{proof}
Let $z_r$ such that $|z_r|=r$ and 
\[H_{r,\eps}(v,p,X) = (p\cdot z_r)(1 + [v\cdot z_r]_+\eps) + \frac{\eps}{2}z_r^T Xz_r.\]
Set $w_r = rp/|p|$. Choosing $z=w_r$ yields
\[H_{r,\eps}(v,p,X) \geq r|p|(1 + [v\cdot w_r]_+\eps) + \frac{\eps}{2}w_r^T Xw_r.\]
Note that for any unit vectors $a,b$ we have
\[1 - a\cdot b = \frac{1}{2}(2 - 2a\cdot b) = \frac{1}{2}|a-b|^2.\]
Therefore
\[r|p| - p\cdot z_r = r|p|\left( 1 - \frac{p}{|p|}\cdot \frac{z_r}{r} \right) = \frac{r|p|}{2}\left| \frac{p}{|p|} -\frac{z_r}{r} \right|^2=\frac{|p|}{2r}\left| w_r -z_r \right|^2.\]
Combining this with the observations above we have
\begin{align*}
\frac{|p|}{2r}\left| w_r -z_r \right|^2&=r|p| - p\cdot z_r\\
&\leq \left( (p\cdot z_r)[v\cdot z_r]_+ + \frac{1}{2}z_r^TXz_r - r|p|[v\cdot w_r]_+ - \frac{1}{2}w_r^TXw_r  \right)\eps\\
&\leq  \left( r|p|\left([v\cdot z_r]_+ - [v\cdot w_r]_+\right) + \frac{1}{2}\left( z_r^TXz_r - w_r^TXw_r \right)  \right)\eps\\
&\leq \left(r|p||v||w_r - z_r| + r\|X\| |w_r-z_r| \right)\eps\\
&= (|p||v| + \|X\|)|w_r-z_r| r\eps.
\end{align*}
Therefore
\[|w_r-z_r| \leq 2\left( |v| + |p|^{-1}\|X\| \right)r^2\eps = 2|p|^{-1}\left( |p||v| + \|X\| \right)r^2\eps,\]
and so
\begin{align*}
H_{r,\eps}(v,p,X) -& p\cdot(w_r + [v\cdot w_r]_+ \eps w_r) - \frac{\eps}{2}w_r^T Xw_r\\
&=p\cdot(z_r + [v\cdot z_r]_+ \eps z_r)+ \frac{\eps}{2}z_r^TX z_r- p\cdot(w_r + [v\cdot w_r]_+ \eps w_r) - \frac{\eps}{2}w_r^T Xw_r\\
&=p \cdot z_r - p\cdot w_r + [v\cdot z_r]_+(p\cdot z_r)\eps -  [v\cdot w_r]_+(p\cdot w_r)\eps + \frac{1}{2}\left( z_r^TXz_r - w_r^TXw_r \right)\eps\\
&\leq p\cdot z_r - r|p| + r|p|\left([v\cdot z_r]_+ - [v\cdot w_r]_+\right)\eps + \|X\||w_r-z_r|r\eps \\
&\leq |p||v||w_r - z_r|r\eps + \|X\||w_r-z_r|r\eps \\
&= (|p||v| + \|X\|)|w_r-z_r|r\eps\\
&\leq 2|p|^{-1}(|p||v| + \|X\|)^2 r^3 \eps^2,
\end{align*}
which completes the proof.
\end{proof}

\begin{lemma}\label{lem:knnLinf}
There exists constants $C_1,C_2,c_1,c_2>0$ such that for $0 < t \leq 1$ and $k\leq c_1 n$, the event that
\begin{align*}
&\left|\L^{n,k}_{s,\infty}u(x)-\rho(x)^{-2/d}\left(\Delta_\infty u(x)-\tfrac{1}{d}\nabla \log \rho(x)\cdot \nabla u(x) \right) \right|\\
&\hspace{2in}\leq C_1\left|(1+|\nabla u(x)|^{-1}\|u\|^2_{C^2(B(x,\eps))} + \|u\|_{C^3(B(x,\eps))})(\delta_n\eps^{-2} + t+\eps)\right|
\end{align*}
holds for all $x\in \X_n\cap \Omega_\eps$, where $\eps= c_2(k/n)^{1/d}$, and for all $u\in C^3(\bar{\Omega})$, has probability at least $1-2n\exp\left( -C_2k\left( \tfrac{k}{n} \right)^{2/d}t^2 \right)$.
\end{lemma}
\begin{proof}
By Lemma \ref{lem:skewball} we have
\[\eta_{\eps_k(x,y)}(|x-y|) = \eta\left( \frac{|x-y|}{\eps_k(x,y)} \right) = \eta\left( \frac{|x-y|(1-[v\cdot(y-x)]_+}{s_k(x)} \right) + O(t\eps + \eps^2), \]
for $|x-y|\leq \eps_k(x,y)$, where $v = -\tfrac{1}{d}\nabla \log \rho(x)$, with probability at least $1-2n\exp\left( -C_2k\eps^2t^2 \right)$, where $0 < t \leq 1$ and $\eps = c(k/n)^{1/d}$. Therefore, by Proposition \ref{prop:max} we have
\begin{align*}
\max_{y\in \X_n}&\left\{\eta_{\eps_k(x,y)}(|x-y|) (u(y)-u(x))\right\}\\
&=\max_{y\in \X_n}\left\{\eta\left(\frac{|x-y|(1 - [v\cdot (y-x)]_+)}{s_k(x)}\right) (u(y)-u(x))\right\} + O(\|u\|_{C^1(B(x,\eps))}\eps^2(t+\eps))\\
&=\max_{y\in \Omega}\left\{\eta\left(\frac{|x-y|(1 - [v\cdot (y-x)]_+)}{s_k(x)}\right) (u(y)-u(x))\right\} + O\left( (1+\|u\|_{C^1(B(x,\eps))})(\delta_n + \eps^2t+\eps^3)\right)\\
\end{align*}
We now make a change of variables, setting
\[z = \Phi(y):= \frac{1}{s_k(x)}(y-x)(1 -  [v\cdot (y-x)]_+).\]
For $y$ sufficiently close to $x$, depending only on $|v|$, the mapping $y\mapsto \Phi(y)$ is invertible and 
\[y = \Phi^{-1}(z) = x + s_k z +[v\cdot z]_+s_k^2 z + O(\eps^3). \] 
where we write $s_k=s_k(x)$ for simplicity, and note that $s_k = O(\eps)$.
Therefore, using Proposition \ref{prop:Hcell} we have
\begin{align*}
\max_{y\in \X_n}&\left\{\eta_{\eps_k(x,y)}(|x-y|) (u(y)-u(x))\right\} \\
&=\max_{z\in B(0,1)}\left\{\eta(|z|) (u(x +s_k z + [v\cdot z]_+s_k^2z)-u(x))\right\}+ O\left( (1+\|u\|_{C^1(B(x,\eps))})(\delta_n + \eps^2t+\eps^3)\right) \\
&=s_k\max_{0 \leq r \leq 1}\left\{\eta(r)\max_{|z|=r}\left\{ \nabla u(x)\cdot ( z +[v\cdot z]_+s_k z) + \tfrac{s_k}{2}z^T\nabla^2 u(x)z\right\}\right\} \\
&\hspace{2in}+ O\left((1+\|u\|_{C^3(B(x,\eps))})(\delta_n + \eps^2t+\eps^3)\right)\\
&=s_k\max_{0\leq r \leq 1}\left\{ \eta(r) H_{r,s_k}(v,\nabla u(x),\nabla ^2u(x)) \right\} +O\left((1+\|u\|_{C^3(B(x,\eps))})(\delta_n + \eps^2t+\eps^3)\right)\\
&=s_k\max_{0\leq r \leq 1}\left\{ r\eta(r)|\nabla u(x)| + \left([v\cdot \nabla u(x)]_+  + \tfrac{1}{2}\Delta_\infty u(x)\right)r^2\eta(r)s_k\right\} \\
&\hspace{1in}+O\left((1+|\nabla u(x)|^{-1}\|u\|^2_{C^2(B(x,\eps))} + \|u\|_{C^3(B(x,\eps))})(\delta_n + \eps^2t+\eps^3)\right).
\end{align*}
Let us set
\[B = \max_{0\leq r \leq 1}\left\{ r\eta(r)|\nabla u(x)| + \left([v\cdot \nabla u(x)]_+  + \tfrac{1}{2}\Delta_\infty u(x)\right)r^2\eta(r)s_k\right\}\]
and let $r_1 \in [0,1]$ so that
\[B = r_1\eta(r_1)|\nabla u(x)| + \left([v\cdot \nabla u(x)]_+  + \tfrac{1}{2}\Delta_\infty u(x)\right)r_1^2\eta(r_1)s_k.\]
By \eqref{eq:strictmax} we have
\begin{align*}
B&\geq r_0\eta(r_0)|\nabla u(x)| + \left([v\cdot \nabla u(x)]_+  + \tfrac{1}{2}\Delta_\infty u(x)\right)r_0^2\eta(r_0)s_k\\
&\geq \theta(r_1-r_0)^2|\nabla u(x)|+  r_1\eta(r_1)|\nabla u(x)| + \left([v\cdot \nabla u(x)]_+  + \tfrac{1}{2}\Delta_\infty u(x)\right)r_0^2\eta(r_0)s_k.
\end{align*}
It follows that
\[\theta(r_1-r_0)^2|\nabla u(x)|\leq \left([v\cdot \nabla u(x)]_+  + \tfrac{1}{2}\Delta_\infty u(x)\right)\left( r_0^2\eta(r_0) - r_1^2\eta(r_1) \right)s_k\]
and so
\[|r_1-r_0|\leq C\theta^{-1}\left(1 + |\nabla u(x)|^{-1}|\Delta_\infty u(x)|\right)s_k.\]
This yields 
\begin{align*}
B &= r_0\eta(r_0)|\nabla u(x)| +\left([v\cdot \nabla u(x)]_+  + \tfrac{1}{2}\Delta_\infty u(x)\right)r_0^2\eta(r_0)s_k\\
&\hspace{2in}+O\left((1+|\nabla u(x)|^{-1})\|u\|^2_{C^2(B(x,\eps))}\eps^2\right).
\end{align*}
Inserting this above we have
\begin{align*}
\max_{y\in \X_n}&\left\{\eta_{\eps_k(x,y)}(|x-y|) (u(y)-u(x))\right\} \\
&\hspace{0.75in}= r_0\eta(r_0)|\nabla u(x)|s_k +\left(\tfrac{1}{2}\Delta_\infty u(x)+[v\cdot \nabla u(x)]_+ \right)r_0^2\eta(r_0)s_k^2\\
&\hspace{1.5in}+O\left((1+|\nabla u(x)|^{-1}\|u\|^2_{C^2(B(x,\eps))} + \|u\|_{C^3(B(x,\eps))})(\delta_n + \eps^2t+\eps^3)\right).
\end{align*}
We can apply the same argument as above to $-u$ to obtain 
\begin{align*}
\min_{y\in \X_n}&\left\{\eta_{\eps_k(x,y)}(|x-y|) (u(y)-u(x))\right\} \\
&\hspace{0.75in}= -r_0\eta(r_0)|\nabla u(x)|s_k +\left(\tfrac{1}{2}\Delta_\infty u(x) - [-v\cdot \nabla u(x)]_+ \right)r_0^2\eta(r_0)s_k^2\\
&\hspace{1.5in}+O\left((1+|\nabla u(x)|^{-1}\|u\|^2_{C^2(B(x,\eps))} + \|u\|_{C^3(B(x,\eps))})(\delta_n + \eps^2t+\eps^3)\right).
\end{align*}
Therefore
\begin{align*}
\L^{n,k}_{s,\infty}u(x)&=\rho(x)^{-2/d}\left(\Delta_\infty u(x)-\tfrac{1}{d}\nabla \log \rho(x)\cdot \nabla u(x)  \right)\\
&\hspace{1in}+O\left((1+|\nabla u(x)|^{-1}\|u\|^2_{C^2(B(x,\eps))} + \|u\|_{C^3(B(x,\eps))})(\delta_n + \eps^2t+\eps^3)\right).
\end{align*}
\end{proof}

\subsection{Discrete to continuum convergence}
\label{sec:continuum}

Many types of discrete to continuum convergence results have been proven recently for various learning problems on graphs, using tools like $\Gamma$-convergence \cite{trillos2019maximum,garcia2019variational,garciatrillos16}, the maximum principle and viscosity solutions \cite{calder2018game,calder2019consistency}, and even Martingale techniques \cite{calder2019rates}. The $\Gamma$-convergence results are variational in nature and do not apply to the game theoretic $p$-Laplacian. The viscosity solution approach in \cite{calder2018game} does not require a variational structure, but used uniform equicontinuity of the sequence of learned functions to establish compactness.

We present here a very general technique for proving discrete to continuum convergence results in a general setting that applies to all the symmetric and nonsymmetric operators discussed in Sections \ref{sec:epsgraph}, \ref{sec:nonknn} and \ref{sec:symknn}, and any others that satisfy certain monotonicity properties. The framework only requires the graph problem to have a maximum principle, and to be pointwise consistent with a well-posed equation in the continuum. It is not necessary to prove that the sequence of functions is equicontinuous, or that the solution of the continuum PDE is smooth. The framework is essentially an adaptation of the Barles-Souganidis framework \cite{barles1991convergence} to convergence of discrete problems on graphs.

Let $\Omega\subset \R^d$ be an open and bounded domain. For each $n\geq 1$, let $\X_n\subset \Omega$ be a collection of $n$ points in $\Omega$, and let $L^2(\X_n)$ denote the space of functions $u:\X_n\to \R$. Let 
\[F_n:L^2(\X_n)\times \R\times \X_n \to \R\]
denote our graph operator, generalizing a graph Laplacian, and suppose $u_n\in L^2(\X_n)$ is a solution of the boundary value problem
\begin{equation}\label{eq:Fnbvp}
\left\{\begin{aligned}
F_n(u_n,u_n(x),x)  &= 0,&&\text{if }x\in\X_n\cap \Omega_{\eps_n}\\ 
u_n(x) &=g(x),&&\text{if }x \in \X_n\cap \partial_{\eps_n}\Omega,
\end{aligned}\right.
\end{equation}
where $g:\Omega\to \R$ is a given continuous function, and $\eps_n>0$ represents the length scale on which the graph is connected. For convenience, we recall that $\partial_r \Omega= \{x\in \Omega \, : \, \text{dist}(x,\partial \Omega)\leq r\}$ and $\Omega_r = \Omega\setminus \partial_r \Omega$.  The graph equation \eqref{eq:Fnbvp} represents our generalization of a semi-supervised learning problem on a graph with labels $g$ on the set $\X_n\cap \partial_{\eps_n}\Omega$. This is just one model for labeled data, and others are possible (see \cite{calder2019rates} for other models). 

We now lay out simple and general conditions on $\X_n$, $u_n$ and $F_n$ that ensure $u_n$ has a well-posed continuum limit; that is, $u_n$ converges uniformly to the solution of a continuum PDE.   We say the points $\X_n$ are \emph{space filling} if
\begin{equation}\label{eq:Xn}
\lim_{n\to \infty} \sup_{x\in \Omega} \text{dist}(x,\X_n) = 0.
\end{equation}
We say $F_n$ is \emph{monotone} if for all $u,v\in L^2(\X_n)$, $t\in \R$ and $x\in \Omega$ 
\begin{equation}\label{eq:monotone}
 u\leq v \implies F_n(u,t,x) \geq F_n(v,t,x).
\end{equation}
We say $F_n$ is \emph{proper} if for all $u\in L^2(\X_n)$, $s,t\in \R$, and $x\in \Omega$
\begin{equation}\label{eq:proper}
 s \leq t \implies F_n(u,s,x) \leq F_n(u,t,x).
\end{equation}
We say the operators are \emph{consistent} as $n\to \infty$ with the differential operator\footnote{$\S(d)$ denotes the space of symmetric real-valued matrices of size $d\times d$.} 
\[F:\S(d)\times \R^d \times \R \times \Omega \to \R\]
if for every $\varphi\in C^\infty(\R^d)$ and every sequence of real numbers $\xi_n\to 0$ we have
\begin{equation}\label{eq:consistency}
\lim_{n\to \infty} \max_{x\in \X_n\cap \Omega_{\eps_n}} \left|F_n(\varphi + \xi_n,\varphi(x) + \xi_n,x) - F(\nabla^2 \varphi(x),\nabla \varphi(x),\varphi(x),x)\right|=0.
\end{equation}
We note that the connectivity length scale $\eps_n$ of the graph is encoded into the consistency statement, since we do not assume consistency near the boundary, where the ball $B(x,\eps_n)$ overlaps with $\partial\Omega$. 

\begin{remark}\label{rem:pc}
In the context of Sections \ref{sec:epsgraph}, \ref{sec:nonknn} and \ref{sec:symknn}, we would set $F_n(u,u(x),x) = -\L u(x)$, where $\L$ is any of the graph Laplacians in those sections. For example, if $\L$ is the unnormalized graph Laplacian \eqref{eq:graph2}, then 
\[F_n(u,t,x) = \sum_{y\in \X_n}w_{xy}(t - u(y)).\]
The first argument of $F_n$ encodes the dependence of $F_n$ on the neighboring values $u(y)$ for $y\in \X_n$ with $y\neq x$, and the second argument $t$ encodes the dependence on $t=u(x)$. Since all graph Laplacians are increasing functions of the difference $u(y)-t$,  they are both monotone and proper.  The consistency results established in Sections \ref{sec:epsgraph}, \ref{sec:nonknn} and \ref{sec:symknn} show that \eqref{eq:consistency} holds with probability one for any choice of $\L$ from those sections, provided $\eps=\eps_n \to 0$ sufficiently slowly. For example, in  Theorem \ref{thm:Ceps}, in order to make sure the results hold with high probability we require
\[n\eps^{3d/2}\gg \log(n) \ \ \text{ and }\ \ n\eps^{d+2}\gg \log (n).\]
This can be rewritten as lower bounds on the length scale $\eps$ as follows
\begin{equation}\label{eq:lowereps}
\eps\gg \left(\frac{\log(n)}{n}\right)^{2/(3d)} \ \ \text{ and }\ \ \eps\gg \left(\frac{\log (n)}{n}\right)^{1/(d+2)}.
\end{equation}
\label{rem:RGG}
\end{remark}

We now follow the Barles-Souganidis framework \cite{barles1991convergence} to prove that the sequence $u_n$ converges uniformly to the solution of the boundary value problem
\begin{equation}\label{eq:bvp}
\left\{\begin{aligned}
F(\nabla^2 u,\nabla u,u,x) &= 0,&&\text{in }\Omega\\ 
u &=g,&&\text{on }\partial \Omega,
\end{aligned}\right.
\end{equation}
provided the equation \eqref{eq:bvp} is well-posed in the viscosity sense with generalized Dirichlet condition $u=g$ on $\partial \Omega$. We review the definition of viscosity solution and the generalized Dirichlet problem in the appendix in Section \ref{sec:viscosity}.

\begin{theorem}\label{thm:D2C}
Assume $g$ is continuous, \eqref{eq:bvp} enjoys strong uniqueness, $F_n$ is monotone, proper, and consistent with $F$, $\X_n$ is space filling, and $\eps_n\to 0$ as $n\to \infty$. Let $u_n$ be a sequence of solutions of \eqref{eq:Fnbvp} that are uniformly bounded. Then
\begin{equation}\label{eq:unifom}
\lim_{n\to \infty} \max_{x\in \X_n}|u_n(x) - u(x)|=0,
\end{equation}
where $u\in C(\bar{\Omega})$ is the unique viscosity solution of \eqref{eq:bvp}.
\end{theorem}
\begin{remark}
In Theorem \ref{thm:D2C}, by uniformly bounded we mean that there exists $C>0$ such that $\max_{x\in \X_n}|u_n(x)| \leq C$ for all $n\geq 1$. Bounds of this nature are generally proved using the discrete maximum principle on the graph, and often rely on graph-connectivity. For a simple example, consider 
\[F_n(u,u(x),x) = u_n(x) + \alpha_n\sum_{y\in \X_n\cap  B(x,\eps_n) }(u_n(x)-u_n(y)) - 1.\]
For an appropriate choice of normalization parameter $\alpha_n$, the equation $F_n=0$ would be consistent with the elliptic equation $u - \Delta u =1$ in the continuum limit. To establish the uniform bound on $u_n$, we use a maximum principle argument. Let $x\in \X_n$ be a point where $u_n$ attains its maximum value over $\X_n$. If $x\in \partial_{\eps_n}\Omega$, then $u_n(x)\leq g(x)$. If $x\in \Omega_{\eps_n}$, then we use that $u_n(x) - u_n(y) \geq 0$ for any  $y\in \X_n$ to obtain
\[0=F_n(u,u(x),x) = u_n(x) + \alpha_n\sum_{y\in \X_n\cap  B(x,\eps_n) }(u_n(x)-u_n(y)) - 1 \geq u_n(x) - 1.\]
Therefore $u_n(x) \leq 1$ and we obtain the bound 
\[\max_{x\in \X_n} u_n(x) \leq \max\{1,\|g\|_\infty\}=:C.\]
A bound on the minimum of $u_n$ is obtained similarly.
\end{remark}
\begin{proof}[Proof of Theorem \ref{thm:D2C}]
We define the weak upper and lower limits
\[\bar{u}(x) = \limsup_{\substack{n\to \infty\\ \Omega_n\ni y \to x }} u_n(y) \ \ \text{ and } \ \ \underline{u}(x) = \liminf_{\substack{n\to \infty\\ \Omega_n\ni y \to x }} u_n(y).\]
Due to \eqref{eq:Xn} and the assumption that $u_n$ are uniformly bounded, we have $-\infty < \underline{u}(x) \leq \bar{u}(x) < \infty$ for all $x\in \bar{\Omega}$. Furthermore, it is straightforward to check (see, e.g., \cite{calderViscosity}) that $\bar{u}\in \usc(\bar{\Omega})$ and $\underline{u}\in \lsc(\bar{\Omega})$, where $\usc(\bar{\Omega})$ and $\lsc(\bar{\Omega})$ denote the spaces of upper semicontinuous and lower semicontinuous functions on $\bar{\Omega}$, respectively. We claim that $\bar{u}$ is a viscosity subsolution of \eqref{eq:bvp} and $\underline{u}$ is a viscosity supersolution. Once we establish this, it follows from strong uniqueness that $\bar{u}\leq \underline{u}$. Therefore $\bar{u}=\underline{u}$ and \eqref{eq:unifom} immediately follows.

We will show that $\bar{u}$ is a viscosity subsolution; the proof that $\underline{u}$ is a supersolution is similar. Let $x_0\in \bar{\Omega}$ and $\phi\in C^\infty(\R^d)$ such $\bar{u}-\phi$ has a local maximum at $x_0$ with respect to $\bar{\Omega}$. Define
\[\psi(x) = \phi(x) + C|x-x_0|^4 + \bar{u}(x_0) - \phi(x_0).\]
Then $\psi(x_0) = \bar{u}(x_0)$, and we can choose $C>0$ large enough so that $\bar{u}-\psi$ has a strict global maximum at $x_0$ relative to $\bar{\Omega}$. It follows that there exists $n_k\to \infty$, $\X_{n_k}\ni x_{n_k}\to x_0$ with $u_{n_k}(x_{n_k}) \to \bar{u}(x_0)$ such that $u_{n_k} - \psi$ attains its maximum value over $\X_{n_k}$ at the point $x_{n_k}$ for each $k$. Set $\xi_k = u_{n_k}(x_{n_k}) - \psi(x_{n_k})$ so that $u_{n_k}\leq \psi + \xi_k$.  Then since $F_n$ is monotone (i.e., \eqref{eq:monotone} holds) we have
\begin{equation}\label{eq:Fnz}
F_{n_k}(\psi + \xi_k,\psi(x_{n_k}) + \xi_k,x_{n_k}) \leq F_{n_k}(u_{n_k},u_{n_k}(x_{n_k}),x_{n_k})  = 0.
\end{equation}

We now have two cases.  

1. If $x_0 \in \Omega$ then $x_{n_k}\in \X_{n_k}\cap \Omega_{\eps_{n_k}}$ for $k$ sufficiently large, and since $\xi_k\to 0$ as $k\to \infty$ we can combine consistency (i.e., \eqref{eq:consistency}) with \eqref{eq:Fnz} to find that 
\[F(\nabla ^2 \psi(x_0),\nabla \psi(x_0),\psi(x_0),x_0) \leq 0.\]
Since $\nabla ^2 \psi(x_0) = \nabla ^2 \phi(x_0)$, $\nabla \psi(x_0)=\nabla \phi(x_0)$ and $\psi(x_0)=\bar{u}(x_0)$ we have, as desired, the viscosity subsolution condition
\begin{equation}\label{eq:subs}
F(\nabla ^2 \phi(x_0),\nabla \phi(x_0),\bar{u}(x_0),x_0) \leq 0.
\end{equation}

2. If $x_0\in \partial \Omega$, then we can pass to a further subsequence, if necessary, so that either $x_{n_k}\in \X_{n_k}\cap \Omega_{\eps_{n_k}}$ or $x_{n_k}\in \X_{n_k}\cap \partial_{\eps_{n_k}}\Omega$ for all $k$. In the former case we again find that \eqref{eq:subs} holds. In the latter case, we have $u_{n_k}(x_{n_k})= g(x_{n_k})$ for all $k$, and since $g$ is continuous we have that $\bar{u}(x_0)=g(x_0)$. Thus, when $x\in \partial \Omega$ we have
\[\min\left\{ F(\nabla ^2\phi(x_0),\nabla \phi(x_0),\bar{u}(x_0),x_0), \bar{u}(x_0)-g(x_0)\right\}\leq 0, \]
which is the viscosity subsolution condition on the boundary. This completes the proof.
\end{proof}
\begin{remark}
In the context of semi-supervised learning, Theorem \ref{thm:D2C} shows that the discrete graph problems are well-posed in the continuum with $O(n\eps_n)$ labeled data points, which is the number of points in $\X_n\cap \partial_{\eps_n}\Omega$, provided $\X_n$ are roughly evenly spread (e.g., an \emph{i.i.d.}~sequence with Lebesgue density). This is a labeling rate of $O(\eps_n)$, which vanishes as $n\to \infty$. However, it may vanish very slowly, since pointwise consistency requires lower bounds on $\eps_n$, as explained in Remark \ref{rem:pc} (see Eq.~\eqref{eq:lowereps}). We emphasize that this is a general result that is \emph{independent} of the structure of the learning algorithm or of the continuum equation. For the game-theoretic $p$-Laplacian with $p>d$, it was shown in \cite{calder2018game} that the $\eps$-ball game-theoretic $p$-Laplacian is well-posed in the continuum with $O(1)$ labels, which is a labeling rate of $O(1/n)$. This is far smaller than $O(\eps_n)$ due to \eqref{eq:lowereps}. A similar result was proved for the variational $p$-Laplacian in \cite{slepcev2019analysis}, though here there is a restriction on $\eps_n$ for well-posedness, even when $p>d$. In \cite{calder2019rates}, it is shown that when $p=2$, Laplacian regularization is well-posed at the lower label rate of $O(\eps_n^2)$ using random walk techniques. For the $p$-Laplacian with $2\leq p \leq d$, it is an open problem to determine the lowest labeling rate at which the algorithm has a well-posed continuum limit. One would conjecture the lowest rate is $O(\eps_n^p)$.
\label{rem:labelingrate}
\end{remark}

\section{Algorithms for p-Laplacian learning}
\label{sec:alg}

We now present algorithms for $p$-Laplacian learning with both the game-theoretic and variational $p$-Laplacians. Section \ref{sec:newton} reviews how to apply Newton's method to the variational $p$-Laplace equation, and discuss how to apply homotopy on $p$ to accelerate convergence. Section \ref{sec:game_theoretic} presents three algorithms for solving the game-theoretic $p$-Laplacian on a graph: a gradient descent approach, a Newton-like method, and a semi-implicit algorithm.

\subsection{Newton's Method for variational p-Laplacian\label{sec:newton}}

Since $J_p$ is smooth and convex, it is natural to use Newton's method to minimize $J_p$. We give here the explicit details of the Newton iteration for minimizing $J_p$. It is useful to first rewrite the function $J_p(u)$ using vector notation. Let $\mX = \{ x^1, \dots, x^{n+m} \}\subset \R^d$, where $\mO = \{  { x}^{n+1}, \dots, { x}^{n+m} \}$ is the observation set. We define $u_i = u(x^i)$ and set ${\bf u}= (u_1, \dots, u_{n}) \in \R^{n}$. Similarly, set $w_{ij} = w_{x_i x_j}$, 
$f_i = f(x^i)$, 
$g_i=g(x^{i+n})$, 
${\bf f} = (f_1, \dots, f_{n}) \in \R^{n}$, 
and ${\bf g} = (g_1, \dots, g_m) \in \R^m$. Then, subject to the constraints in \eqref{eq:plap_optimality}, we can write
\begin{equation} J_p({\bf u}) = \frac{1}{p} \Bigg(\sum^{n}_{i=1} \sum^n_{j=i+1} w_{ij} |u_i - u_j|^p + \sum^{n}_{i=1} \sum^{m}_{j=1} w_{i,j+n} |u_i - g_j|^p \Bigg) + \sum^{n}_{i=1} f_i u_i
.\label{eq:Hp_definition}\end{equation}

Newton's method corresponds to the iteration
\begin{equation} {\bf u}^{k+1} = {\bf u}^k - \Big[\nabla^2 J_p({\bf u}^k)\Big]^{-1} \nabla J_p({\bf u}^k). \label{eq:newton_iteration}\end{equation}
For notational convenience, define $a_{ij}({\bf u})=w_{ij}|u_i-u_j|^{p-2}$ and $b_{ij}({\bf u})=w_{i,j+n}|u_i-g_j|^{p-2}$, and 
\begin{equation}  
d_i({\bf u}) = \sum^n_{j=1} a_{ij} ({\bf u}) + \sum^m_{j=1} b_{ij} ({\bf u}).
\end{equation}
Then, we write $A({\bf u}) = \big(a_{ij} ({\bf u}) \big)_{ij} \in \R^{n \times n}, B({\bf u}) = \big(b_{ij} ({\bf u}) \big)_{ij} \in \R^{n \times m}$ and $D({\bf u}) = \text{diag}(d_i({\bf u}))\in \R^{n\times n}$. In this notation, we compute 
\[ \nabla J_p({\bf u}) = L({\bf u}){\bf u} - B({\bf u}) {\bf g}  + {\bf f} 
\hsp\hsp\mbox{ and } \hsp\hsp \nabla^2 J_p({\bf u}) = (p - 1) L({\bf u}), \]
where $L( {\bf u}) := D( {\bf u}) - A({\bf u})$,  and thus the Newton update is given by
\begin{equation} {\bf u}^{k+1} = \frac{p-2}{p-1} {\bf u}^k + \frac{1}{p-1} L({\bf u}^k)^{-1}  
\Big[ B({\bf u}^k) {\bf g} - {\bf f} \Big]
. \label{eq:newton_iteration_matrix} \end{equation}
The inversion of $L({\bf u}^k)$ is performed with an iterative method, such as the preconditioned conjugate gradient. The matrix $L({\bf u}^k)$, being a graph-Laplacian, is always positive semi-definite. If the graph is connected and ${\bf u}$ is nondegenerate, then it is also non-singular.  
\begin{proposition}\label{prop:L}
If the graph on $n$ nodes with weights $(a_{ij}({\bf u}))_{i,j=1}^n$ is connected, and $b_{ij}({\bf u})>0$ for some $i,j$, then $L({\bf u})$ is positive definite. 
\end{proposition}
\begin{proof}
Since $L({\bf u})$ is positive semi-definite, we simply have to prove that $L({\bf u})$ is non-singular. The proof follows a maximum principle argument. If $L({\bf u})\x = 0$, $\x=(x_1,\dots,x_n)$, then we have
\begin{equation}\label{eq:graphL}
\sum_{j=1}^n a_{ij}({\bf u})(x_i - x_j) + x_i\sum_{j=1}^m b_{ij}({\bf u}) = 0
\end{equation}
for all $i=1,\dots,n$. 

Let $i$ be an index for which $x_i \geq x_j$ for all $j$; that is, the node where $\x$ attains its maximum over the graph. We have two cases now.

Case 1. If $\sum_{j=1}^m b_{ij}({\bf u})=0$, then it follows from \eqref{eq:graphL}, and the fact that $x_i-x_j\geq 0$ for all $j$, that
\[ a_{ij}({\bf u})(x_i - x_j) =0  \ \ \text{for all }j=1,\dots,n.\]
 Thus, at any neighbor in the graph where $a_{ij}({\bf  u})>0$ we have $x_i=x_j$.

Case 2. If $\sum_{j=1}^m b_{ij}({\bf u})>0$, then it follows from \eqref{eq:graphL} that 
\[x_i\sum_{j=1}^m b_{ij}({\bf u}) \leq 0,\]
and so $x_i\leq 0$.

Note the observations above hold for any node $i$ where $\x$ attains its maximum value. We claim that these observations imply that $\max_{1\leq j \leq n}x_j \leq 0$. To see this, choose an index $i$ for which $\x$ attains it maximum value, and let $k\in \{1,\dots,n\}$ be a node for which
\begin{equation}\label{eq:cond}
\sum_{j=1}^m b_{kj}({\bf u})>0,
\end{equation}
which is guaranteed to exist by assumption. Since the graph with weights $(a_{ij}({\bf u}))_{i,j=1}^n$ is assumed connected, we can construct a path $i=i_0,i_1,\dots,i_q=k$ between node $i$ and $k$ for which $a_{i_\ell,i_{\ell+1}}({\bf u})>0$ for $\ell=0,\dots,q-1$. We can assume that 
\[\sum_{j=1}^m b_{i_{\ell},j}({\bf u})=0 \ \ \text{ for }\ell=0,\dots,q-1\]
or else we can redefine $k$ as the earliest node along the path for which \eqref{eq:cond} holds. Thus, we can apply Case 1 above along the path to show that $x_{i_{\ell}}=x_{i_{\ell+1}}$ for $\ell=0,\dots,q-1$. Thus, $x_i=x_k$. Then we apply Case 2 above, since \eqref{eq:cond} holds, to show that $x_i=x_k\leq 0$, which proves the claim.

We have proved that if $L({\bf u})\x = 0$ then $x_i\leq 0$ for all $i$. Since the equation is linear, we also have $L({\bf u})(-\x) = 0$ and so $-x_i\leq 0$ for all $i$. Thus, $L({\bf u})\x = 0$ implies $\x\equiv 0$, and so $L({\bf u})$ is nonsingular.
\end{proof}
\begin{remark}
The conditions in Proposition \ref{prop:L} hold when the original graph $G=(\X,\mW)$ is connected, and $u_i\neq u_j$ for all $i\neq j$, and $u_i\neq g_j$ for some $i,j$ with $w_{i,j+n}>0$. In other words, if all the values $\{u_1,\dots,u_n,g_1,\dots,g_m\}$ are unique then $L({\bf u})$. We can get into trouble when ${\bf u}$ is constant, or locally constant, since the conditions in the Proposition \ref{prop:L} fail to hold.  In practice we find that $L({\bf u})$ remains non-singular throughout the Newton iterations.
\label{lab:Lnewton}
\end{remark}

\begin{remark}
By the Newton-Kantorovich Theorem~\cite{ortega1968newton}, Newton's method is guaranteed to converge provided the initial guess is sufficiently close to the true solution. Since $J_p$ is convex, but not strongly convex for $p>2$, convergence may not be quadratic. In fact, according to \cite{flores2018algorithms}, the number of iterations required for Newton's method to converge may scale at least linearly in some cases. 
\end{remark}

\begin{remark}[Homotopy on $p$]
Let us mention that one can significantly speed up Newton's method with a good starting point ${\bf u^0}$, since most of the computational cost of Newton's method comes from approaching the solution, and once we reach the quadratic convergence region, only a few steps are required for convergence. As a result, we can solve \eqref{eq:plap_optimality} efficiently with a homotopy method in $p$. That is, we compute the solution for an increasing sequence of values of $p$, starting at $p=2$, and initializing Newton's method each time from the solution from the previous value of $p$. If the steps in $p$ are small enough, this initialization falls in the quadratic convergence region and only a handful of Newton iterations are required for each increment in $p$.  This approach is known in the literature as homotopy, and has been applied to a variety of problems in order to improve performance on related problems. For example, see \cite{4785969, fletcher1971calculation}. We illustrate the effectiveness of Newton's method with homotopy in Section \ref{sec:homotopy_numerics}.
\label{rem:homotopy}
\end{remark}

\subsection{Algorithms for the game theoretic formulation \label{sec:game_theoretic}}

We now consider algorithms for solving the game theoretic problem \eqref{eq:plap_graph_game}. We shall discuss a gradient descent-type method in Section \ref{sec:game_grad}, a Newton-like algorithm in Section \ref{sec:newtonlike}, and a semi-implicit method in Section \ref{sec:semi}. 

\subsubsection{Gradient Descent Approach\label{sec:game_grad}}
We first consider a gradient descent-type approach to solving \eqref{eq:plap_graph_game}, which is based on iterating 
\begin{equation}  \label{eq:game_grad_iteration}
u^{k+1}(x) = 
\begin{cases}
u^k(x) + \alpha \big( \mathscr{L}^G_p u^k(x) + f(x) \big),&\text{if }x\in \X\setminus \mO,\\
g(x),&\text{if }x\in \mO,
\end{cases}
\end{equation}
where $\alpha>0$ is the time step. We call this a gradient descent-type iteration, where $\mathscr{L}^G_p u^k(x)+f(x)$ plays the role of a gradient, but it is important to point out it is not gradient descent, since the game theoretic $p$-Laplacian does not arise from a variational principle. The most straightforward stopping condition is to fix $\eps>0$ and iterate until
\[|\mathscr{L}^G_p u^k(x) + f(x) |\leq \eps\]
for all $x\in \X\setminus \mO$.  However, this stopping condition does not guarantee that $u^k$ is within $\eps$ of the solution $u$ of the $p$-Laplacian learning problem \eqref{eq:plap_graph_game}, since the stability of the operator $\mathscr{L}^G_p $ depends in complicated ways on the graph and choice of boundary nodes $\mO$. We present here a modification of the gradient descent method that uses the comparison principle to inform the stopping condition.

\begin{lemma}[Comparison principle]\label{lem:comparison}
Assume $w_{xy}\leq 1$ for all $x,y\in \X$. Let $u^k,v^k:\X\to \R$ satisfy 
\begin{equation}\label{eq:subGD}
u^{k+1}(x) \leq u^k(x) +\alpha \big( \mathscr{L}^G_p u^k(x) + f(x) \big) 
\end{equation}
and
\begin{equation}\label{eq:superGD}
v^{k+1}(x) \geq v^k(x) +\alpha \big( \mathscr{L}^G_p v^k(x) + f(x) \big) 
\end{equation}
for all $x\in \X\setminus \mO$ and $0 \leq k \leq T-1$, where $T\in \N$, and assume that $u^0 \leq v^0$ and $u^k(x)\leq v^k(x)$ for all $x\in \mO$ and $1\leq k \leq T$. If $\alpha \leq p / (2p - 3)$ then $u^k\leq v^k$ on $\X$ for all $0 \leq k \leq T$.
\end{lemma}
\begin{proof}
Fix $u,v:\X\to \R$ and assume $u(x)\leq v(x)$ for all $x\in \X$. Fix $x\in \X$ and let $y_1,y_2\in \X$ such that 
\[\min_{y\in \X}w_{xy}(v(y)-v(x)) = w_{xy_1}(v(y_1)-v(x))\]
and
\[\max_{y\in \X}w_{xy}(u(y)-u(x)) = w_{xy_2}(u(y_2)-u(x)).\]
Since $w_{xy}\leq 1$ and $\alpha \leq p / (2p - 3)$ we have that
\[1-\frac{\alpha}{p} -\frac{\alpha}{p} (p-2)(w_{xy_1}+w_{xy_2}) \geq 0.\]
We can now compute
\begin{align*}
u(x) + \alpha\L^G_p u(x) &=u(x) +  \frac{\alpha}{d_x p} ~ \Delta^G_2 u(x) + \alpha\left(1-\tfrac{2}{p}\right) \Delta^G_\infty u(x) \\
&=u(x) + \frac{\alpha}{d_x p}\sum_{y\in \mX}w_{xy}(u(y)-u(x)) \\
&\hspace{1in}+ \alpha(1-\tfrac{2}{p}) \left\{\min_{y\in \mX}w_{xy}(u(y)-u(x)) + \max_{y\in \mX}w_{xy}(u(y)-u(x))\right\}\\
&\leq\left( 1-\frac{\alpha}{p} -\frac{\alpha}{p} (p-2)(w_{xy_1}+w_{xy_2}) \right)u(x) + \frac{\alpha}{p}\frac{1}{d_x }\sum_{y\in \mX}w_{xy}u(y)\\
&\hspace{2.5in}+ \frac{\alpha}{p}(p-2)(w_{xy_1}u(y_1)+ w_{xy_2}u(y_2))\\
&\leq\left( 1-\frac{\alpha}{p} -\frac{\alpha}{p} (p-2)(w_{xy_1}+w_{xy_2}) \right)v(x) + \frac{\alpha}{p}\frac{1}{d_x }\sum_{y\in \mX}w_{xy}v(y)\\
&\hspace{2.5in}+ \frac{\alpha}{p}(p-2)(w_{xy_1}v(y_1)+ w_{xy_2}v(y_2))\\
&\leq v(x) + \frac{\alpha}{d_x p}\sum_{y\in \mX}w_{xy}(v(y)-v(x)) \\
&\hspace{1in}+ \alpha(1-\tfrac{2}{p}) \left\{\min_{y\in \mX}w_{xy}(v(y)-v(x)) + \max_{y\in \mX}w_{xy}(v(y)-v(x))\right\}\\
&\leq v(x) + \alpha\L^G_p v(x).
\end{align*}
Thus, we have shown that when $\alpha \leq p / (2p - 3)$ we have
\[u\leq v \implies u + \alpha\L^G_p u \leq v + \alpha\L^G_p v.\]
The proof is completed by using \eqref{eq:subGD} and \eqref{eq:superGD} to write
\[u^{k+1}(x) - v^{k+1}(x)\leq u^k(x) + \alpha\L^G_p u^k(x) - (v^k(x) + \alpha\L^G_p v^k(x)),\]
for $x\in \X\setminus \mO$ and using induction.
\end{proof}
   
We present our method in the case that $f(x)=0$. We define 
\begin{equation}\label{eq:uzero}
\bar{u}^0(x) = 
\begin{cases}
\max_{x\in \mO}g(x),&\text{if }x\in \X\setminus \mO,\\
g(x),&\text{if }x\in \mO,
\end{cases}
\end{equation}
and
\begin{equation}\label{eq:uzero2}
\underline{u}^0(x) = 
\begin{cases}
\min_{x\in \mO}g(x),&\text{if }x\in \X\setminus \mO,\\
g(x),&\text{if }x\in \mO,
\end{cases}
\end{equation}
and define $\bar{u}^k(x)$ and $\underline{u}_k(x)$ for $k\geq 1$ by
\begin{equation}\label{eq:iter}
\bar{u}^{k+1}(x) = \bar{u}^k(x) + \alpha \L^G_p \bar{u}^k(x), \ \ \text{ and } \ \  \underline{u}^{k+1}(x) = \underline{u}^k(x) + \alpha \L^G_p \underline{u}^k(x),
\end{equation}
for $x\in \X\setminus \mO$, and $\bar{u}^{k+1}(x) = \underline{u}^{k+1}(x) = g(x)$ for $x\in \mO$. 
\begin{remark}
The extension of this method to $f\neq 0$ is not immediately obvious. It is important for the convergence analysis in Theorem \ref{thm:GDconv} below that $\bar{u}^0$ and $\underline{u}^0$ are super- and subsolutions of \eqref{eq:plap_graph_game}, respectively. When $f$ is nonzero, it is not clear how to construct such super- and subsolutions to initialize the method. We leave the extension of the method to nonzero $f$ to future work. We recall that in semi-supervised learning we always take $f=0$, so this extension is not needed for machine learning applications (though it may be of interest in numerical analysis).
\label{rem:fnonzero}
\end{remark}

   We prove below that this iteration scheme converges to the solution $u_*:\X\to \R$ of  \eqref{eq:plap_graph_game} with $f=0$. 
We remark that the solution $u_*$ of \eqref{eq:plap_graph_game} is unique when the graph is connected \cite{calder2018game,calder2019consistency}. 
\begin{theorem}[Convergence]\label{thm:GDconv}
Assume $w_{xy}\leq 1$ for all $x,y\in \X$, and assume the graph is connected. Let $u_*$ be the solution of \eqref{eq:plap_graph_game} with $f=0$, and define $u^k(x) = \tfrac{1}{2}(\bar{u}^k(x) + \underline{u}_k(x))$. If $\alpha \leq p / (2p - 3)$ then for all $k\geq 1$ and all $x\in \X$ we have
\begin{equation}\label{eq:GDconv}
|u^k(x) - u_*(x)| \leq \frac{1}{2}|\bar{u}^k(x) - \underline{u}_k(x)|
\end{equation}
and
\begin{equation}\label{eq:conv}
\lim_{k\to \infty}\bar{u}^k(x) = \lim_{k\to \infty}\underline{u}^k(x) = u_*(x).
\end{equation}
\end{theorem}
\begin{remark}
Theorem \ref{thm:GDconv} proves convergence of the gradient descent-type scheme, and gives us a simple way to set the stopping condition. If we fix $\eps>0$ and iterate until
\begin{equation}\label{eq:stopping}
|\bar{u}^k(x) - \underline{u}_k(x)|\leq 2\eps,
\end{equation}
then Theorem \ref{thm:GDconv} guarantees that $u^k(x) = \tfrac{1}{2}(\bar{u}^k(x) + \underline{u}_k(x))$ satisfies $|u^k(x)-u_*(x)|\leq \eps$.
\label{rem:stopping}
\end{remark}
\begin{proof}
By Lemma \ref{lem:comparison} we have
\[\underline{u}_k(x) \leq u_*(x) \leq \bar{u}^k(x)\]
for all $k\geq 0$ and $x\in \X$. It follows that
\[u^k(x) - u_*(x) \leq \frac{1}{2}(\bar{u}^k(x) + \underline{u}_k(x)) - \underline{u}_k(x) \leq \frac{1}{2} |\bar{u}^k(x) - \underline{u}_k(x)|,\]
and
\[u_*(x) - u^k(x) \leq \bar{u}^k(x) - \frac{1}{2}(\bar{u}^k(x) + \underline{u}_k(x)) \leq \frac{1}{2} |\bar{u}^k(x) - \underline{u}_k(x)|.\]
This establishes \eqref{eq:GDconv}.

We now prove \eqref{eq:conv}. We claim that $\L^G_p \bar{u}^k(x) \leq 0$ for all $k\geq 0$ and $x\in \X\setminus \mO$. The proof is by induction. Since $\bar{u}^0(y) - \bar{u}^0(x) \leq 0$ for all $x\in \X\setminus \mO$ and all $y\in \X$, the case $k=0$ is trivial. Now assume $\L^G_p \bar{u}^k(x) \leq 0$ for all $x\in \X\setminus \mO$. Fix $x_0\in \X\setminus \mO$ and define
\[w(x) = 
\begin{cases}
\bar{u}^k(x_0) + \alpha\L^G_p \bar{u}^k(x_0),&\text{if }x=x_0,\\
\bar{u}^k(x),&\text{if }x\neq x_0.
\end{cases}\]
Then $w\leq \bar{u}^k$ and so 
\[w(x_0) + \alpha\L^G_p w(x_0) \leq \bar{u}^k(x_0) + \alpha\L^G_p \bar{u}^k(x_0) = \bar{u}^{k+1}(x_0),\]
as in the proof of Lemma \ref{lem:comparison}. Since $w(x_0)=\bar{u}^{k+1}(x_0)$ and $w \geq \bar{u}^{k+1}$, we have 
\[0 \geq \L^G_p w(x_0) \geq \L^G_p \bar{u}^{k+1}(x_0),\]
which establishes the claim. 

By a similar argument, we have that $\L^G_p \underline{u}^k(x) \geq 0$ for all $k\geq 0$ and $x\in \X\setminus \mO$. It follows that $\bar{u}^{k+1}\leq \bar{u}^k$ and $\underline{u}^{k+1} \geq \underline{u}^k$ for all $k\geq 0$. Therefore, there exists $\bar{u},\underline{u}:\X\to \R$ such that
\[\lim_{k\to \infty}\bar{u}^k(x) = \bar{u}(x) \ \  \text{ and }\ \ \lim_{k\to \infty}\underline{u}^k(x) = \underline{u}(x).\]
By continuity of $\L^G_p$ we have that $\bar{u}$ and $\underline{u}$ both satisfy \eqref{eq:plap_graph_game}. Since the graph is connected, solutions of \eqref{eq:plap_graph_game} are unique, and hence $\bar{u}=\underline{u}=u_*$, which completes the proof.
\end{proof}

If we make a minor modification to the equation, then we can obtain a linear convergence rate. Let $\eps>0$ and consider the iteration 
\begin{equation}\label{eq:iter2}
u_\eps^{k+1}(x) = u_\eps^k(x) + \alpha (\L^G_p u_\eps^k(x) - \eps u_\eps^k(x)),
\end{equation}
for $x\in \X\setminus \mO$, and $u^{k+1}(x) = g(x)$ for $x\in \mO$. We prove below that this iteration scheme converges at a \emph{linear rate} to the solution $u_\eps:\X\to \R$ of
\begin{equation}\label{eq:plapeps}
\left\{\begin{aligned}
\eps u_\eps - \L^G_p u_\eps(x) &= 0&&\text{if }x\in \X\setminus \mO\\ 
u_\eps(x) &=g(x)&&\text{if }x\in \mO.
\end{aligned}\right.
\end{equation}
When $\eps>0$, the solution $u_\eps$ of \eqref{eq:plapeps} is unique even when the graph is disconnected. We have the following convergence theorem, which is an adaptation of  the contraction argument from \cite{oberman2006convergent}.
\begin{theorem}[Convergence rate]\label{thm:LinearRate}
Assume $w_{xy}\leq 1$ for all $x,y\in \X$ and $\alpha \leq p / ((2+\eps)p - 3)$. Let $u^0:\X\to \R$ and define $u_\eps^k$ by \eqref{eq:iter2}, and let $u_\eps$ be the solution of \eqref{eq:plapeps}. Then for every $k\geq 0$ we have
\begin{equation}\label{eq:rate}
\max_{x\in \X}|u^k_\eps(x) - u_\eps(x)| \leq (1-\eps)^k\max_{x\in \X}|u^0(x) - u_\eps(x)|.
\end{equation}
\end{theorem}
\begin{remark}
Letting $u_*$ solve \eqref{eq:plap_graph_game} and $u_\eps$ solve \eqref{eq:plapeps}, it follows from Theorem \ref{thm:LinearRate} that
\[\max_{x\in \X}|u^k_\eps(x) - u_*(x)| \leq C\left[\max_{x\in \X}|u_\eps(x) - u_*(x)| + (1-\eps)^k\right].\]
If we can quantify the error $|u_\eps - u_*|$, then this would prove a convergence rate for the original problem \eqref{eq:plap_graph_game} with $\eps=0$. However, it seems that proving error estimates between \eqref{eq:plap_graph_game} and \eqref{eq:plapeps} would require some additional strong assumptions on properties of the graph.
\label{rem:comb}
\end{remark}
\begin{proof}
Define
\begin{equation}\label{eq:Phi}
\Phi_\eps[u](x) =
\begin{cases}
u(x) + \alpha(\L^G_p u(x) - \eps u(x)),&\text{if }x\in \X\setminus \mO,\\
g(x),&\text{if }x\in \mO.
\end{cases}
\end{equation}
Then we have $u^k_\eps = \Phi^k_\eps[u^0]$. As in the proof of Lemma \ref{lem:comparison} we have
\[u\leq v \implies \Phi_0[u] \leq \Phi_0[v].\]
We also have
\[\Phi_0[u+C] = \Phi_0[u]+C\]
for any constant $C>0$.  It follows that 
\[\Phi_0[u] - \Phi_0[v]=\Phi_0[u - \max_{x\in \X}(u-v)] - \Phi_0[v] + \max_{x\in \X}(u-v) \leq \max_{x\in \X}(u-v).\]
Since
\[\Phi_\eps[u]=(1-\alpha\eps)\left( u + \frac{\alpha}{1-\alpha\eps}\L^G_p u \right),\]
we have
\[\Phi_\eps[u] - \Phi_\eps[v] \leq (1-\eps)\max_{x\in \X}(u-v)\]
provided 
\[\frac{\alpha}{1-\alpha\eps}\leq \frac{p}{2p-3},\]
which is equivalent to
\[\alpha \leq \frac{p}{(2+\eps)p - 3}.\]
It follows that $\Phi_\eps$ is a contraction in the norm $\|u\|_\infty:=\max_{x\in \X}|u(x)|$, and so there exists a unique fixed point $u_\eps:\X\to \R$ such that $\Phi_\eps[u_\eps]=u_\eps$ and
\[\|u^k_\eps - u_\eps\|_\infty \leq (1-\eps)^k\|u^0 - u_\eps\|_\infty\] 
for any $u:\X\to \R$. This completes the proof.
\end{proof}

\subsubsection{A Newton-like Algorithm\label{sec:newtonlike}}

We now consider a Newton-like method for solving \eqref{eq:plap_graph_game}. Newton's method is based on iteratively solving a linearized version of the problem.  In order to linearize (\ref{eq:plap_graph_game}), we define $y^k_+$ and $y^k_-$ by
\begin{equation}
\begin{alignedat}{2}
y^k_+(x) &\in \mbox{argmax} ~~ w_{xy} ( u^k(x) - u^k(y) ) \\
y^k_-(x)  &\in \mbox{argmin}  ~~ w_{xy} ( u^k(x) - u^k(y) ), 
\end{alignedat} \label{eq:index_beta}
\end{equation}
 and define $\beta^k_{xy}$ by
\begin{equation} \beta^k_{xy} = w_{xy} \Big(1 + d_x (p-2) \big[ \delta_{y = y^k_+} + \delta_{y = y^k_-} \big] \Big). \label{eq:beta_coeffs} \end{equation}
We also define
\begin{equation} \mathscr{L}^G_{p,k}u(x) :=\frac{1}{d_xp}\sum_y \beta^k_{xy} ( u(y) - u(x) ). \label{eq:newtonlike_def} \end{equation}
The Newton-like iteration computes $u^{k+1}$ as the solution of
\begin{equation}\label{eq:newtonlike_iter}
\left\{\begin{aligned}
-\mathscr{L}^G_{p,k} u^{k+1}(x) &= f(x)&&\text{if }x \in \mX\setminus \mO\\ 
u(x) &=g(x)&&\text{if }x \in \mO.
\end{aligned}\right.
\end{equation}
We note that if $y^{k+1}_\pm=y^k_\pm(x)$ for all $x$, then $f=-\mathscr{L}^G_{p,k} u^{k+1}=-\L^G_p u^{k+1}$. Hence, we obtain convergence to the exact solution as soon as the locations of the min and max in \eqref{eq:index_beta} are correct. Thus, the algorithm can converge to the exact solution in a finite number of iterations. It seems rather difficult to construct a convergence proof for the Newton-like iterations, and we leave this to future work. We note that $\mathscr{L}^G_{p,k}$ is a linear operator, but may not be symmetric, which is one drawback of the method.

\subsubsection{Semi-implicit Approach\label{sec:semi}}

Here, we extend the semi-implicit method of Oberman \cite{oberman2013finite} to the graph setting.  Given $\theta(x) \geq 1$, we add $- \theta(x)  \Delta^G_2 u(x) / (2 d_x)$ to both sides of \eqref{eq:plap_graph_game}, to obtain
\[ -\frac{\theta(x)}{2 d_x} \Delta^G_2 u(x) = - \bigg( \frac{\theta(x)}{2 d_x} - \frac{1}{d_x p} \bigg) \Delta^G_2 u(x) + \bigg(1-\frac{2}{p}\bigg) \Delta^G_\infty u(x) + f(x). \]
 Solving for $\Delta^G_2 u(x)$ on the left hand side reduces this equation to
\[ -\Delta^G_2 u(x) = - \bigg( \frac{\theta(x) p - 2}{\theta(x) p} \bigg) \Delta^G_2 u(x) + \frac{2 d_x}{\theta(x) p} \big(p-2\big) \Delta^G_\infty u(x) + \frac{2 d_x}{\theta(x)} f(x), \] 
 which suggests the iterative scheme
\begin{equation} -\Delta^G_2 u^{k+1}(x) = \beta(x) \Big(2 \gamma(x) \Delta^G_\infty u^k(x) - \Delta^G_2 u^k(x) \Big) + \frac{2 d_x}{\theta(x)} f(x), \label{eq:semi_iteration}\end{equation}
 where we have defined
\[ \beta(x) = \frac{\theta(x) p - 2}{\theta(x) p} \hsp\hsp\hsp \mbox{ and } \hsp\hsp\hsp \gamma(x) = d_x ~ \frac{p-2}{\theta(x) p - 2}. \]

One fundamental advantage of the semi-implicit iteration \eqref{eq:semi_iteration} is that, unlike (\ref{eq:newtonlike_iter}), the iteration (\ref{eq:semi_iteration}) requires the solution of the same symmetric positive definite system at each iteration. This means we can dramatically speed up the solver, for large scale problems, by pre-computing a Cholesky factorization and using it at each iteration, or pre-computing an Incomplete Cholesky factorization, to be used as a preconditioner for CG. Another favorable feature of this scheme is the fact that the conditioning of the system is independent of $p$, allowing the scheme to reliably solve problems for any $p$.    

The choice of $\theta(x)$ affects stability and convergence of the scheme. We give a heuristic argument here suggesting the iteration \eqref{eq:semi_iteration} is a contraction when 
\begin{equation} \label{eq:theta_bound} 
\theta(x) \geq \frac{2}{p} + d_x \Big(1 - \frac{2}{p} \Big) =: \eta(x).
\end{equation}
Define $y^k_\pm(x)$ as in \eqref{eq:index_beta} and write $y_\pm$ in place of $y^k_\pm(x)$. Noting that \eqref{eq:theta_bound} implies $-1 \leq 2\gamma(x) - 1 \leq 1$, we have
\begin{align*}
&|2 \gamma \Delta^G_\infty u^k(x) - \Delta^G_2 u^k(x)| \\
&=\left| (2 \gamma(x) - 1) \Big( w_{xy_{-}} \big(u^k(y_-) - u^k(x)\big) + w_{xy_+} \big(u^k(y_+) - u^k(x)\big) \Big) - \sum_{y \neq y_\pm} w_{xy} \big( u^k(y) - u^k(x) \big)\right|\\
&\leq|2 \gamma(x) - 1| \Big( w_{xy_{-}} |u^k(y_-) - u^k(x)| + w_{xy_+} |u^k(y_+) - u^k(x)| \Big) + \sum_{y \neq y_\pm} w_{xy} | u^k(y) - u^k(x) |\\
&\leq  \sum_{y \in \mX} w_{xy} | u^k(y) - u^k(x) |.
\end{align*}
Therefore, for  $f\equiv 0$ it follows from the definition of the iteration \eqref{eq:semi_iteration} that
\[\left|\sum_{y \in \mX} w_{xy} ( u^{k+1}(y) - u^{k+1}(x) )\right| \leq \beta(x)\sum_{y \in \mX} w_{xy} | u^k(y) - u^k(x) |,\]
where $0 < \beta(x)<1$.  While this is not a contraction, it is suggestive of what we observe in practice, namely that the semi-implicit iteration is a contraction when \eqref{eq:theta_bound} is satisfied. It seems that a convergence proof for the semi-implicit method not straightforward. Indeed, even in the case of a uniform grid in $2$-dimensions, a proof of convergence for the semi-implicit method is not available \cite{oberman2013finite}.

\section{Algorithm Comparisons}
\label{sec:experiments}

We now give a numerical study of the performance of each algorithm on synthetic problems. This allows us, in particular, to control the intrinsic dimension of the graph---the dimension of the ambient Euclidean space or manifold from which the graph is sampled---and study how the intrinsic dimension affects convergence rates.

\subsection{Experiment Design}
This section specifies design choices for our synthetic data experiments. 
\subsubsection{Problem S (synthetic data)\label{sec:synthetic}} 
We sample $\mX$ from a uniform distribution on $[0,1]^{d}$, and label $10$ of these points with labels sampled from a uniform distribution on $[0, 1]$. The experiments we report on use $d=2$, $d = 5$ and $d=10$, as the computational cost did not change substantially for $d > 10$.

\subsubsection{Graph Construction\label{sec:graph_construction}}

Define the relation $\sim_K$ on $\mX$ by $x\sim_K y$ if $x$ is among the $k$ nearest neighbors of $y$ in Euclidean distance. We construct a symmetric $K$-nearest neighbor graph as follows
\begin{equation} \label{eq:weights_definition}
w_{xy} =
\begin{cases}
\exp \left( - \frac{ |x - y|^2} { \sigma^2} \right) ,& \text{if }x \sim_K y \text{ or } y \sim_K x \\
0,&  \text{otherwise.}
\end{cases}
\end{equation}
The constant $\sigma$ plays the role of being a \emph{typical length scale} for the problem. In our experiments, we compute it as 
$ \sigma = \frac{1}{2} \max \big\{ |x - y| \, : \, w_{xy} > 0 \big\}$. We use $K = 10$ in all experiments. Other choices of $K$ would produce similar accuracy, but larger $K$ would increase the cost of each iteration. The only requirement is that $K$ needs to be large enough to ensure the graph is connected. It should be noted that, even though we prescribe $K = 10$, after we apply (\ref{eq:weights_definition}), most vertices will have a higher number of neighbors (so that symmetry can be enforced). 

\subsubsection{Error Reporting\label{sec:vari_error_conventions}}
For the variational formulation, we report
\begin{equation} \ep  = \frac{1}{n \sigma^{d + p - 1}}\max_{x \in \mX} \big| \Delta^G_p u(x) \big|, \label{eq:vari_res_tol} \end{equation}
which amounts to scaling the largest residual in equation (\ref{eq:plap_optimality}), by the number of points, and the typical length scale $\sigma$. This scaling ensures a fair comparison across different $n, p, d$. For the game theoretic formulation, the scaling of the residual in equation (\ref{eq:plap_graph_game}) is different, and we report 
\begin{equation} \ep = \frac{1}{\sigma} \max_{x \in \mX} \big| \mathscr{L}^G_p u(x) \big|. \label{eq:game_res_tol} \end{equation}

\subsection{Homotopy Results for Newton and Newton-like Methods}
\label{sec:homotopy_numerics}

Homotopy is very effective at improving the performance and reliability of Newton's method for higher $p$. We illustrate this by solving problem S with $d=10$, $p=50$, $m=10$ and $n=10^4$. We use the solution to the $p = 2$ problem as ${\bf u}^0$ for the $p = 3$ problem. Then, we apply homotopy until $p = 50$ is reached. For each $p$, we require $\ep < 10^{-12}$ according to (\ref{eq:vari_res_tol}). Table \ref{table:homotopy_newton} reports $\ep$ at each iteration. \nl 

\nd We repeat the experiment using the Newton-like algorithm for the game-theoretic formulation (now $\ep$ is given by (\ref{eq:game_res_tol})). As we can see from Table \ref{table:homotopy_newton_like}, homotopy is equally effective in this case.

\begin{table}
\begin{center} 
\begin{tabular}{|c | c | c | c | c | c | c | c | c | c | c | c |} \cline{2-12}
\multicolumn{1}{c}{}  & \multicolumn{11}{|c|} {$p$} \\ \hline
$N$ & $3$ & $4$ & $6$ & $8$ & $10$ & $15$ & $20$ & $25$ & $30$ & $40$ & $50$ \\ \hline 
1 & 2e-01 & 3e-02 & 1e-02 & 8e-04 & 2e-04 & 8e-05 & 4e-06 & 3e-07 & 3e-08 & 1e-09 & 2e-11 \\ \hline 
2 & 5e-02 & 3e-03 & 4e-04 & 1e-05 & 7e-07 & 1e-06 & 2e-08 & 6e-10 & 3e-11 & 5e-12 & 3e-14 \\ \hline 
3 & 7e-03 & 2e-04 & 4e-05 & 7e-08 & 2e-10 & 3e-10 & 3e-13 & 4e-15 & 1e-16 & 9e-17 & - \\ \hline 
4 & 8e-04 & 1e-05 & 4e-06 & 2e-09 & 1e-13 & 2e-14 & - & - & - & - & - \\ \hline 
5 & 5e-05 & 4e-07 & 8e-07 & 5e-12 & - & - & - & - & - & - & - \\ \hline 
6 & 4e-07 & 1e-09 & 7e-08 & 2e-16 & - & - & - & - & - & - & - \\ \hline 
7 & 2e-11 & 2e-14 & 6e-10 & - & - & - & - & - & - & - & - \\ \hline 
8 & 8e-15 & - & 1e-12 & - & - & - & - & - & - & - & - \\ \hline 
9 & - & - & 2e-16 & - & - & - & - & - & - & - & - \\ \hline 
\end{tabular} \end{center} 
\caption{Newton's method with homotopy solved problem S for $p=50,\ep < 10^{-12}$ in just $56$ iterations. We take increasingly large steps in $p$ while remaining in the quadratic convergence regime, suggesting we could go much farther than $p = 50$. \label{table:homotopy_newton} \vspace{-20pt}}
\end{table}

\begin{table}
\begin{center} 
\begin{tabular}{|c | c | c | c | c | c | c | c | c | c | c | c |} \cline{2-12}
\multicolumn{1}{c}{}  & \multicolumn{11}{|c|} {$p$} \\ \hline
$N$ & $3$ & $4$ & $6$ & $8$ & $10$ & $15$ & $20$ & $25$ & $30$ & $40$ & $50$ \\ \hline 
1 & 4e-02 & 1e-02 & 9e-03 & 4e-03 & 2e-03 & 3e-03 & 1e-03 & 7e-04 & 5e-04 & 6e-04 & 3e-04\\ \hline 
2 & 4e-03 & 3e-03 & 3e-03 & 2e-03 & 1e-03 & 3e-03 & 2e-03 & 7e-04 & 3e-04 & 9e-04 & 6e-04\\ \hline 
3 & 1e-03 & 8e-04 & 2e-03 & 6e-04 & 5e-04 & 2e-03 & 9e-04 & 1e-04 & 9e-05 & 1e-04 & 3e-05\\ \hline 
4 & 1e-04 & 1e-04 & 5e-04 & 5e-04 & 6e-05 & 3e-04 & 5e-04 & 3e-16 & 3e-16 & 3e-16 & 3e-16\\ \hline 
5 & 6e-06 & 1e-05 & 9e-05 & 2e-04 & 2e-07 & 6e-05 & 3e-16 & - & - & - & -\\ \hline 
6 & 2e-16 & 2e-16 & 2e-16 & 8e-06 & 4e-16 & 3e-16 & - & - & - & - & -\\ \hline 
7 & - & - & - & 3e-16 & - & - & - & - & - & - & -\\ \hline 
\end{tabular} \end{center} 
\caption{The Newton-like method with homotopy solved problem S for $p = 50$ and $\ep < 10^{-12}$ in $52$ iterations.
Similar to Newton's method, increasingly large step sizes in $p$ are possible, suggesting the ability to go even higher in $p$. \vspace{-20pt}}
\label{table:homotopy_newton_like}
\end{table}

\subsection{Computational Cost vs.~Dimension \texorpdfstring{$d$}{d} and Size \texorpdfstring{$n$}{n} \label{sec:comp_cost}}

Our timing results were performed on a laptop with 16GB of RAM memory. For the gradient descent method, we used a C-language implementation. Other methods were implemented in MATLAB 2018, as the computational cost was dominated by matrix operations. For these methods, we observed the cost of direct solvers increase by as much as $100 \times$, when going from $d = 2$ up to $d = 100$, hence we used iterative methods instead. For Newton's method, we used a conjugate gradient's method (CG), preconditioned with an Incomplete Cholesky factorization. Given that the matrix system for Newton-like's method is no longer symmetric, we used a Generalized Minimal Residual Algorithm (GMRES) \cite{saad1986gmres}, preconditioned with an Incomplete LU factorization instead. The semi-implicit method performed very well with a preconditioned CG method. Since the matrix to be inverted does not change, we compute an Incomplete Cholesky factorization, with a drop tolerance of $10^{-1}$, at the very beginning, and use it as our preconditioner in each iteration.

We explore how the Newton, Newton-like, Semi-Implicit, and Gradient Descent algorithms scale with $n, d$. To measure this relationship, we solved problem S for $p = 11$, until $\ep < 10^{-7}$, according to (\ref{eq:vari_res_tol}) and (\ref{eq:game_res_tol}) respectively. These results are reported in Figure \ref{fig:comp_cost1}. In the figure, the complexity of Newton's method scales roughly like $n^{0.8}$, the Newton-like method scales roughly like $n^{1.3}$, the semi-implicit method scales like $n^1$ for $d=2$ and $d^{0.75}$ for $d=5,10$, and the gradient descent-type method scales like $n^2$ for $d=2$, and $n^{1.5}$ for $d=5,10$.  We note that Figure \ref{fig:comp_cost1}a and Figures \ref{fig:comp_cost1}(b-d) cannot be directly compared, since the parameters $p, \ep$ do not have the same meaning in both problems, but give a rough idea of the relative performance of the variational and game-theoretic solvers.  We also illustrate our ability to solve large scale problems by solving problem S, with $m = 10$, $K = 10$, $n = 5 \cdot 10^5$ and $d = 10$. For the variational formulation, we use Newton's method starting from the solution of the $p = 2$ problem ($1$-step homotopy). For the game theoretic formulation, we choose the semi-implicit method. In both cases, iterative solvers were used. We plot the results, averaged over $5$ trials in Figure \ref{fig:large_scale_new}. In the figure, the CPU time for Newton's method scales like $\eps^{-0.05}$ with tolerance $\eps>0$, while the CPU time for the semi-implicit method scales like $\eps^{-0.2}$.  

\begin{figure}
\centering
\hspace{-1cm}
\subfloat[Newton's cost for $2^6 \leq n \leq 2^{15}$, $p = 11$.\label{fig:method_vs_ndA} ]
{\includegraphics[width=0.51\textwidth,clip = true, trim = 10 30 50 70]{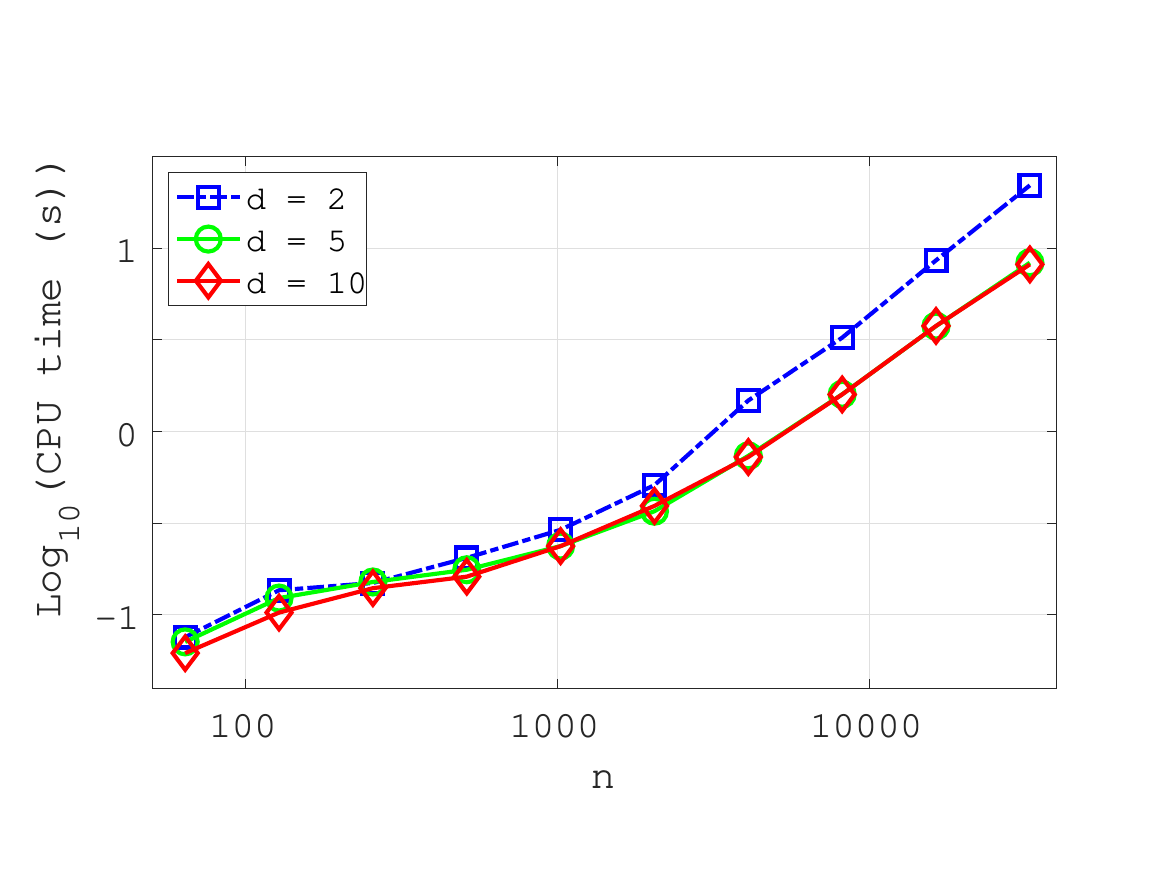}}
\subfloat[Newton-like's cost for $2^6 \leq n \leq 2^{15}$, $p = 11$. \label{fig:method_vs_ndB} ]
{\includegraphics[width=0.51\textwidth,clip = true, trim = 10 30 50 70]{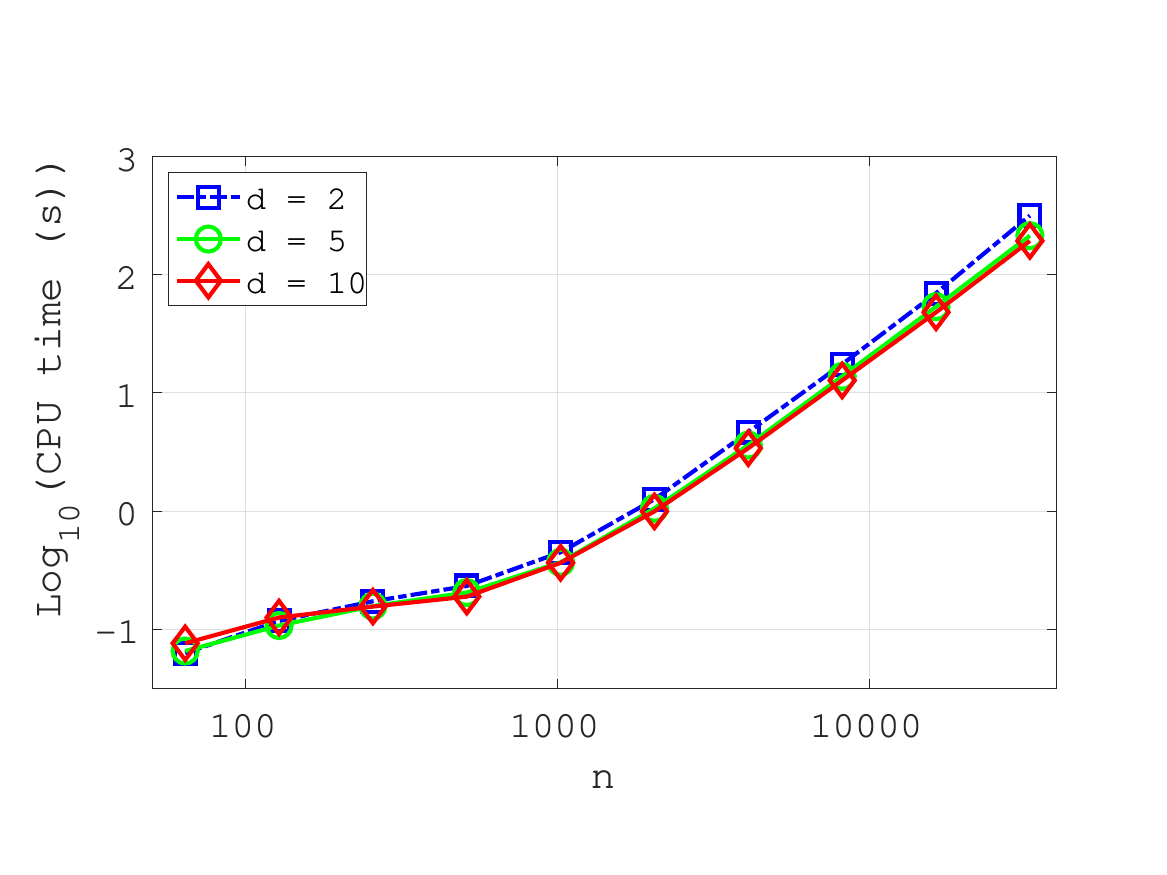}}\\
\vspace{-0.3cm}\hspace{-1cm}
\subfloat[Semi-Implicit's cost for $2^6 \leq n \leq 2^{15}$, $p = 11$. \label{fig:method_vs_ndC} ]
{\includegraphics[width=0.51\textwidth,clip = true, trim = 10 30 50 70]{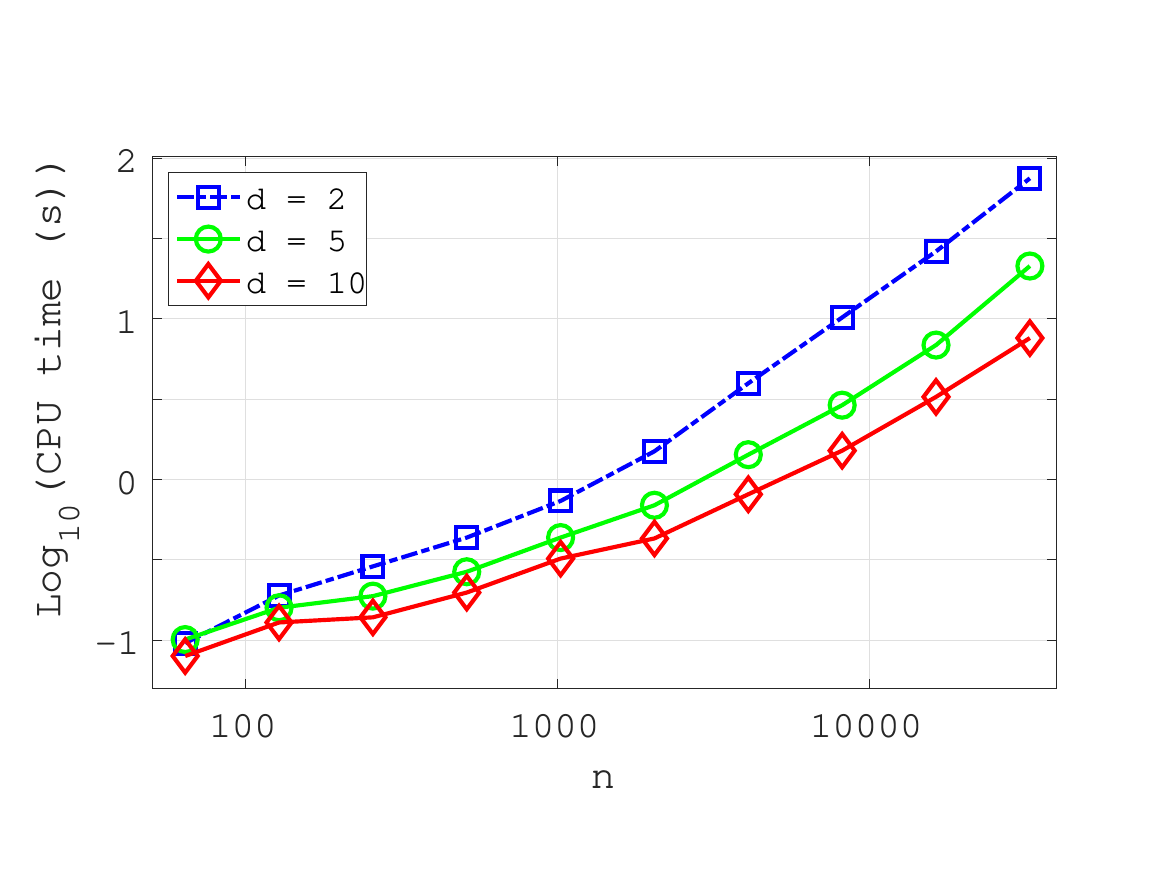}}
\subfloat[Gradient descent's cost, $2^6 \leq n \leq 2^{15}, p = 11$. \label{fig:method_vs_ndD} ]
{\includegraphics[width=0.51\textwidth,clip = true, trim = 10 30 50 70]{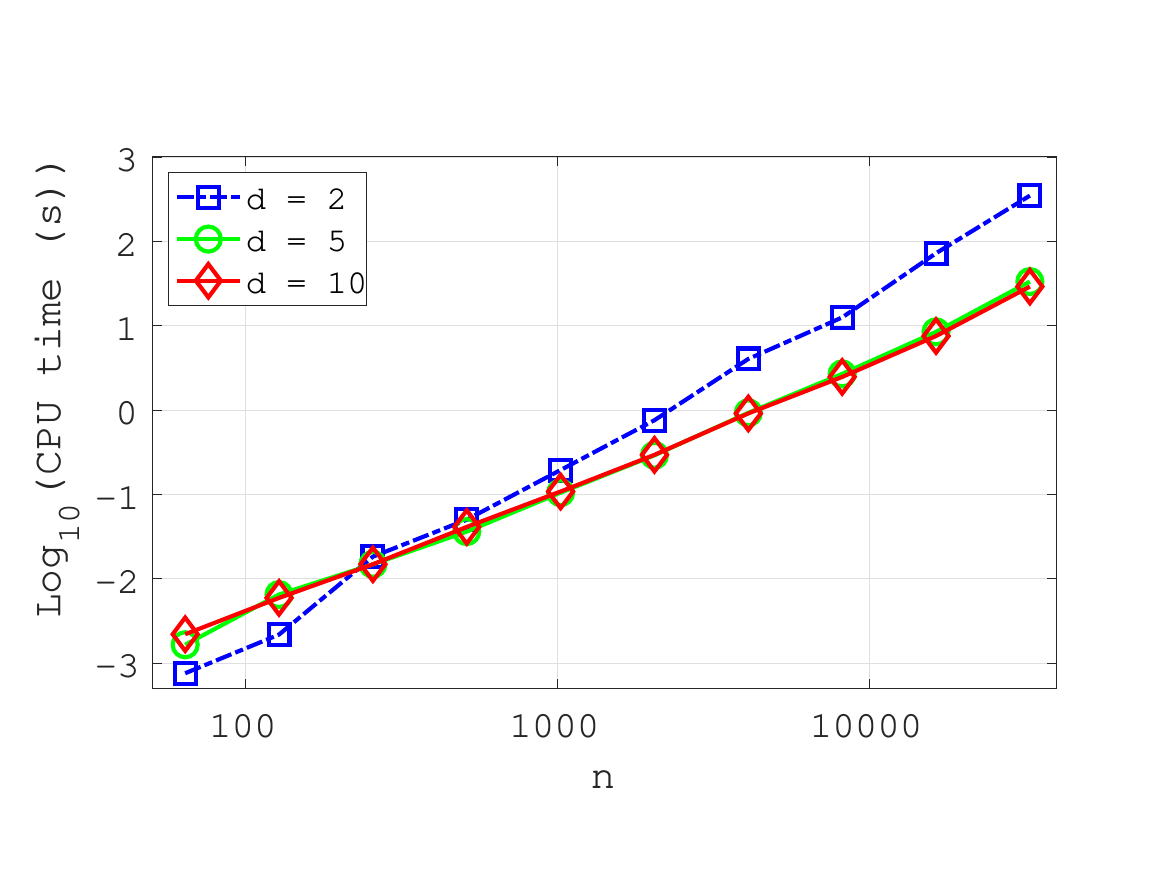}}\\
\caption{The computational cost of solving problem $S$, averaged over $5$ trials. (a) Newton's cost in solving problem $S$ until $\ep < 10^{-7}$ with preconditioned CG. (b) Newton-like's cost in solving problem $S$, until $\ep < 10^{-7}$, with preconditioned GMRES. (c) Semi-implicit method's cost for solving problem $S$ until $\ep < 10^{-7}$, with preconditioned CG. (d) Gradient descent's cost for solving problem $S$ until $\ep < 10^{-7}$.}
\label{fig:comp_cost1} 
\end{figure}

\begin{remark}
It is interesting to point out that the results in Figure \ref{fig:comp_cost1} show that the cost of all four methods decreases (sometimes significantly) as the dimension $d$ increases. This is due to the graph Laplacian matrices having better condition numbers in higher dimensions. To see why this is the case, we recall that on a uniform grid, the condition number for a discrete Laplacian is on the order of $\Delta x^{-2}$, where $\Delta x$ is the grid resolution. When $\Delta x$ is small, the Laplacian is poorly conditioned and iterative methods are stiff without some form of preconditioning. The same scaling for the condition number holds for graph Laplacians constructed from random geometric graphs, due to the spectral convergence results in, for example, \cite{calder2019improved}, but now $\Delta x$ should be replaced by the length scale $\eps>0$ on which the graph is constructed. In dimension $d$, the length scale is at least $\eps \geq n^{-1/d}$, to ensure graph connectivity. This implies that as a function of $n$ and $d$, the condition number of the graph Laplacian should scale like $n^{-2/d}$. Thus, the curse of dimensionality \emph{improves} the condition number of graph Laplacians in higher dimensions, and explains why iterative methods perform better for larger $d$. The stiffness of Laplace equations is essentially a phenomenon of high resolution meshes, which are only possible to construct in low dimensional spaces (e.g., $d=2,3$). By the manifold assumption in machine learning \cite{ssl}, we expect our graphs to have intrinsic dimensionality significantly higher than $d=3$, and so the graph Laplacians we encounter in practice are better conditioned than what we may expect from experience with solving PDEs in $d=2$ or $d=3$ dimensions. 
\label{rem:condition}
\end{remark}

\begin{figure}
\centering
\hspace{-1cm}
\subfloat[Newton]{\includegraphics[width=0.51\textwidth,clip = true, trim = 10 65 70 100]{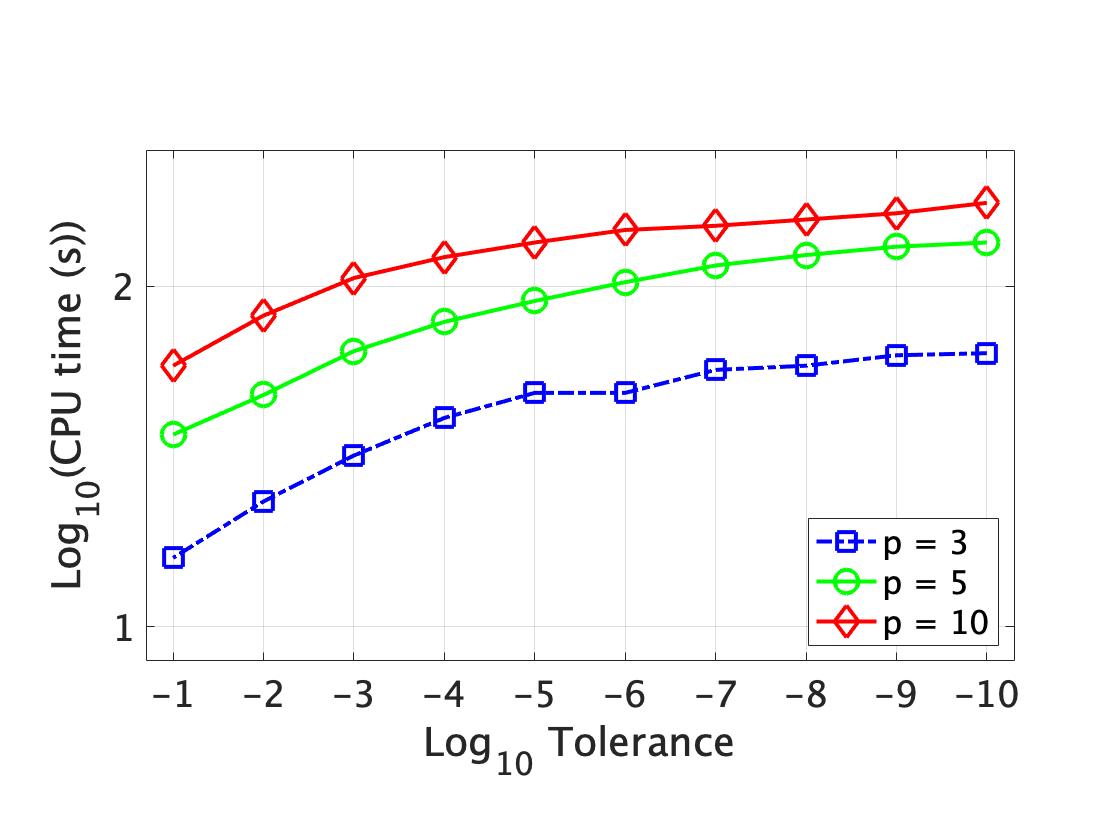}}
\subfloat[Semi-implicit]{\includegraphics[width=0.51\textwidth,clip = true, trim = 10 65 70 100]{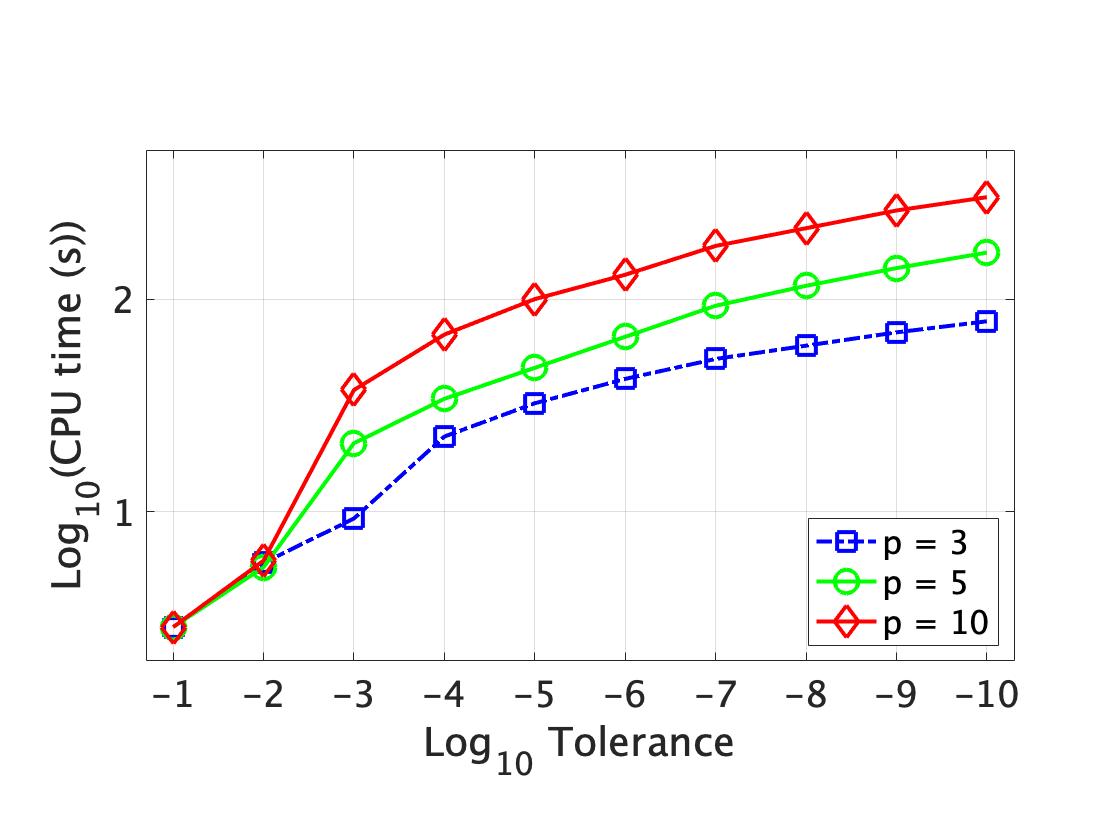}}\\
\caption{Solving problem S at large scale with $n = 5 \cdot 10^5$, $p = 5$ and $d = 10$. \\
(a) Newton's cost for solving the variational formulation, using preconditioned CG.
(b) Semi-implicit's cost for the game-theoretic problem, with preconditioned CG.}
\label{fig:large_scale_new} 
\end{figure}

\subsection{Synthetic Experiment Conclusions}

The experiments shown demonstrate that both variational and game-theoretic formulations can be solved with the algorithms we propose for large scale graphs with high intrinsic dimensionality, similar to what will be observed in practice. For the game-theoretic problem, the semi-implicit formulation is likely to perform the best, although it lacks the convergence guarantees enjoyed by the gradient descent method. The Newton and Newton-like algorithms show strong resemblance to one another, and both methods improve their performance and reliability substantially when homotopy is utilized. 

\section{Experiments with Real Data}
\label{sec:real}

We now give an experimental study of $p$-Laplacian semi-supervised learning classification on real datasets, including MNIST, and the more complex datasets Fashion MNIST and Extended MNIST. 

\subsection{Description of the Experiments and Datasets}

On each dataset, we shall perform two experiments. For the first, we explore the performance as the number of labels varies, with the goal of understanding how different models perform with increasing, but very small, amounts of labeled data. For the second experiment, we fix the number of labels to $1$ per class, and explore the performance as a function of the amount of unlabeled data. This experiment helps to illustrate the degeneracy of the $2$-Laplacian, while also illustrating how other models can profit from an increasing amount of unlabeled data, which is the premise upon which semi-supervised learning is built. Before proceeding to the results, we provide some detail on the datasets we work with.

In all three problems, we first preprocess each image via a Scattering Transform \cite{bruna2013invariant} in order to extract features upon which to build the graph. The graphs are constructed as $K$-nearest neighbor graphs in the feature space using the weights \eqref{eq:weights_definition}. In all experiments we chose $K = 10$, to ensure graph connectivity and sufficient sparsity. We also experimented with $K = 25$ and $K = 50$ and found that increasing the number of neighbors in the graph does not provide any advantages, and results in significantly higher computational cost. 

We solve the $M$-class classification problem with the standard one-vs-rest approach, whereby we solve $M$ binary classification problems classifying each digit against the rest. This gives $M$ scores for each unlabeled image, and the final classification is chosen as the class with the highest score.

\subsubsection{MNIST}

The MNIST dataset is a standard benchmark that consists of $70,000$ images of handwritten digits $0$ through $9$ \cite{lecun1998gradient}. Each image is a $28 \times 28$ pixel grayscale image, meaning it can be represented as a vector in $\R^{784}$ dimensions. The classes are well balanced, with roughly $7,000$ examples for each digit.

\subsubsection{Fashion MNIST}

This dataset, recently introduced by \cite{xiao2017fashion}, was designed as a drop-in replacement for MNIST, in order to test classification accuracy for a significantly harder problem. This dataset also consists of $70,000$ grayscale images of size $28\times 28$ pixels, except the $10$ handwritten digits are replaced by $10$ classes of clothing items, such as sandals, and dresses. For an example, see Figure \ref{fig:mnist_fashion_sample}.

\begin{figure} \begin{center} 
\includegraphics[width = 0.6\textwidth]{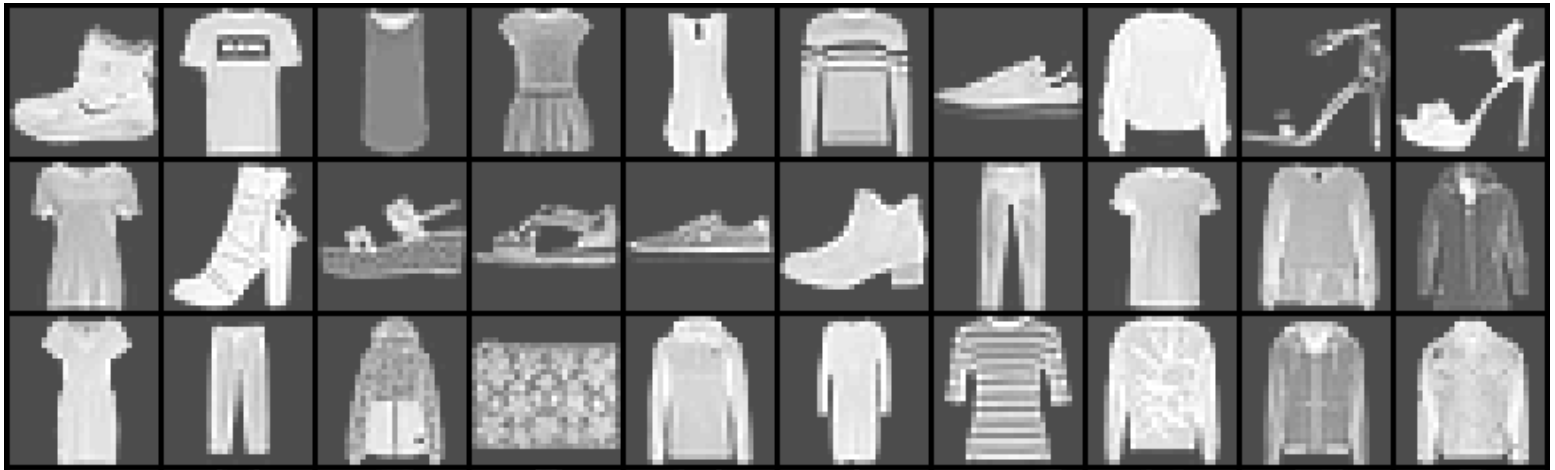}
\caption{Sample $28 \times 28$-pixel images from the Fashion MNIST dataset. \label{fig:mnist_fashion_sample}}
\end{center} \end{figure}

\subsubsection{Extended MNIST}

The Extended MNIST (EMNIST) dataset \cite{cohen2017emnist} contains images of handwritten letters, of $28 \times 28$ pixels. Given that letters may be uppercase or lowercase, and there are now $26$ classes, the semisupervised task on this dataset is substantially more challenging. To simplify the task slightly, we downsample the dataset for the first experiment (performance vs.~number of labels) so that all classes have the same number of labels. We end up with $3419$ samples for each of the $26$ classes, for a total of $n = 88,894$ images. For the second task, we measure performance with subsampled datasets of $2^5, 2^6, \dots, 2^{11}$ samples per class, always having $1$ labeled sample. 

\subsection{Experimental Results}

We report the results of our experiments in Figure \ref{fig:mnist_results}. In all experiments, we report results using $5$ formulations: the standard $2$-Laplacian, the Weighted Non-local Laplacian (WNLL) \cite{shi2017weighted}, and the game-theoretic $p$-Laplacian for $p = 5$, $p = 9$ and $p = \infty$ (e.g. Lipschitz learning). All models were solved to a tolerance of $\ep < 10^{-2}$, which provides consistent results. The results we report paint a similar picture across all three datasets. For the first type of experiment (left column), all formulations improve their accuracy as the number of labels increases, and they gradually approach a similar level as $m$ grows. This means that the choice of a good model is of particular importance in the regime when labeled data is extremely limited. Both the $p$-Laplace and WNLL methods clearly outperform the $2$-Laplacian at low label rates, while the WNLL and $p$-Laplacian give fairly similar results.

The second experiment (right column) provides further evidence of the superiority of the $p$-Laplacian model over the $2$-Laplacian. In this case we fix $1$ label per class, which means $m = 10$ for MNIST and Fashion MNIST, while $m = 26$ for EMNIST. The premise of semi-supervised is that we may achieve superior performance by including both labeled and unlabeled data. If we keep $m$ fixed, we should expect the accuracy to improve as the number of unlabeled images increases. In all three datasets, the $2$-Laplacian classification performance decreases substantially as the amount of unlabeled data $n$ grows. This illustrates the need for alternative models when $n \gg m$. On the other hand, the $p$-Laplacian models (for sufficiently large $p$) do not degenerate as $n$ grows, and in fact, their performance increases slightly as $n$ grows. The Weighted Non-local Laplacian exhibits mixed results. For the MNIST and Fashion MNIST datasets, performance decreases slightly when $n = 70,000$. Meanwhile, this model outperforms others on the EMNIST dataset.

Let us make some final remarks comparing and contrasting WNLL and $p$-Laplace learning. The experimental results in Figure \ref{fig:mnist_results} show that neither method is strictly better than the other. For the most part, the methods are comparable. The largest difference is seen at 1 label per class on MNIST where $p$-Laplace is roughly 10\% better than WNLL. The difference is much more pronounced in terms of the theoretical guarantees for each method. It was shown in \cite{calder2018game} (see \cite{slepcev2019analysis} for the variational $p$-Laplacian) that $p$-Laplace learning is well-posed with arbitrarily few labeled examples; in fact, we can take the number of labeled examples to be finite while sending the number of unlabeled examples to infinity, and still obtain a well-posed continuous extension of the label values. On the other hand, the WNLL was shown in \cite{calder2019properly} to be \emph{ill-posed} in the same setting of finite labeled data and infinite unlabeled data. In particular, \cite[Corollary 3.8]{calder2019properly} shows that the WNLL with finite labeled data converges to a constant labeling function as the amount of unlabeled data tends to infinity. To the best of our knowledge, the only semi-supervised learning algorithm with theoretical guarantees at arbitrarily low label rates, such as the ones given in \cite{calder2018game,slepcev2019analysis}, is the graph $p$-Laplacian.

\begin{figure}
\centering
\hspace{-1cm}
\subfloat[MNIST, $n = 70000, 10 \leq m \leq 50$.\label{fig:mnist_resultsA} ]
{\includegraphics[width=0.51\textwidth,clip = true, trim = 10 50 70 100]{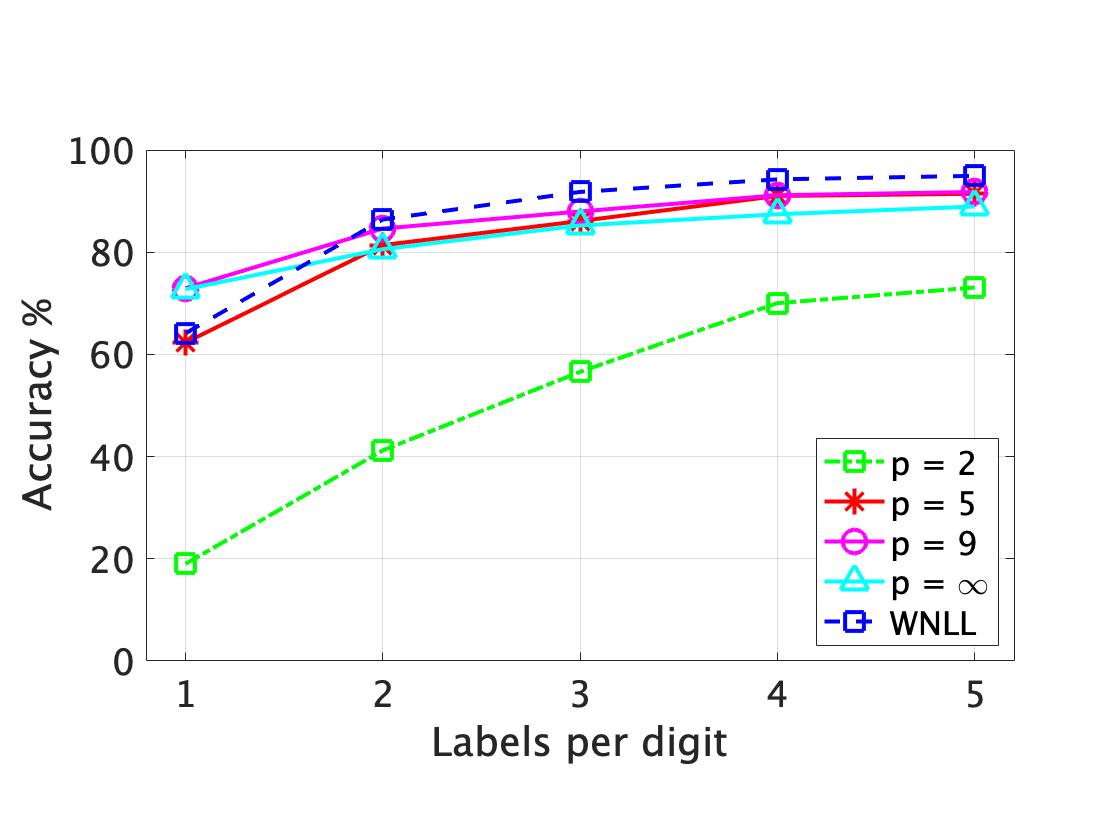}}
\subfloat[MNIST, $m = 10, 2188 \leq n \leq 70000$. \label{fig:mnist_resultsB} ]
{\includegraphics[width=0.51\textwidth,clip = true, trim = 10 50 70 100]{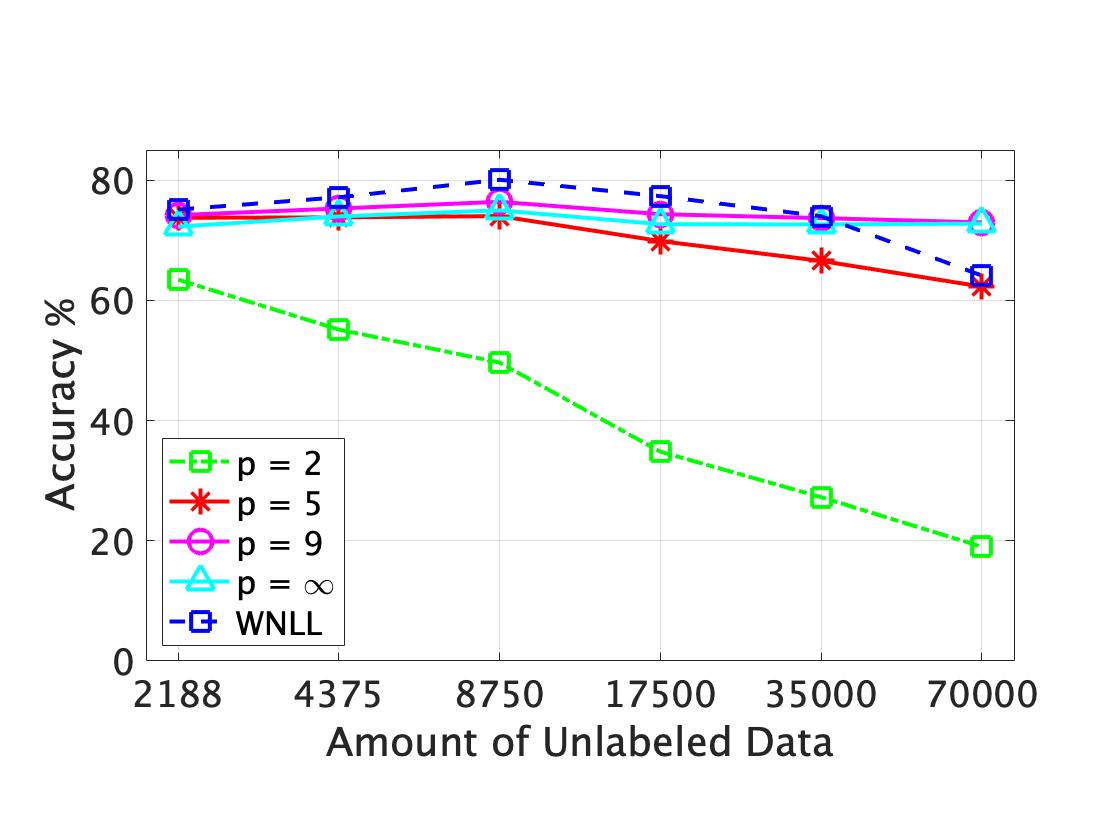}}\\
\vspace{-0.3cm}\hspace{-1cm}
\subfloat[Fashion MNIST, $n = 70000, 10 \leq m \leq 50$. \label{fig:mnist_resultsC} ]
{\includegraphics[width=0.51\textwidth,clip = true, trim = 10 50 70 100]{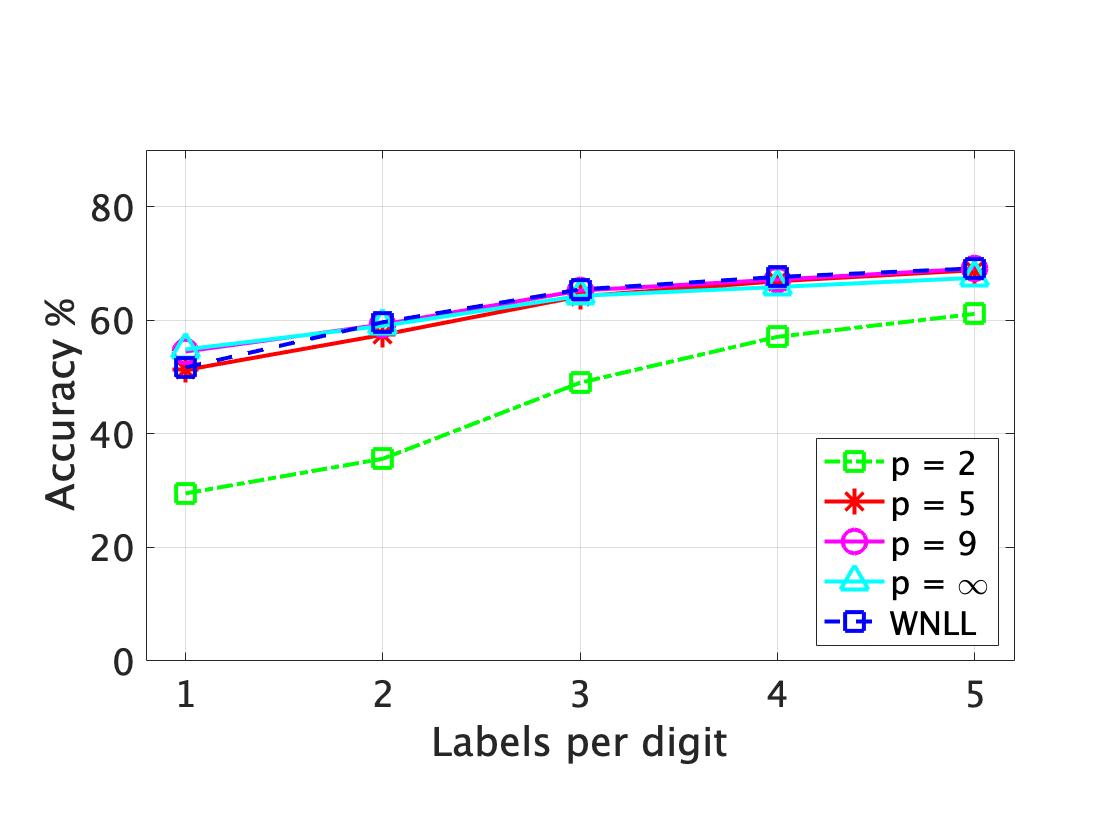}}
\subfloat[Fashion MNIST, $m = 10, 2188 \leq n \leq 70000$. \label{fig:mnist_resultsD} ]
{\includegraphics[width=0.51\textwidth,clip = true, trim = 10 50 70 100]{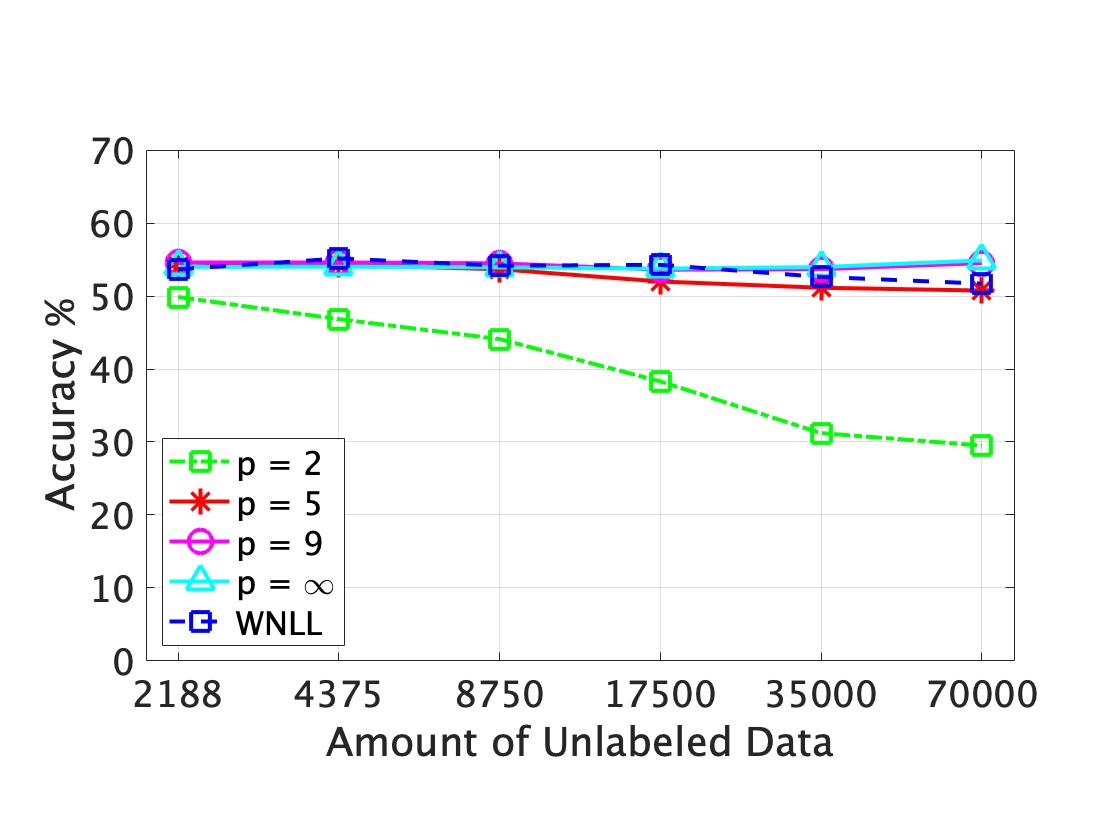}}\\
\vspace{-0.3cm}\hspace{-1cm}
\subfloat[EMNIST, $n = 88894, 10 \leq m \leq 50$.\label{fig:emnist_resultsA} ]
{\includegraphics[width=0.51\textwidth,clip = true, trim = 10 50 70 100]{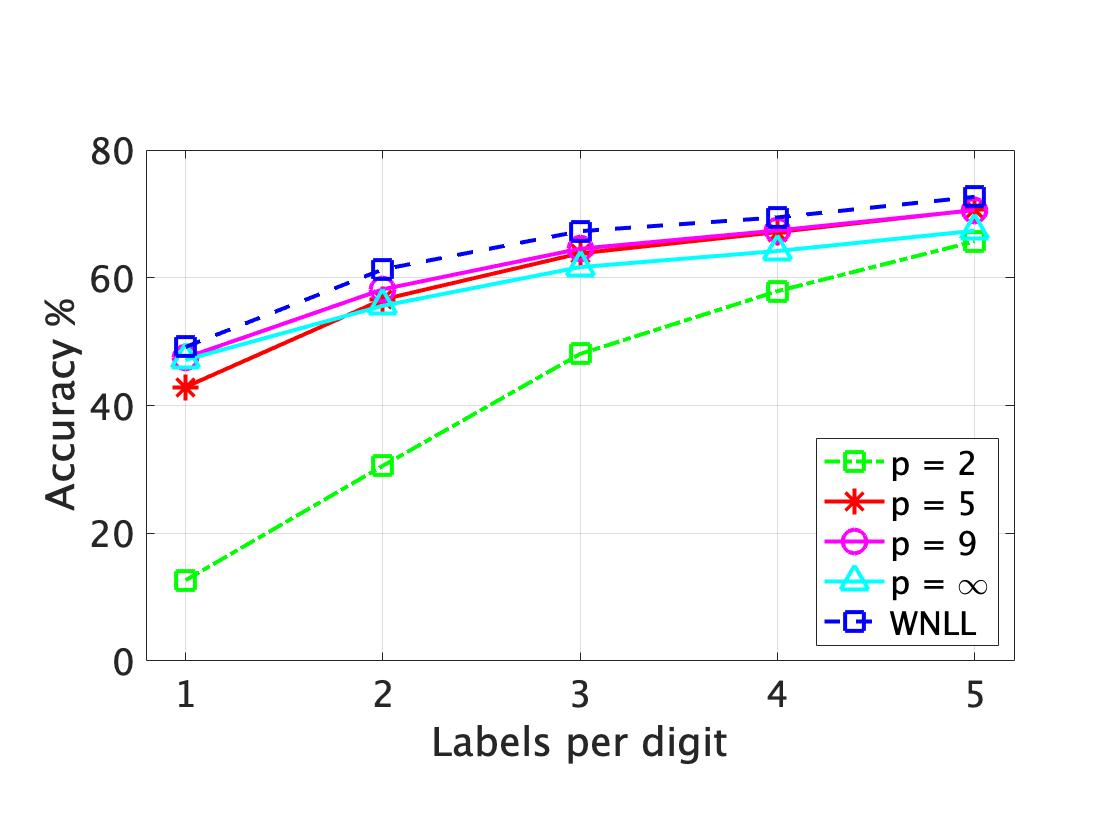}}
\subfloat[EMNIST, $m = 26, 832 \leq n \leq 53248$. \label{fig:emnist_resultsB} ]
{\includegraphics[width=0.51\textwidth,clip = true, trim = 10 50 70 100]{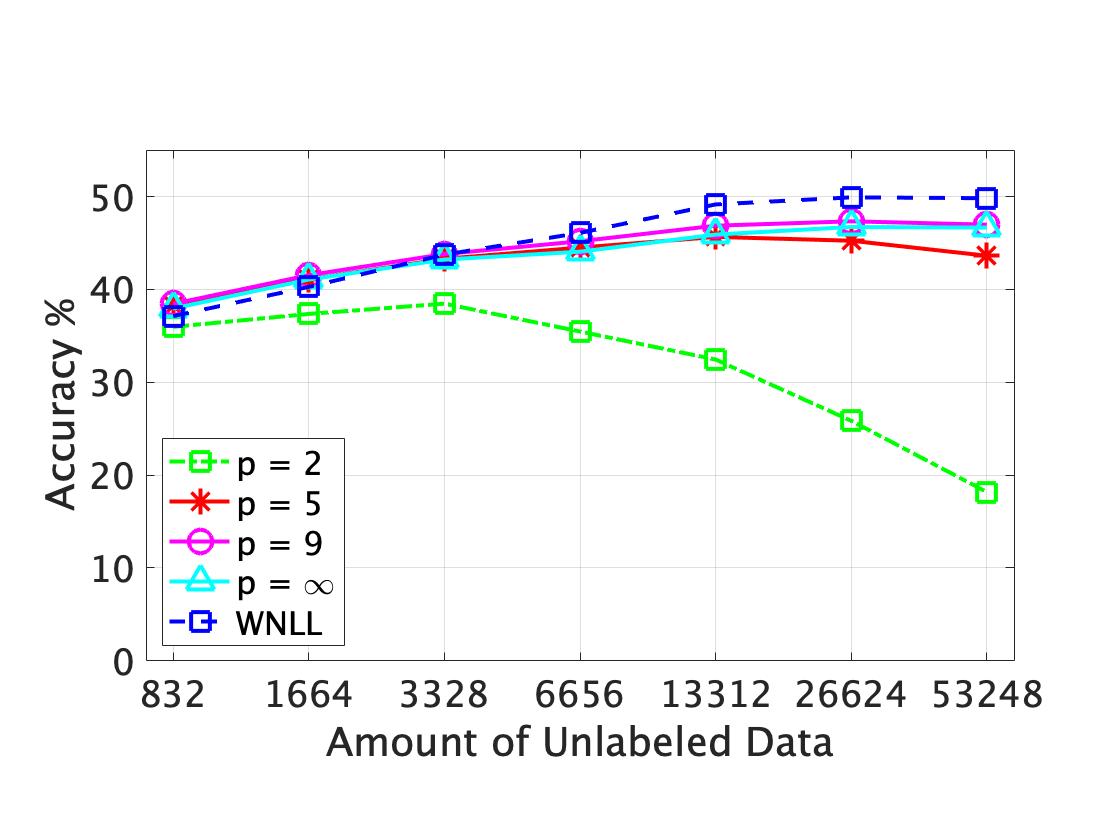}}\\
\caption{Classification accuracy for the MNIST, Fashion MNIST and Extended MNIST datasets. Figures (a,c,e) show the performance as a function of the number of labels per class, while (b,d,f) show the performance for 1 label per class as a function of the amount of unlabeled data used.}
\label{fig:mnist_results} 
\end{figure}

\section{Conclusions\label{sec:conclusions}}

This paper is focused on theory and applications of $\ell_p$-Laplacian regularized semi-supervised learning. We preformed a detailed analysis of discrete to continuum theory for the $p$-Laplacian on $k$-NN graphs, which are more commonly used in practice compared to random geometric graphs, and made the surprising discovery that the $p$-Laplacian models retain information about the data distribution as $p\to \infty$ on $k$-NN graphs, contrary to conventional wisdom from the existing $p$-Laplace theory on random geometric graphs. We also presented a simple and very general framework for proving discrete to continuum convergence results that only requires pointwise consistency and a monotonicity property. We expect this framework to be useful in future work.

We also studied and developed algorithms for solving the variational and game-theoretic formulations of the $p$-Laplacian on a weighted graph. The variational formulation may be solved efficiently using Newton's method with homotopy. The semi-implicit method is the fastest method for solving the game-theoretic formulation, while the gradient-descent approach enjoys rigorous convergence guarantees. Our experiments with real data show that $p$-Laplacian learning is superior to Laplace learning ($p=2$) at very low label rates on common image classification datasets including MNIST, FashionMNIST and EMNIST.

\appendix

\section{Review of viscosity solutions}
\label{sec:viscosity}

Viscosity solutions are a notion of weak solution for partial differential equations that obeys the maximum principle and enjoys strong stability and uniqueness theorems. The theory is especially useful for passing from discrete to continuum limits (see, e.g., \cite{calderViscosity,calder2018limit,calder2014hamilton}). We review here the basic definitions. Let $\usc(\bar{\Omega})$ (resp.~$\lsc(\bar{\Omega})$) denote the collection of functions that are upper (resp.~lower) semicontinuous at all points in $\bar{\Omega}$. We make the following definitions.
\begin{definition}\label{def:visc_bc}
We say $u \in \usc(\bar{\Omega})$ is a \emph{viscosity subsolution} of \eqref{eq:bvp} if for all $x \in \bar{\Omega}$ and every $\phi \in C^\infty(\R^n)$ such that $u-\phi$ has a local maximum at $x$ with respect to $\bar{\Omega}$
\[\begin{cases}
F(\nabla ^2u(x),\nabla \phi(x),u(x),x) \leq 0,&\text{if } x \in \Omega\\
\min\left\{ F(\nabla ^2\phi(x),\nabla \phi(x),u(x),x), u(x)-g(x)\right\} \leq 0&\text{if }x \in \partial \Omega.
\end{cases}\]

Likewise, we say that $u \in \lsc(\bar{\Omega})$ is a \emph{viscosity supersolution} of \eqref{eq:bvp} if for all $x \in \bar{\Omega}$ and every $\phi \in C^\infty(\R^n)$ such that $u-\phi$ has a local minimum at $x$ with respect to $\bar{\Omega}$
\[\begin{cases}
F(\nabla ^2u(x),\nabla \phi(x),u(x),x) \geq 0,&\text{if } x \in \Omega\\
\max\left\{ F(\nabla ^2\phi(x),\nabla \phi(x),u(x),x), u(x)-g(x)\right\} \geq 0&\text{if }x \in \partial \Omega.
\end{cases}\]

Finally, we say that $u$ is a \emph{viscosity solution} of \eqref{eq:bvp} if $u$ is both a viscosity sub- and supersolution. In this case, we say that the boundary conditions in \eqref{eq:bvp} hold in the \emph{viscosity sense}, which is also known as the generalized Dirichlet sense.
\end{definition}

\begin{definition}
We say that \eqref{eq:bvp} enjoys \emph{strong uniqueness} if whenever $u\in \usc(\bar{\Omega})$ is a subsolution of \eqref{eq:bvp} and $v\in \lsc(\bar{\Omega})$ is a supersolution, we have $u\leq v$ on $\bar{\Omega}$.
\label{def:SU}
\end{definition}

We refer the reader to \cite{crandall1992user} for the proof of the comparison principle for viscosity solutions with generalized Dirichlet boundary conditions, which implies strong uniqueness. 

\bibliographystyle{abbrv}
\bibliography{ref}

\end{document}